\documentclass[11pt, reqno, dvipsnames]{amsart}

\usepackage[english]{babel}
\usepackage[utf8]{inputenc}

\usepackage[pdftex]{graphicx}
\usepackage{amsmath}
\usepackage{amssymb}
\usepackage{amsfonts}
\usepackage{amsthm}
\usepackage[all,cmtip]{xy}
\usepackage{tikz}
\usepackage{tikz-cd}
\usepackage[new]{old-arrows}
\usepackage{blindtext}
\usepackage{accents}
\usepackage{scalerel}
\usepackage{stmaryrd}
\usepackage{mathrsfs}
\usepackage{dsfont}
\usepackage{relsize}
\usepackage{url}
\usepackage{contour}
\usepackage{ulem}
\usepackage{footmisc}
\usepackage{enumerate}
\usepackage[T1]{fontenc}
\usepackage{fancyhdr}

\usetikzlibrary{decorations.markings}
\tikzset{double line with arrow/.style args={#1,#2}{decorate,decoration={markings, mark=at position 0 with {\coordinate (ta-base-1) at (0,1pt);
\coordinate (ta-base-2) at (0,-1pt);},
mark=at position 1 with {\draw[#1] (ta-base-1) -- (0,1pt);
\draw[#2] (ta-base-2) -- (0,-1pt);
}}}}

\usepackage{mathtools}

\usetikzlibrary{trees}

\tikzstyle{level 1}=[level distance=2.4cm, sibling distance=6.5cm]
\tikzstyle{level 2}=[level distance=2.4cm, sibling distance=2.5cm]
\tikzstyle{level 3}=[level distance=3cm, sibling distance=0.8cm]

\usepackage[linktocpage=true]{hyperref}
\hypersetup{
    colorlinks = true,
    linkcolor = {ForestGreen},
    citecolor = {RoyalBlue},
    urlcolor = {Black}
}

\newcommand{\listref}[1]{{\color{ForestGreen}(}\ref{#1}{\color{ForestGreen})}}

\newcommand{\Rr}{\mathbb{R}}

\newcommand{\Zz}{\mathbb{Z}}
\newcommand{\Ff}{\mathbb{F}}
\newcommand{\Nn}{\mathbb{N}}
\newcommand{\Qq}{\mathbb{Q}}
\newcommand{\Hh}{\mathbb{H}}
\newcommand{\Pp}{\mathbb{P}}

\newcommand{\Aa}{\mathbb{A}}
\newcommand{\Gg}{\mathbb{G}}
\newcommand{\Bb}{\mathbb{B}}
\newcommand{\Tt}{\mathbb{T}}
\newcommand{\Ll}{\mathbb{L}}
\newcommand{\Mm}{\mathbb{M}}
\newcommand{\Dd}{\mathbb{D}}

\newcommand{\solid}{\otimes^{\scalebox{0.5}{$\blacksquare$}}}
\newcommand{\dsolid}{\otimes^{\LL \scalebox{0.5}{$\blacksquare$}}}

\newcommand{\solidif}{\scalebox{0.5}{$\blacksquare$}}
\newcommand{\dsolidif}{\LL \scalebox{0.5}{$\blacksquare$}}

\newcommand{\petc}{\text{\myuline{$\pet$}}}
\newcommand{\dRc}{\underline{\dR}}

\newcommand{\cl}[1]{\mathcal{#1}}

\contourlength{0.8pt}

\newcommand{\myuline}[1]{%
  \uline{\phantom{#1}}%
  \llap{\contour{white}{#1}}%
 }

\DeclareMathOperator*{\colim}{colim}
\DeclareMathOperator{\ett}{\acute{e}t}
\DeclareMathOperator{\pet}{pro\acute{e}t}
\DeclareMathOperator{\qpet}{qpro\acute{e}t}
\DeclareMathOperator{\an}{an}

\DeclareMathOperator{\rk}{rk}
\DeclareMathOperator{\GL}{GL}

\DeclareMathOperator{\Gal}{Gal}
\DeclareMathOperator{\Hom}{Hom}

\DeclareMathOperator{\coker}{coker}

\DeclareMathOperator{\Spa}{Spa}

\DeclareMathOperator{\dR}{dR}

\DeclareMathOperator{\Fil}{Fil}
\DeclareMathOperator{\gr}{gr}
\DeclareMathOperator{\Sp}{Sp}

\DeclareMathOperator{\RigSm}{RigSm}
\DeclareMathOperator{\Mod}{Mod}
\DeclareMathOperator{\cont}{cont}
\DeclareMathOperator{\intt}{int}
\DeclareMathOperator{\nonint}{n-int}
\DeclareMathOperator{\Kos}{Kos}
\DeclareMathOperator{\Vect}{Mod}

\DeclareMathOperator{\im}{im}

\DeclareMathOperator{\LC}{LC}
\DeclareMathOperator{\Ext}{Ext}

\DeclareMathOperator{\cond}{cond}

\DeclareMathOperator{\Set}{Set}
\DeclareMathOperator{\aff}{aff}
\DeclareMathOperator{\Shv}{Shv}
\DeclareMathOperator{\Ab}{Ab}
\DeclareMathOperator{\Solid}{Solid}
\DeclareMathOperator{\ssolid}{solid}
\DeclareMathOperator{\LL}{L}
\DeclareMathOperator{\Berk}{Berk}

\DeclareMathOperator{\DR}{dR}
\DeclareMathOperator{\PN}{PN}
\DeclareMathOperator{\FF}{FF}
\DeclareMathOperator{\nuc}{nuc}

\DeclareMathOperator{\Tot}{Tot}
\DeclareMathOperator{\Perf}{Perf}
\DeclareMathOperator{\opp}{op}
\DeclareMathOperator{\CondAb}{CondAb}
\DeclareMathOperator{\CondSet}{CondSet}
\DeclareMathOperator{\ExtrDisc}{ExtrDisc}
\DeclareMathOperator{\ttop}{top}
\DeclareMathOperator{\topp}{top}
\DeclareMathOperator{\LCA}{LCA}
\DeclareMathOperator{\trace}{tr}
\DeclareMathOperator{\Ban}{Ban}
\DeclareMathOperator{\disc}{disc}
\DeclareMathOperator{\Bun}{Bun}

\DeclareMathOperator{\cris}{cris}

\DeclareMathOperator{\Isoc}{Isoc}

\numberwithin{equation}{section}
\newtheorem{theorem}{Theorem}
\numberwithin{theorem}{section}

\newtheorem{lemma}[theorem]{Lemma}
\newtheorem{cor}[theorem]{Corollary}
\newtheorem{prop}[theorem]{Proposition}

\theoremstyle{definition}
\newtheorem{df}[theorem]{Definition}
\newtheorem{convention}[theorem]{Convention}
\newtheorem{convnot}[theorem]{Notation and conventions}
\newtheorem{notation}[theorem]{Notation}

\newtheorem*{Acknowledgments}{Acknowledgments}

\theoremstyle{remark}
\newtheorem{rem}[theorem]{Remark}
\newtheorem{example}[theorem]{Example}
\newtheorem{examples}[theorem]{Examples}

 \AtBeginDocument{%
 \def\MR#1{}
}

\date{\today}

\makeatletter
\renewcommand{\address}[1]{\gdef\@address{#1}}
\renewcommand{\email}[1]{\gdef\@email{\url{#1}}}
\newcommand{\@endstuff}{\par\vspace{\baselineskip}\noindent\small
\begin{tabular}{@{}l}\scshape\@address\\\textit{E-mail address:} \@email\end{tabular}}
\AtEndDocument{\@endstuff}
\makeatother

\title[On the $p$-adic pro-\'etale cohomology of Drinfeld symmetric spaces]{On the $p$-adic pro-\'etale cohomology \\ of Drinfeld symmetric spaces}
\author{Guido Bosco}
\address{IMJ-PRG, Sorbonne Université, 4 place Jussieu, 75005 Paris, France}
\email{gbosco@imj-prg.fr}

\begin{document}
 
\begin{abstract}
 Via the relative fundamental exact sequence of $p$-adic Hodge theory, we determine the geometric $p$-adic pro-étale cohomology of the Drinfeld symmetric spaces defined over a $p$-adic field, thus giving an alternative proof of a theorem of Colmez--Dospinescu--Nizioł. Along the way, we describe, in terms of differential forms, the geometric pro-étale cohomology of the positive de Rham period sheaf on any connected, paracompact, smooth rigid-analytic variety over a $p$-adic field, and we do it with coefficients. A key new ingredient is the condensed mathematics recently developed by Clausen--Scholze.
\end{abstract}

\maketitle

\setcounter{tocdepth}{2}

\tableofcontents

\section{\textbf{Introduction}}
 \sectionmark{}

 Let $p$ be a fixed prime number. Let $K$ be a finite extension of $\Qq_p$, and let $\cl O_K$ denote its ring of integers.
 Let $\overline K$ be a fixed algebraic closure of $K$. Let us denote by $C:=\widehat{\overline K}$ the completion of $\overline K$, $\cl O_C$ its ring of integers, and $\mathscr{G}_K:=\Gal(\overline K/K)$ the absolute Galois group of $K$.
 
 \subsection{History and motivation}
 In this article, using perfectoid methods, we study the geometric $p$-adic pro-étale cohomology of the Drinfeld upper half-spaces defined over $K$ (a non-archimedean variant of the Poincaré upper half-spaces),\footnote{See Definition \ref{defdrinfeld}.} whose interest lies in the understanding of the $p$-adic local Langlands program.
 
 The $p$-adic local Langlands correspondence for $\GL_2(\Qq_p)$, envisioned by Breuil, and established in full generality by Colmez, \cite{Colmezlanglands}, is given by a covariant exact functor $\Pi\mapsto \mathbf{V}(\Pi)$ from the category of admissible unitary $\Qq_p$-Banach representations of $\GL_2(\Qq_p)$,  which are residually of finite length and admitting a central character, towards the category of finite-dimensional continuous $\Qq_p$-representations of $\mathscr{G}_{\Qq_p}$. As shown by Colmez--Dospinescu--Paškūnas, \cite[Theorem 1.1]{CDP}, the functor $\Pi\mapsto \mathbf{V}(\Pi)$ induces a bijection between the isomorphism classes of absolutely irreducible non-ordinary $\Qq_p$-Banach representations of $\GL_2(\Qq_p)$ in the source, and the isomorphism classes of $2$-dimensional absolutely irreducible continuous $\Qq_p$-representations of $\mathscr{G}_{\Qq_p}$. Moreover, by \cite[Theorem 1.3]{CDP}, such functor encodes the classical $\ell$-adic local Langlands correspondence for $\GL_2(\Qq_p)$, for a prime number $\ell\neq p$. The latter correspondence is known more generally for $\GL_n(K)$, for any integer $n\ge 1$, thanks to the work of Harris--Taylor \cite{HT}, Henniart \cite{Henniart}, and, more recently, Scholze \cite{Scholzelanglands}.  The proof of Harris--Taylor, combined with \cite{Faltingsmoduli}, \cite{Fargues}, in particular shows that, for the supercuspidal representations, such correspondence can be realized in the geometric $\ell$-adic pro-étale cohomology of the Drinfeld tower, which is a tower of rigid-analytic étale Galois coverings of the Drinfeld upper half-space.
  
 On the other hand, very little is known about a $p$-adic local Langlands correspondence for $\GL_n(K)$, beyond the case of $\GL_2(\Qq_p)$.\footnote{Although there are by now at least two candidates for a $p$-adic local Langlands correspondence for $\GL_n(K)$: Caraiani--Emerton--Gee--Geraghty--Paskunas--Shin constructed a functor from representations of $\mathscr{G}_K$ to representations of $\GL_n(K)$, \cite{CEGGPS}, and Scholze produced a functor going in the opposite direction, \cite{Scholzelubintate}.} 
 The story of the $\ell$-adic counterpart teaches us that one way to understand better, and possibly generalize, the correspondence for $\GL_2(\Qq_p)$ might be to find a geometric incarnation of it, but one difficulty is that the $p$-adic (pro-)étale cohomology of $p$-adic rigid-analytic varieties is much more intricate than the $\ell$-adic one. 
 
 Despite these difficulties, recently Colmez--Dospinescu--Nizioł, \cite{CDN0}, were able to show that the geometric $p$-adic étale cohomology of the Drinfeld tower over $\Qq_p$ in dimension 1 realizes the $p$-adic local Langlands correspondence for $\GL_2(\Qq_p)$, for the 2-dimensional de Rham representations of $\mathscr{G}_{\Qq_p}$ of Hodge-Tate weight 0 and 1, whose associated Weil--Deligne representation is irreducible. Moreover, their computation suggests that the geometric $p$-adic étale cohomology of the Drinfeld tower over $K$ in dimension 1 should encode a still hypothetical $p$-adic Langlands correspondence for $\GL_2(K)$, for a general $K$. 
 
 For the Drinfeld tower in higher dimension, all that we can say so far is for the level 0 of the tower, i.e. for the Drinfeld upper half-space. In fact, Colmez--Dospinescu--Nizioł, \cite{CDN1}, proved the following surprising result, of which we propose an alternative proof that is amenable to several generalizations.
 
 \subsection{Main result} 
 
  Given an integer $d\ge 1$, let $\Hh^d_K$ denote the Drinfeld upper half-space of dimension $d$ defined over $K$, and let $\Hh_C^d$ be its base change to $C$. Let $G:=\GL_{d+1}(K)$. 

 \begin{theorem}[{\cite[Theorem 4.12]{CDN1}, Theorem \ref{azz}}]\label{first}For all $i\ge 0$, there is a strictly exact sequence of $G\times \mathscr{G}_K$-Fr\'echet spaces over $\Qq_p$
  $$0\to \Omega^{i-1}(\Hh_C^d)/\ker d\to H^i_{\pet}(\Hh_C^d, \Qq_p(i))\to \Sp_i(\Qq_p)^*\to 0$$
  where $\Sp_i(\Qq_p)^*$ denotes the weak topological dual of the locally constant special representations of $G$ with coefficients in $\Qq_p$.\footnote{See Definition \ref{speciali}.}
 \end{theorem}

  \begin{rem}
   In \cite{CDN1} this result  is deduced from a general comparison theorem, between the geometric $p$-adic pro-étale cohomology and the Hyodo--Kato and de Rham cohomologies, for rigid-analytic Stein spaces over $K$ having a semistable weak formal model over $\cl O_K$ (e.g. $\Hh_K^d$); in turn, such comparison is obtained via the geometric syntomic cohomology of Fontaine--Messing. We recall that Colmez--Dospinescu--Nizioł also computed the rational and integral $p$-adic étale cohomology of $\Hh_C^d$ in \cite{CDN1} and \cite{CDNint}, respectively.
   
   Moreover, recently, Orlik gave an alternative proof of Theorem \ref{first}, \cite{Orlik},\footnote{However, in \textit{loc. cit.} the cohomology groups $H^i_{\pet}(\Hh_C^d, \Qq_p(i))$ are only determined as $G\times \mathscr{G}_K$-modules, and not as topological $G\times \mathscr{G}_K$-vector spaces over $\Qq_p$.} which is based on his strategy for describing global sections of certain $G$-equivariant vector bundles on $\Hh_K^d$, \cite{Orlikequi}.
  \end{rem}

  \begin{rem}
   The relation between Theorem \ref{first} and a hypothetical $p$-adic local Langlands correspondence is still not so transparent, and it will not be explored in this paper. 
  \end{rem}

  Before summarizing the strategy we use to prove Theorem \ref{first}, we first recall its $\ell$-adic version due to Schneider--Stuhler, for a prime number $\ell\neq p$, and explain why the proof of the latter fails in the $p$-adic setting.
  
  \begin{theorem}[Schneider--Stuhler, \cite{SS}] Let $\ell\neq p$ be a prime number.
  For all $i\ge 0$, there is an isomorphism of $G\times \mathscr{G}_K$-modules
   $$H_{\pet}^i(\Hh_C^d, \Qq_\ell(i))\cong \Sp_i(\Qq_\ell)^*.$$
  \end{theorem}

  Schneider--Stuhler computed the cohomology groups of $\Hh_K^d$ for a general abelian sheaf cohomology theory, defined on the category of smooth rigid-analytic varieties over $K$, satisfying a number of axioms, the most restricting one being the homotopy invariance with respect to the $1$-dimensional open unit disk $\accentset{\circ}{\Dd}_K$. These axioms are satisfied by the geometric pro-étale cohomology with coefficients in $\Qq_\ell$, but the homotopy invariance with respect to $\accentset{\circ}{\Dd}_K$ fails with coefficients in $\Qq_p$: in fact, $H^1_{\pet}(\accentset{\circ}{\Dd}_C, \Qq_p)$ is an infinite-dimensional $C$-vector space (see \cite[Theorem 3]{CNaff} and \cite[\S 3]{LeBras1}). Therefore, a different strategy is needed.
  
 \subsection{Overview of the strategy}
  To prove Theorem \ref{first} we consider the sheaf-theoretic version of the fundamental exact sequence of $p$-adic Hodge theory on the pro-étale site $\Hh^d_{C, \pet}$ 
 \begin{equation}\label{fund}
   0\to\mathbb{Q}_p\to \Bb_e\to \Bb_{\dR}/\Bb_{\dR}^+\to 0.
 \end{equation}
 Here, $\Bb_{\dR}^+$ is the positive de Rham period sheaf, $\Bb_{\dR}$ denotes the de Rham period sheaf, and $\Bb_e$ is defined as the Frobenius invariants $\Bb[1/t]^{\varphi=1}$ where $\Bb$ is the pro-étale sheaf-theoretic version of the ring $B$ of the analytic functions on the ``punctured open unit disk'' $\Spa(A_{\inf}, A_{\inf})\setminus V(p[p^\flat])$, introduced by Fargues--Fontaine in their work on \textit{the curve}, and $t$ is the  Fontaine’s $2\pi i$ (see \S\ref{petsheaves}).
 
 Therefore, in order to determine the $p$-adic pro-étale cohomology of $\Hh_C^d$, we reduce to study the pro-étale cohomology of $\Hh_C^d$ with coefficients in $\Bb_{\dR}^+$, $\Bb_{\dR}$ and $\Bb_e$. 
 
 Using Scholze's Poincar\'{e} lemma for $\Bb_{\dR}^+$ (Proposition \ref{poincare}), we determine, in terms of differential forms, the geometric pro-étale cohomology of the period sheaf $\Bb_{\dR}^+$ on any connected, paracompact, smooth rigid-analytic variety over $K$ (such as $\Hh_K^d$), as explained in more details in \S \ref{relev}.
 
 On the other hand, for the geometric pro-étale cohomology of the period sheaves $\Bb_{\dR}$ and $\Bb_e$, we show that they satisfy a slight variant of the above-mentioned axioms of Schneider--Stuhler, including the homotopy invariance with respect to the $1$-dimensional open unit disk $\accentset{\circ}{\Dd}_K$ (see Proposition \ref{BdRSS}, and Proposition \ref{BeSS}, respectively). The proof of the latter axiom for the geometric pro-étale cohomology of $\Bb_e$ is essentially due to Le Bras (see Remark \ref{lbrem}), and it is largely inspired by the strategy used by Bhatt--Morrow--Scholze to relate their $A_{\inf}$-cohomology theory with the $q$-de Rham
 cohomology, \cite{BMS1}.
 
 This will be enough to show Theorem \ref{first}. We refer the reader to \cite{Bosco2} for a more in-depth study of the geometric pro-étale cohomology of $\Bb_e$. 
 
 \begin{rem}\label{lbrem}
  The fundamental exact sequence (\ref{fund}) has already been used by Le Bras to compute the geometric $p$-adic pro-étale cohomology of the rigid-analytic affine space, and the open polydisk of any dimension, \cite[\S 3]{LeBras1}. 
 \end{rem}

 \subsection{Relevance of the condensed and solid formalisms}\label{relev}
 In the search for a geometric incarnation of the hypothetical $p$-adic Langlands correspondence for $G=\GL_{d+1}(K)$, in Theorem \ref{first} it is crucial to describe the cohomology groups $H^i_{\pet}(\Hh_C^d, \Qq_p(i))$ as \textit{topological} $G\times \mathscr{G}_K$-vector spaces over $\Qq_p$, and not merely as $G\times \mathscr{G}_K$-modules. 
 We note also that, from a purely geometric point of view, for many cohomology theories appearing in $p$-adic Hodge theory (such as the $p$-adic (pro-)étale, de Rham, Hyodo--Kato, etc.), the cohomology groups of non-proper rigid-analytic varieties (e.g. $\Hh_K^d$) are usually huge; therefore, it is important to exploit the topological structure that they may carry in order to study them. But, in doing so, one quickly runs into topological issues, mainly due to the fact that the category of topological abelian groups is \textit{not} abelian.
 
 In our case, one encounters an example of such a topological issue already at the start of the strategy we have outlined. In fact, we consider the long exact cohomology sequence of $\Qq_p$-vector spaces associated to the fundamental exact sequence (\ref{first}) on $\Hh^d_{C, \pet}$
 \begin{equation}\label{less}
  \cdots\to H^i_{\pet}(\Hh_C^d, \Qq_p)\to H^i_{\pet}(\Hh_C^d, \Bb_e)\to H^i_{\pet}(\Hh_C^d, \Bb_{\dR}/\Bb_{\dR}^+)\to \cdots
 \end{equation}
 and we want to endow these cohomology groups with a natural topology in such a way that all the maps in (\ref{less}) are continuous. One may try to work in the category of locally convex $\Qq_p$-vector spaces, and put a topology à la \v{C}ech on such cohomology groups, using the fact that the global sections of a pro-étale period sheaf on an affinoid perfectoid space over $\Spa(C, \cl O_C)$ carry a natural topology (cf. \S \ref{petsheaves}), but \textit{a priori} it is not clear whether, in this way, the boundary maps of the long exact sequence (\ref{less}) are continuous.\footnote{A classical solution to this particular issue is to work instead in an \textit{abelian envelope} of the category of locally convex $\Qq_p$-vector spaces. Cf. \cite[\S 2.1.1]{CDN1}.}
 
 As we explain below, since the pro-étale cohomology groups appearing in (\ref{less}) have a natural structure of condensed $\Qq_p$-vector spaces, by their very definition, we found it convenient and fruitful to work in the condensed mathematics framework, recently introduced by Clausen--Scholze, \cite{Scholzecond}, which is precisely designed to overcome the kind of topological issues we have described.
 
 We denote by $\CondAb$ the category of condensed abelian groups,\footnote{See \S \ref{conve} for the set-theoretic conventions we adopt.} which we identify with the category of pro-étale sheaves of abelian groups on the geometric point $\Spa(C, \cl O_C)=*$ (see Remark \ref{geopoint}). Recall that $\CondAb$ is a nice abelian category, containing most topological abelian groups of interest, \cite[Proposition 1.7, Theorem 2.2]{Scholzecond}. Then, we give the following definition.
 
 \begin{df}\label{Hcintro}(Definition \ref{Hc})
  Let $f: X\to \Spa(C, \cl O_C)$ be an analytic adic space, and let $\cl F$ be a sheaf of abelian groups on $X_{\pet}$.  We define the complex of $D(\CondAb)$
   $$R\Gamma_{\petc}(X, \cl F):=Rf_{\pet *} \cl F$$
   with $i$-th cohomology, for $i\ge 0$, the \textit{condensed pro-étale cohomology group}
   $H^i_{\petc}(X, \cl F)=R^if_{\pet *}\cl F.$
  \end{df}
  
 \begin{rem}
   Note that the underlying abelian group $H^i_{\petc}(X, \cl F)(*)$ is the usual pro-étale cohomology group $H^i_{\pet}(X, \cl F)$.
 \end{rem}
 
 Coming back to (\ref{less}), if we denote by $f:\Hh_C^d\to \Spa(C, \cl O_C)$ the structure morphism, now we can simply apply the derived functor $Rf_{\pet *}$ to (\ref{fund}), and consider the associated long exact sequence in cohomology, which will be an exact sequence of condensed $\Qq_p$-vector spaces. \medskip
 
 Let us now explain in more detail how we compare the geometric pro-étale cohomology with coefficients in $\Bb_{\dR}$ and $\Bb_{\dR}^+$, with the de Rham cohomology, and the role that the solid formalism plays in this. In the following, we denote by $T\mapsto \underline{T}$ the functor from topological groups/rings/$K$-vector spaces/etc. to condensed groups/rings/$K$-vector spaces/etc.\footnote{Again, see \S \ref{conve} for the set-theoretic conventions.} Moreover, we denote by $\Vect_K^{\cond}$ the category of condensed $K$-vector spaces, and by $\Vect_K^{\ssolid}$ the symmetric monoidal subcategory of solid $K$-vector spaces, endowed with the solid tensor product $\solid_K$ (see Appendix \ref{condfun}). \medskip
 
 We will prove the following theorem, which extends results of Scholze \cite[Theorem 7.11]{Scholze}, and Le Bras \cite[Proposition 3.17]{LeBras1}. We refer the reader to Theorem \ref{BDR} for a version with coefficients of the statement below.

 \begin{theorem}[cf. Theorem \ref{BDR}]\label{equi}
    Let $X$ be a smooth rigid-analytic variety\footnote{All rigid-analytic varieties will be assumed to be quasi-separated.} defined over $K$.
    
    We define\footnote{See Definition \ref{FsolidA} for the relevant notation.} the \textit{de Rham cohomology of the base change of $X$ to $B_{\dR}$} as the complex of $D(\Vect_K^{\cond})$
   $$R\Gamma_{\dRc}(X_{B_{\dR}}):=R\Gamma(X, \underline{\Omega_X^\bullet}\solid_K B_{\dR}).$$
   \begin{enumerate}[(i)]
   \item\label{equi:1}
   We have a $\underline{\mathscr{G}_K}$-equivariant, compatible with filtrations, natural isomorphism in $D(\Vect_K^{\ssolid})$
   $$ R\Gamma_{\petc}(X_C, \Bb_{\dR})\simeq R\Gamma_{\dRc}(X_{B_{\dR}}).$$
   \item\label{equi:2}  Assume that $X$ is connected and paracompact.
   Then, for each $r\in \Zz$, we have a $\underline{\mathscr{G}_K}$-equivariant isomorphism in $D(\Vect_K^{\ssolid})$
   $$R\Gamma_{\petc}(X_C, \Fil^r\Bb_{\dR})\simeq\Fil^r(R\Gamma_{\dRc}(X)\dsolid_K B_{\dR}).$$ 
   Here, $R\Gamma_{\dRc}(X)$ denotes the de Rham cohomology complex in $D(\Vect_K^{\cond})$ (Definition \ref{condDR}). 
   \end{enumerate}
  \end{theorem}

 \begin{rem}
  The \textit{paracompact} assumption (Definition \ref{parac}) in Theorem \ref{equi} is not very restrictive, in fact most of the rigid-analytic varieties over $K$ that arise in nature are paracompact. Examples include any separated rigid-analytic variety over $K$ of dimension 1, the rigid analytification of a separated scheme of finite type over $K$, a Stein space over $K$ (e.g. $\Hh_K^d$), an admissible open of a quasi-compact and quasi-separated rigid-analytic variety over $K$ (see Examples \ref{list}).
 \end{rem}
 
 Let us briefly describe how we prove Theorem \ref{equi}: we show part (\ref{equi:1}) using Scholze's Poincaré lemma (Proposition \ref{poincare}); then, part (\ref{equi:2}) follows from part (\ref{equi:1}) and the following base change result, recalling that $\Fil^jB_{\dR}=t^jB_{\dR}^+$ for $j\in \Zz$ and $B_{\dR}^+$ is a $K$-Fréchet algebra.

 \begin{theorem}[Theorem \ref{Wbasechange}]\label{basechangeintro}
  Let $X$ be a connected, paracompact, rigid-analytic variety defined over $K$. Let $\cl F^\bullet$ be a bounded below complex of sheaves of topological $K$-vector spaces whose terms are coherent $\cl O_X$-modules. Let $A$ be a $K$-Fréchet algebra, regarded as a condensed $K$-algebra. Then, we have a natural isomorphism in $D(\Vect_K^{\ssolid})$
  \begin{equation}\label{basec}
    R\Gamma(X, \underline{\cl F}^\bullet)\dsolid_K A\overset{\sim}{\to} R\Gamma(X, \underline{\cl F}^\bullet \solid_K A).
  \end{equation}
 \end{theorem}
 
 The proof of Theorem \ref{basechangeintro} is based on results of Clausen--Scholze on the category of \textit{nuclear $K$-vector spaces} (see \S \ref{nucnuc}), which allow us to reduce the statement to the case when $X$ is affinoid.
 
 \begin{rem}
  Under the hypotheses of Theorem \ref{basechangeintro},  by the flatness of $K$-Fréchet spaces with respect to the solid tensor product $\solid_K$ (a result due to Clausen--Scholze), taking cohomology in (\ref{basec}) we have the following isomorphism in $\Vect_K^{\ssolid}$ (Corollary \ref{fundcor})
  $$H^i(X, \underline{\cl F}^\bullet)\solid_K A\cong H^i(X, \underline{\cl F}^\bullet \solid_K A)$$
   for all $i\in \Zz$. We note that such a statement is characteristic of the condensed mathematics realm: in fact, its naive analogue in the category of locally convex $K$-vector spaces, with the completed projective tensor product replacing the solid tensor product, is trivially false, due to the fact that the cohomology group $H^i(X, {\cl F}^\bullet)$ can be a non-Hausdorff locally convex $K$-vector space (see Remark \ref{vs}). This is just an instance of the fact that the solid formalism works very well also for ``non-Hausdorff objects'' (more precisely, for non-quasi-separated condensed sets).
 \end{rem}

 As a consequence of Theorem \ref{equi}, in some special cases we can give a particularly nice description of the geometric pro-étale cohomology with coefficients in $\Bb_{\dR}/\Bb_{\dR}^+$, appearing in the fundamental exact sequence (\ref{fund}), including in cases where the de Rham cohomology groups are otherwise pathological as topological vector spaces. The following result explains in particular how, for a smooth rigid-analytic variety $X$ over $K$, the (non-)degeneration of the Hodge-de Rham spectral sequence is reflected in its geometric $p$-adic pro-étale cohomology.
 
  \begin{cor}[Corollary \ref{B/B+}]\label{nice}
 Let $X$ be a smooth rigid-analytic variety over $K$. Let $i\ge 0$.
  \begin{enumerate}[(i)]
  \item \label{nice:1} If $X$ is proper,  we have a $\underline{\mathscr{G}_K}$-equivariant isomorphism in $\Vect_K^{\ssolid}$ 
   $$H^i_{\petc}(X_C, \Bb_{\dR}/\Bb_{\dR}^+)\cong(H^i_{\dRc}(X)\otimes_K B_{\dR})/\Fil^0.$$
  \item \label{nice:2} If $X$ is an affinoid space,  we have the following $\underline{\mathscr{G}_K}$-equivariant exact sequence in $\Vect_K^{\ssolid}$
    $$0\to H^i_{\dRc}(X)\solid_K B_{\dR}/t^{-i}B_{\dR}^+ \to H^i_{\petc}(X_C, \Bb_{\dR}/\Bb_{\dR}^+)\to \underline{\Omega^i(X)/\ker d}\solid_K C(-i-1)\to 0.$$
  \end{enumerate} 
 \end{cor}
 
 \begin{rem} Let us shortly comment on the statement and the proof of Corollary \ref{nice}.
\begin{enumerate}[(i)]
 \item If $X$ is proper, the proof uses crucially the degeneration of the Hodge-de Rham spectral sequence, \cite[Corollary 1.8]{Scholze}. In this case, the de Rham cohomology groups are finite-dimensional $K$-vector spaces, therefore it not a surprise that $H^i_{\dRc}(X)=\underline{H_{\dR}^i(X)}$ (Lemma \ref{propclass}),  and Corollary \ref{nice}\listref{nice:1} can be stated in classical topological terms.
 \item If $X$ is an affinoid space, the proof relies on  Tate's acyclicity theorem (Lemma \ref{condensedtate}\listref{condensedtate:1}), using which one can show that there is a $\underline{\mathscr{G}_K}$-equivariant exact sequence in $\Vect_K^{\ssolid}$
  \begin{equation}\label{introB_dR^+}
    0\to H^i_{\dRc}(X)\solid_K t^{-i+1}B_{\dR}^+ \to H^i_{\petc}(X_C, \Bb_{\dR}^+)\to \underline{\Omega^i(X)}^{d=0}\solid_K C(-i)\to 0.
  \end{equation}
 Note that $H^i_{\dRc}(X)=\underline{\Omega^i(X)}^{d=0}/\underline{d\Omega^{i-1}(X)}$, but, in general, $H^i_{\dRc}(X)\neq \underline{\Omega^i(X)^{d=0}/d\Omega^{i-1}(X)}$ (since the latter displayed quotient can be non-Hausdorff, see Remark \ref{vsCN}). 
 \item If $X$ is a Stein space, using Kiehl's acyclicity theorem (Lemma \ref{A&B}), one can prove that the exact sequence (\ref{introB_dR^+}) also holds for such $X$. In this case, we have  $H^i_{\dRc}(X)=\underline{\Omega^i(X)^{d=0}/d\Omega^{i-1}(X)}$ (Lemma \ref{drstein}), and, \textit{a posteriori}, the exact sequence (\ref{introB_dR^+}) can be restated in the category of topological $K$-vector spaces with the completed projective tensor product replacing the solid tensor product (see Remark \ref{affstein}). 
 On the contrary, in the affinoid case, it seems hard to get such a clean statement as the one of Corollary \ref{nice}\listref{nice:2} in the usual topological setting. This kind of issue is classically avoided replacing affinoid spaces with dagger affinoid spaces (cf. \cite{CDN1}, \cite{CN}); however, the point we want to make here is that, in our setting, there is no need for this.
\end{enumerate}
 \end{rem}

 Let us also mention that Heuer has recently observed in \cite[Remark 5.9]{Heuer} that, for $X$ a smooth Stein space over $K$, the exact sequence (\ref{introB_dR^+}) might be used to describe the ``exotic line bundles''  appearing when passing from the analytic site to the pro-étale site of $X_C$. \medskip
 
 As a final note, we want to add that, since at the time this article was written the work of Clausen--Scholze was fairly recent, we had to put a particular accent on comparing topological spaces to condensed sets. However, it should be noted that the main results and proofs of this paper could be entirely formulated in the condensed language.

 \subsection{Leitfaden of the paper}
 
 We have organized the paper as follows. In \S \ref{condcohgroups}, we discuss several ways of endowing an abelian sheaf cohomology group with a structure of condensed abelian group, and we explain the relations between them. In \S \ref{SSmodif}, we redefine in the condensed setting, and slightly modify (in order to include the geometric pro-étale cohomology of $\Bb_{\dR}$ and $\Bb_e$), the cohomology theories for which Schneider--Stuhler computed the cohomology groups of the Drinfeld upper half-spaces in \cite{SS}; then, we outline the computation of \textit{loc. cit.}  highlighting the main points where, due to our modifications, an additional argument is needed.
 In \S \ref{petsheaves}, after having recalled the definitions and the basic results we use on the pro-\'etale period sheaves, we study, using the solid formalism, a Cartan--Leray spectral sequence for the cohomology of such pro-étale sheaves, and we give a proof of the relative fundamental exact sequences of $p$-adic Hodge theory, via the relative Fargues--Fontaine curves. In \S \ref{coherent}, we prove Theorem \ref{basechangeintro}; the latter will be then used in \S \ref{sectionBdr} to show Theorem \ref{equi}. In \S \ref{sectionBdr}, we also prove that the geometric pro-étale cohomology of $\Bb_{\dR}$ satisfies the axioms of Schneider--Stuhler (\ref{axioms}). In \S \ref{bee}, we show that the geometric pro-étale cohomology of $\Bb_e$ satisfies the axioms of Schneider--Stuhler (\ref{axioms}).
 Finally, in \S \ref{appli}, we gather the results of the previous sections and prove Theorem \ref{first}.
 
 We end with two appendices. Appendix \ref{condfun} is devoted to the condensed functional analysis over non-archimedean local fields, developed by Clausen--Scholze. In appendix \ref{ccg}, we recall the definition of condensed group cohomology, and, using the solid formalism,  we relate it to Koszul complexes in some cases of particular interest to us.

 \subsection{Notation and conventions}\label{conve}\footnote{This notation, and these conventions, are adopted throughout the paper, except for Appendix \ref{condfun} and Appendix \ref{ccg}. See \S \ref{ape}.}
  Fix a prime number $p$. Unless otherwise specified, we will denote by $(K, |\cdot|)$ a complete discretely valued non-archimedean extension of $\Qq_p$, with perfect residue field $k$, ring of integers $\cl O_K$, and fixed uniformizer $\varpi\in \cl O_K$.
  
  We fix an algebraic closure $\overline K$ of $K$. We will denote $C:=\widehat{\overline K}$ the completion of $\overline K$, $\cl O_C$ its ring of integers, and $\mathscr{G}_K:=\Gal(\overline K/K)$ the absolute Galois group of $K$. \medskip
  
  Throughout the paper, all Huber rings and pairs will be assumed to be complete. \medskip
  
  To avoid set-theoretic issues, we fix an uncountable cardinal $\kappa$ as in \cite[Lemma 4.1]{Scholze3}: $\kappa$ is a strong limit cardinal; the cofinality of $\kappa$ is uncountable;  for all cardinals $\lambda<\kappa$, there is a strong limit cardinal $\kappa_{\lambda}<\kappa$ such that the cofinality of $\kappa_{\lambda}$ is greater than $\lambda$.\footnote{There exists an arbitrary large such cardinal $\kappa$.}
  \medskip

  We say that an analytic adic space $X$ is \textit{$\kappa$-small} if the cardinality of the underlying topological space $|X|$ is less than $\kappa$, and for all open affinoid subspaces $\Spa(R, R^+)\subset X$, the ring $R$ has cardinality less than $\kappa$.  In this paper, all the analytic adic spaces will be assumed to be $\kappa$-small.
  \medskip

  We say that a profinite set is \textit{$\kappa$-small} if it has cardinality less than $\kappa$. We denote by $*_{\kappa{\text-}\pet}$ the site of $\kappa$-small profinite sets, with coverings given by finite families of jointly surjective maps. Recall that the category of $\kappa$-condensed sets/groups/rings/etc. is defined as the category of sheaves on $*_{\kappa{\text -}\pet}$ with values in  sets/groups/rings/etc., \cite[Definition 2.1]{Scholzecond}, and it is equivalent to the category of contravariant functors from $\kappa$-small extremally disconnected sets to sets/groups/rings/etc. taking finite disjoint unions to finite products, \cite[Proposition 2.7]{Scholzecond}.  \medskip
  
  Unless explicitly stated otherwise, all condensed sets will be $\kappa$-condensed sets (and often the prefix ``$\kappa$'' is tacit). We will denote by $\CondSet$ the category of $\kappa$-condensed sets, and by $\CondAb$ the category of $\kappa$-condensed abelian groups. We denote by $\Solid\subset \CondAb$ the full subcategory of $\kappa$-solid abelian groups (see Theorem \ref{mainsolid}, Notation \ref{solidnot}, and Remark \ref{solidcut}). \medskip
  
   All condensed rings will be $\kappa$-condensed commutative unital rings. Given a ($\kappa$-)condensed ring $A$, we denote by $\Mod_A^{\cond}$ the category of $A$-modules in $\CondAb$. We write $\underline{\Hom}_A(-, -)$ for the internal Hom in the category $\Mod_A^{\cond}$ (and in the case $A=\Zz$, we often omit the subscript $\Zz$). \medskip
   
   Given  a condensed group $G$, and a $G$-module $M$ in $\CondAb$, the \textit{condensed group cohomology of $G$ with coefficients in $M$} will be denoted by  
   $$R\Gamma_{\underline{\cond}}(G, M):=R\underline{\Hom}_{\Zz[G]}(\Zz, M)\in D(\CondAb)$$
  where $\Zz$ is endowed with the trivial $G$-action (see Appendix \ref{ccg}). \medskip
  
  We denote by $T\mapsto \underline{T}$ the functor from the category of topological spaces/groups/rings/etc. to the category of $\kappa$-condensed sets/groups/rings/etc., where $\underline{T}$ is defined via sending a $\kappa$-small profinite set $S$ to the set/group/ring/etc. of the continuous functions $\mathscr{C}^0(S, T)$.

  \medskip

  \begin{Acknowledgments}
    I am indebted to Arthur-César Le Bras for explaining to me the relevance of the relative fundamental exact sequence of $p$-adic Hodge theory in the study of the  $p$-adic pro-étale cohomology of rigid-analytic varieties, and for numerous suggestions and comments. I heartily thank Wiesława Nizioł and Pierre Colmez for encouraging me to write this paper. I also thank Wiesława Nizioł for suggesting to generalize some previously obtained results, and for her helpful comments. I am very grateful to Dustin Clausen for answering with enthusiasm my questions on condensed mathematics, and for explaining to me much of the content of Appendix \ref{condfun}. I thank Johannes Anschütz, Sebastian Bartling, and Akhil Mathew for very helpful discussions. I also thank Johannes Anschütz, K\k{e}stutis Česnavičius, Haoyang Guo, and David Hansen for their precise comments on an earlier version of this manuscript.
   
   Parts of this manuscript were written while visiting the Mathematical Research Institute of Oberwolfach, which I thank for the hospitality.
   
  \end{Acknowledgments}
 
  \clearpage
  
  \section{\textbf{Condensed cohomology groups}}\label{condcohgroups}
  \sectionmark{}
  
  In this section, we endow the cohomology groups of a pro-étale abelian sheaf on an analytic adic space defined over $\Zz_p$ with a structure of condensed abelian group. Then, we discuss several other situations in which a cohomology group has a natural condensed structure. \medskip
  
  \subsection{The pro-étale topology}
  
  \begin{notation}
  Let $\kappa$ be a cut-off cardinal as in \S \ref{conve}, and let $\Perf_{\kappa}$ denote the category of $\kappa$-small perfectoid spaces of characteristic $p$. In the following, all diamonds are assumed to live as pro-étale sheaves on $\Perf_{\kappa}$, \cite[Definition 11.1]{Scholze3}. 

  Given a diamond $Y$, we denote by $Y_{\ett}$ its  étale site, and by $Y_{\qpet}$ its (bounded by $\kappa$) quasi-pro-étale site, \cite[Definition 14.1]{Scholze3}. 
  \end{notation}

  We begin by recollecting some generalities on the pro-étale topology that will be used throughout the paper. The following class of perfectoid spaces will serve as a convenient basis for the quasi-pro-étale site.

  \begin{df}[{\cite[Definition 1.1]{MW}}]\label{wcon}
  A \textit{w-contractible space} $Y$ is a w-strictly local perfectoid space (\cite[Definition 7.17]{Scholze3}) such that the set of its connected components $\pi_0(Y)$ is an extremally disconnected profinite set.
  \end{df}

  We recall that a $w$-contractible space is in particular a strictly totally disconnected perfectoid space, \cite[Definition 7.15, Proposition 7.16]{Scholze3}. The key result about w-contractible spaces is the following.
  
  \begin{lemma}[{\cite[Lemma 1.2]{MW}}]\label{wfund}\label{wbasis} 
   Let $Y$ be a diamond. 
   \begin{enumerate}[(i)]
    \item\label{wfund:1} The w-contractible spaces in $Y_{\qpet}$ form a basis of the site $Y_{\qpet}$. 
    \item\label{wfund:2} Every w-contractible space $U\in Y_{\qpet}$ is quasi-pro-étale weakly contractible, i.e. any pro-étale cover of $Y$ admits a section. In particular, for any such $U$ and any sheaf of abelian groups $\cl F$ on $Y_{\pet}$ we have $H_{\qpet}^i(U, \cl F)=0$ for $i>0$.
   \end{enumerate}
  \end{lemma}
   
   \begin{cor}
     Let $Y$ be a diamond. The topos $Y_{\qpet}^\sim$ is locally weakly contractible (\cite[Definition 3.2.1]{BS}), and hence replete.
   \end{cor}
   \begin{proof}
    The first assertion follows from Lemma \ref{wfund}, and then the second one follows from \cite[Proposition 3.2.3, (1)]{BS}.
   \end{proof}

  An important feature of the category underlying the site $Y_{\qpet}$ is that it is  ``tensored over $\kappa$-small profinite sets'', as we now recall.
  
  \begin{rem}\label{tensprof}
   Let $Y$ be a diamond. Given a $\kappa$-small profinite set $S$, by abuse of notation we also denote by $S$ the functor on $\Perf_{\kappa}$ sending a perfectoid space $Z$ to the continuous functions $\mathscr{C}^0(|Z|, S)$; this defines a pro-étale sheaf on $\Perf_{\kappa}$. Then, we observe that $Y\times S\in Y_{\qpet}$. In fact, writing $S=\varprojlim_{i\in I} S_i$ as a cofiltered limit of finite sets $S_i$ along a $\kappa$-small index category $I$,\footnote{Such a presentation of $S$ exists by the proof of \cite[Tag 08ZY]{Thestack}.} and recalling \cite[Definition 10.1]{Scholze3}, we have
   $$Y\times S=\varprojlim_{i\in I}(Y\times S_i)=\varprojlim_{i\in I}\coprod_{s_i\in S_i} Y\in Y_{\qpet}.$$
   In general, denoting $S_Y:=Y\times S$, for any $Y'\in Y_{\qpet}$, we have $Y'\times_{Y}S_Y\in Y_{\qpet}$.
  \end{rem}

  Next, we restrict our attention to the category of analytic adic spaces defined over $\Zz_p$. \medskip
  
  Recall that there is a natural functor $X\mapsto X^\diamondsuit$ from the category of analytic adic spaces defined over $\Spa(\Zz_p, \Zz_p)$ to the category of locally spatial diamonds, satisfying $|X|=|X^\diamondsuit|$ and $X_{\ett}\cong X_{\ett}^\diamondsuit$ (see \cite[Definition 15.5, Lemma 15.6]{Scholze3}). We can then give the following definitions.
  
  \begin{df}\label{qpet}
   Let $X$ be an analytic adic space defined over $\Spa(\Zz_p, \Zz_p)$. We denote by $$X_{\pet}:=X^\diamondsuit_{\qpet}$$ its \textit{pro-\'etale site}.\footnote{Note that the definition of pro-étale site that we adopt is different from the one given in \cite{Scholze} (cf. \cite[Remark 4.1.11]{BS}). However, the results on $X_{\pet}$ of \textit{loc. cit.} that we will need still hold with this definition.} 
  \end{df}
  
  Recalling \S \ref{conve}, in what follows, we denote by $*_{\kappa{\text -}\pet}$ the site of $\kappa$-small profinite sets, with coverings given by finite families of jointly surjective maps. Moreover, we denote by $\CondSet$ the category of $\kappa$-condensed sets, and by $\CondAb$ the category of $\kappa$-condensed abelian groups.
  
  \begin{df}\label{Hc}
   Let $X$ be an analytic adic space defined over $\Spa(\Zz_p, \Zz_p)$. We denote by 
   $$f_{\cond}: X_{\pet}\to *_{\kappa{\text -}\pet}$$
   the natural morphism of sites defined by sending $S\in *_{\kappa{\text -}\pet}$ to $X^\diamondsuit\times S\in X_{\pet}$.
   
   Given $\cl F$ a sheaf of abelian groups on $X_{\pet}$, we define the complex of $D(\CondAb)$
   $$R\Gamma_{\petc}(X, \cl F):=Rf_{\cond *} \cl F$$
   with $i$-th cohomology, for $i\ge 0$, the \textit{condensed pro-étale cohomology group}
   $H^i_{\petc}(X, \cl F)=R^if_{\cond *}\cl F$,
   an object of $\CondAb$.
  \end{df}

  \begin{rem}
  Note that the condensed pro-étale cohomology groups are an enrichment of the usual cohomology groups, in fact, for $i\ge 0$, we have
  $H^i_{\petc}(X, \cl F)(*)=H^i_{\pet}(X, \cl F)$ as abelian groups.
  \end{rem}

  \begin{rem}\label{geopoint}
  In the case $X$ is an analytic adic space over $\Spa(C, \cl O_C)$, then Definition \ref{Hc} agrees with Definition \ref{Hcintro}. In fact, the category of sheaves of sets on the site $\Spa(C, \cl O_C)_{\pet}$ is equivalent to the category of $\kappa$-condensed sets. For this, we first note that the topos associated to the site $\Spa(C, \cl O_C)_{\pet}$ is equivalent to the topos associated to the subsite $\Spa(C, \cl O_C)_{\pet}^{\aff}$, consisting of the affinoid pro-étale maps in $\Spa(C, \cl O_C)_{\pet}$, \cite[Definition 7.8, (i)]{Scholze3}. We recall that a map $U\to \Spa(C^\flat, \cl O_{C^\flat})$, from a perfectoid space $U$ to the tilt of $\Spa(C, \cl O_C)$, is an affinoid pro-étale map if and only if $U=\Spa(C^\flat, \cl O_{C^\flat})\times S$ for some $\kappa$-small profinite set $S$.   
  Then, the functor sending $U\in \Spa(C, \cl O_C)_{\pet}^{\aff}$ to the underlying topological space $|U|$ induces an equivalence of categories between the category underlying $\Spa(C, \cl O_C)_{\pet}^{\aff}$ and the category of $\kappa$-small profinite sets.
  Now, the claim easily follows unraveling the definition of the covering families of $\Spa(C, \cl O_C)_{\pet}^{\aff}$.
  \end{rem}
 
  \subsection{Condensed structure à la \v{C}ech}\label{summary}
  In this paper, we will have to deal with sheaves of topological abelian groups/rings/etc. on analytic adic spaces (e.g. the structure sheaf). In some cases, it is reasonable to equip the cohomology groups of such sheaves with a condensed structure à la \v{C}ech, using the fact that their local sections have a natural topology, and then to ask how this compares with Definition \ref{Hc}. This is explained in the following remarks, that we will use to define a structure of condensed abelian group on the geometric cohomology of the pro-étale period sheaves, in \S\ref{petsheaves} (checking its agreement with Definition \ref{Hc}), and on the cohomology of coherent sheaves on rigid-analytic varieties, in \S \ref{coherent}. \medskip
  
  We start with some reminders on the category $\CondAb$ that we will use throughout the paper.
  
  \begin{theorem}[{\cite[Theorem 2.2]{Scholzecond}}]\label{condbasic}
   The category of $\kappa$-condensed abelian groups $\CondAb$ is an abelian category which satisfies the Grothendieck’s axioms (AB3), (AB3*), (AB4), (AB4*), (AB5) and (AB6), and it is generated by a set of compact projective objects.\footnote{Recall also that, instead, the category of all condensed abelian groups (\cite[Definition 2.11]{Scholzecond}) has a class, and not a set, of compact projective generators.} 
  \end{theorem}

  \begin{rem}\label{summary:1}
    In particular, we deduce from Theorem \ref{condbasic} that, for any small site $\cl C$, the  category of sheaves on $\cl C$ with values in $\CondAb$ has enough injectives.\footnote{More precisely, it has functorial injective embeddings, \cite[Tag 0139, Tag 079H]{Thestack}.} \smallskip
    
    Then, given $\cl M$ a complex in the bounded below derived category $D^{+}\left(\Shv_{\CondAb}(\cl C)\right)$ of sheaves on $\cl C$ with values in $\CondAb$, for any object $V\in \cl C$ we can define the sheaf cohomology complex $$R\Gamma(V, \cl M)\in D(\CondAb)$$ by taking an injective resolution of $\cl M$.\footnote{An alternative to ``taking injective resolutions'' would be to work with the $\infty$-category of sheaves on $\cl C$ with values in the derived $\infty$-category of condensed abelian groups $\cl D(\CondAb)$. However, the theory of $\infty$-categories will not play a crucial role in this paper.} For $i\in \Zz$, we denote by $H^i(V, \cl M)$ its $i$-th cohomology, which is an object of $\CondAb$.
  \end{rem}

  \begin{rem}\label{summary:2}
   Let $\cl C$ be a small site with a basis $\cl B$. Let $\cl F$ be a sheaf of topological abelian groups on $\cl B$.
  Since the functor $T\mapsto \underline{T}$ from topological abelian groups to $\CondAb$ preserves limits, the presheaf given by
  $$\cl B^{\opp}\to \CondAb: U\mapsto \underline{\cl F(U)}$$
  is a sheaf. We denote by $\underline{\cl F}$ the associated sheaf on $\cl C$ with values in $\CondAb$, and, for $V\in \cl C$, we consider the sheaf cohomology complex $R\Gamma(V, \underline{\cl F})\in D(\CondAb)$, defined following Remark \ref{summary:1}.
  By Verdier's hypercovering theorem, \cite[Theorem 8.6]{DHI}, we have that $H^i(V, \underline{\cl F})$ has the usual sheaf cohomology group as underlying abelian group, i.e.
  $$H^i(V, \underline{\cl F})(*)=H^i(V, \cl F)$$
  as abelian groups (here, $\cl F$ is regarded as a sheaf of abelian groups).
  Note that Verdier's hypercovering theorem, combined with Remark \ref{setcut}, also shows that, if the basis $\cl B$ has cardinality less than $\kappa$, and, for all $U\in \cl B$, the topological space $\cl F(U)$ is T1 with cardinality less than $\kappa$, then $H^i(V, \underline{\cl F})$ does not depend on our choice of the cardinal $\kappa$.
  The previous discussion also holds replacing $\cl F$ with a bounded below complex of such sheaves on $\cl C$.
  \end{rem}

 \subsection{Comparison to the étale topology}\label{cet}

  In the rest of this section,  we denote by $X$ an analytic adic space over $\Zz_p$, and we identify it with the associated diamond $X^\diamondsuit$. \medskip
  
  We remark that Definition \ref{Hc} takes advantage of Remark \ref{tensprof}, which obviously fails for the étale site $X_{\ett}$. In particular, by applying the pushforward functor along the natural morphism of sites $X_{\pet}\to X_{\ett}$ one may lose ``topological information''. However, for the comparison theorems proven in this paper, we will need to pass from the pro-étale site to the étale site, still retaining the information captured by profinite sets. Thus, to remedy to this issue, we give the following definition, which will play a crucial role later on (see Corollary \ref{frompoincare}).
  
  \begin{df}\label{crulo} Let $\tau\in\{\pet, \ett \}$.
   The site $X_{\tau, \cond}$ has underlying category the pairs $(U, S)$ with $U\in X_{\tau}$ and $S\in *_{\kappa{\text -}\pet}$, and coverings given by the families of morphisms $$\left\{(f_i, s_j): (U_i, S_j)\to (U, S)\right\}_{(i, j)\in I\times J}$$ with $\{f_i: U_i\to U\}_{i\in I}$ a covering of $X_{\tau}$ and $\{s_j: S_j\to S\}_{j\in J}$ a covering of $*_{\kappa{\text -}\pet}$.
  \end{df}

  \begin{rem}
    Note that $X_{\tau, \cond}$ is indeed a site. There is a natural projection of sites $$\pi:X_{\tau, \cond}\to X_{\tau}$$ 
  given by sending $U\in X_{\tau}$ to the pair $(U, *)\in X_{\tau, \cond}$. There is also a natural morphism of sites 
  $$\mu: X_{\pet}\to X_{\tau, \cond}$$
  defined by sending a pair $(U, S)\in X_{\tau, \cond}$ to $u(U)\times S\in X_{\pet}$, where $u: X_{\tau}\to X_{\pet}$ denotes the natural continuous functor.
  \end{rem}
  
  \begin{rem}\label{cohcond}
   We have an equivalence between the category of sheaves of sets on $X_{\tau, \cond}$ and the category of sheaves of $\kappa$-condensed sets on $X_{\tau}$
   \begin{equation}\label{triangled}
    \Shv_{\Set}(X_{\tau, \cond})\overset{\sim}{\longrightarrow} \Shv_{\CondSet}(X_{\tau}): \cl F\mapsto \cl F^{\blacktriangledown}
   \end{equation}
   where the sheaf $\cl F^{\blacktriangledown}$ is defined by assigning to $U\in X_{\tau}$ the $\kappa$-condensed set $\cl F^{\blacktriangledown}(U):S\mapsto \cl F(U, S)$.
   A quasi-inverse of the functor (\ref{triangled}) is given by sending a sheaf of $\kappa$-condensed sets $\cl G$ on $X_{\tau}$ to the sheaf of sets on $X_{\tau, \cond}$ defined by assigning to $(U, S)\in X_{\tau, \cond}$ the set $\cl G(U)(S)$.
   Under this equivalence, an abelian group object/ring object/etc. in $\Shv_{\Set}(X_{\tau, \cond})$, i.e. a sheaf on $X_{\tau, \cond}$ with values in abelian groups/rings/etc., corresponds to a sheaf on $X_{\tau}$ with values in $\kappa$-condensed abelian groups/rings/etc. Moreover, the functor (\ref{triangled}) induces an equivalence of derived categories
   $$D\left(\Shv_{\Ab}(X_{\tau, \cond})\right)\overset{\sim}{\longrightarrow} D\left(\Shv_{\CondAb}(X_{\tau})\right): \cl F\mapsto \cl F^{\blacktriangledown}.$$
  \end{rem}

  Then, recalling Remark \ref{summary:1}, we can give the following definition.
  
  \begin{df}
  Given $\cl F\in D^{+}\left(\Shv_{\Ab}(X_{\tau, \cond})\right)$ we define the complex of $D(\CondAb)$
   $$R\Gamma_{\tau, \cond}(X, \cl F):=R\Gamma_{\tau}(X, \cl F^{\blacktriangledown})$$
   with $i$-th cohomology, for $i\in \Zz$, denoted by $H^i_{\tau, \cond}(X, \cl F)$,
   an object of $\CondAb$.
  \end{df}

  In the case $\tau=\pet$, the latter definition gives us only a slightly alternative point of view on Definition \ref{Hc}.
  \begin{lemma}\label{pet=petcond}
   Let $\cl F$ be a sheaf of abelian groups on $X_{\pet}$. Let $\mu: X_{\pet}\to X_{\pet, \cond}$ be the natural morphism of sites.  Then, for all $i\ge 0$, we have
   $$H^i_{\petc}(X, \cl F)=H_{\pet, \cond}^i(X, R\mu_* \cl F)$$
   as condensed abelian groups.\footnote{Of course, we have $H_{\pet, \cond}^i(X, R\mu_* \cl F)(*)=H_{\pet}^i(X, \cl F)$ as abelian groups.}
  \end{lemma}
  \begin{proof}
   Let $f_{\cond}: X_{\pet}\to *_{\kappa{\text -}\pet}$ denote the morphism of sites of Definition \ref{Hc}. It suffices to observe that the pushforward functor $f_{\cond *}$ is given by the composition
   $$\Shv_{\Ab}(X_{\pet})\overset{\mu_*}{\longrightarrow}\Shv_{\Ab}(X_{\pet, \cond})\overset{\blacktriangledown}{\longrightarrow}\Shv_{\CondAb}(X_{\pet})\overset{\Gamma(X, -)}{\longrightarrow}\CondAb$$
   and then use the associated Grothendieck spectral sequence. In fact, for any $\kappa$-small extremally disconnected set $S$, we have
   $$(f_{\cond *}\cl F)(S)=\cl F(X\times S)=(\mu_*\cl F)(X, S)=\Gamma\left(X, (\mu_*\cl F)^{\blacktriangledown}\right)(S)$$
   and one readily checks that the restriction morphisms agree as well.
  \end{proof}

  \section{\textbf{Schneider--Stuhler's theorem}}\label{SSmodif}
  \sectionmark{}

 In  \cite{SS} Schneider--Stuhler computed the cohomology groups of the Drinfeld upper half-spaces over $K$, for a general abelian sheaf cohomology theory that satisfies some natural axioms, the most important one requiring the homotopy invariance with respect to the $1$-dimensional open unit disk. Examples of such cohomology theories include the de Rham cohomology and the étale cohomology with coefficients in a finite ring whose order is prime to $p$. In this section, we slightly change the axioms of Schneider--Stuhler, \cite[\S 2]{SS}, in order to include the (condensed) geometric pro-étale cohomology of the period sheaves $\Bb_{\dR}$ and $\Bb_e$.\footnote{See \S\ref{petsheaves} for the definition of such pro-étale period sheaves.} Then, we highlight the main points of the proof of \textit{loc. cit.} where, due to such changes, an additional argument is needed.

 \subsection{Drinfeld upper half-space}

 We start by recalling the following definition.
 
 \begin{df}\label{defdrinfeld}
  Let $K$ be a finite extension of $\Qq_p$. Let $d\ge 1$ be an integer. The \textit{Drinfeld upper half-space} over $K$ of dimension $d$ is defined by
  $$\Hh_K^d:=\Pp^d_K\setminus \bigcup_{H\in \cl H} H$$
  where $\cl H:=\Pp((K^{d+1})^*)$ denotes the set of the $K$-rational hyperplanes in the rigid-analytic projective space $\Pp^d_K$.
 \end{df}

 \begin{rem}\label{coverdrinfeld}
  The Drinfeld upper half-space $\Hh_K^d$ is an admissible open of the rigid-analytic projective space $\Pp_K^d$, \cite[\S 1, Proposition 1]{SS}, in particular we can (and we do) regard it as a rigid-analytic variety over $K$. Moreover, $\Hh^d_K$ has an increasing admissible covering $\{U_n\}_{n\in \Nn}$ defined as follows: for a fixed uniformizer $\varpi$ of $\cl O_K$ (see \S \ref{conve}), and $n\in \Nn$, we set $\cl H_n:=\Pp((\cl O_K^{d+1})^*/\varpi^n)$ and
  \begin{equation}\label{Un}
     U_n:=\Pp_K^d\setminus \bigcup_{H\in \cl H_n} N_n(H)
  \end{equation}
  where, for a hyperplane $H\in \cl H$, we denote by
  $N_n(H)$ the $|\varpi|^n$-neighbourhood of $H$.\footnote{Namely, chosen a unimodular representative $\ell_H\in (\cl O_K^{d+1})^*$ among the linear forms defining $H$, the subset $N_n(H)$ of $\Pp^d_K$ is defined by identifying the conjugate points over $K$ of the set  $\{z\in \Pp^d(C):|\ell_H(z)|\le |\varpi|^n\}$, where we use unimodular coordinates to represent a point of $\Pp^d(C)$. See also \cite[\S 1, (C)]{SS}.}
 \end{rem}

  We note that $\Hh_K^d$ admits a natural action of the group $\GL_{d+1}(K)$ via homographies.

 \subsection{Schneider--Stuhler's axioms} Next, we define a variant of the cohomology theories for which Schneider--Stuhler were able to compute the cohomology groups of the Drinfeld upper half-space. Using that the category of $\kappa$-condensed abelian groups $\CondAb$ is an abelian category satisfying the same Grothendieck’s axioms of the category of abelian groups, and it is generated by a set of compact projective objects, we will also be able to deal with cohomology theories defined by a complex of sheaves with values in $\CondAb$ (see Remark \ref{summary:1}). \medskip
 
 \begin{convnot}
   Let $\RigSm_K$ denote the category of quasi-separated smooth rigid-analytic varieties over $K$, endowed with a fixed Grothendieck topology $\tau$ that is finer than the analytic topology.\footnote{For example, $\tau$ can be the analytic, étale, or pro-étale topology.} We denote by $D^{\ge 0}(\RigSm_{K})$  the derived category of complexes of sheaves on $\RigSm_K$, concentrated in non-negative degrees, with values in $\CondAb$. 
   
   We fix a complex $\cl F\in D^{\ge 0}(\RigSm_{K})$, and we denote by
  $$H^\bullet:=H^\bullet(-, \cl F): \RigSm_{K}\to \CondAb$$
  the sheaf cohomology with coefficients in $\cl F$ defined on $\RigSm_{K}$. Moreover, given $X\in \RigSm_{K}$, and $U\subseteq X$ an open subvariety, we denote by $H^\bullet(X, U):=H^\bullet(X, U; \cl F)$ the relative sheaf cohomology of the pair $(X, U)$ with coefficients in $\cl F$.\footnote{We recall that the \textit{relative sheaf cohomology} of the pair $(X, U)$ with coefficients in $\cl F$ is the derived functor of the ``sections of $\cl F$ on $X$ vanishing on $U$'' functor. It can be thought as the ``cohomology with support in $X\setminus U$''.} 
 \end{convnot}

  \begin{df}\label{axioms}
  Let $A$ be a condensed ring. The sheaf cohomology theory $H^\bullet=H^\bullet(-, \cl F)$ satisfies the \textit{axioms of Schneider--Stuhler (relative to $A$)} if takes values in the category of $A$-modules in $\CondAb$ $$H^\bullet: \RigSm_K\to \Mod_A^{\cond}$$ and the following conditions are satisfied.
   \begin{enumerate}
    \item\label{axiom:1} The \textit{homotopy invariance} with respect to the $1$-dimensional open unit disk $\accentset{\circ}{\Dd}_K$ over $K$, i.e. for any affinoid space $X\in \RigSm_K$, the natural projection $X\times \accentset{\circ}{\Dd}_K\to X$ induces an isomorphism in cohomology
    $$H^\bullet(X)\overset{\sim}{\to} H^\bullet(X\times \accentset{\circ}{\Dd}_K).$$
    \item\label{axiom:2} There is a  morphism in $D^{+}(\RigSm_{K})$
    $$\cup: \cl F\otimes^{\LL}_{\Zz}\cl F\to \cl F$$
    called \textit{cup product}, that is associative and commutative with unit $e:\Zz\to \cl F$.
    \item\label{axiom:3} There exists a non-archimedean local field $F$ such that $A$ is a condensed $F$-algebra.\footnote{The non-archimedean local field $F$ is regarded as a condensed ring.} Moreover, we have\footnote{Here, $*_K$ denotes the rigid-analytic point $\Spa(K)$.} 
    $$H^i(*_K)=
    \begin{cases} 
    A & \mbox{if }i=0 \\
    0 & \mbox{if } i>0.
    \end{cases}
    $$
    \item\label{axiom:4} We have that $H^i(\Pp_K^d)=0$ for $i$ odd, or $i>2d$. Moreover, there exists a morphism in $D^{\ge 0}(\RigSm_{K})$, called \textit{cycle class map},\footnote{Here, $\Gg_m$ denotes the multiplicative group sheaf, regarded as a sheaf of discrete condensed abelian groups.}
    $$c:\Gg_m[-1]\to \cl F$$
    satisfying the following condition: if $\eta\in H^2(\Pp_K^d)$ denotes the image of the canonical line bundle $\cl O(1)\in H^1(\Pp_K^d, \Gg_m)$ under the map $H^1(\Pp_K^d, \Gg_m)\to H^2(\Pp_K^d)$ induced by $c$, then, for all $0\le i\le d$, the map
    \begin{equation}\label{cycleiso}
     f^*(-)\cup \eta^i: A=H^0(*_K)\to H^{2i}(\Pp_K^d)
    \end{equation}
    is an isomorphism, where $f:\Pp_K^d\to *_K$ is the structure morphism, and $\eta^i:=\eta \cup \cdots\cup \eta$ is the $i$ times repeated cup product of $\eta$.
    \end{enumerate}
  \end{df}
  
   In Schneider--Stuhler's paper, \cite[\S 2]{SS}, axiom \listref{axiom:3} requires $A$ to be a (discrete) Artinian ring: this is the main condition we have changed. In \textit{loc. cit.} this axiom guarantees the vanishing of $R^1\varprojlim$ in some cases of particular interest. In our case, such vanishing will follow from Lemma \ref{ML} below. \medskip
   
   In the following, we let $F$ be a non-archimedean local field. We say that $V\in \Vect_{F}^{\cond}$ is \textit{finite-dimensional} if there exists an integer $m\ge 0$, and an isomorphism $V\cong F^{\oplus m}$ of condensed $F$-vector spaces.

  \begin{lemma}[cf. Lemma \ref{condML}]\label{ML}
   Let $\{V_n\}$ be a countable inverse system of finite-dimensional condensed $F$-vector spaces. Let $A$ be a condensed $F$-algebra. Then, the inverse system $\{V_n\otimes_{F} A\}$ is Mittag-Leffler.\footnote{Here, and in the following, we adopt Grothendieck's definition of the \textit{Mittag-Leffler condition}, \cite[\S 13.1]{Groth}, which we now recall. Let $\mathcal C$ be an abelian category satisfying (AB3*). We say that an inverse system $(A_i, f_{ij})$ of objects in $\mathcal C$, indexed by a directed set $I$, is \textit{Mittag-Leffler} if, for each $i\in I$, there exists $j\ge i$ such that for all $k\ge j$ we have $f_{ij}(A_j)=f_{ik}(A_k)$.}  In particular,  $R^j\varprojlim_n (V_n\otimes_{F}A)=0$, for all $j>0$.
  \end{lemma}
  \begin{proof}
   By \cite[\S 13.1.2]{Groth} any countable inverse system of finite-dimensional $F$-vector spaces is Mittag-Leffler; we deduce from Lemma \ref{acyclic} that $\{V_n\}$ is Mittag-Leffler, and then that $\{V_n\otimes_{F} A\}$ is Mittag-Leffler too. Equivalently, for every extremally disconnected set $S$, the inverse system $\{(V_n\otimes_{F} A)(S)\}$ is Mittag-Leffler, hence, by \cite[Proposition 13.2.2]{Groth} combined with \cite[Lemma 3.18]{Scholze}, we have that $R^j\varprojlim_n (V_n\otimes_{F}A)=0$, for all $j>0$.
  \end{proof}

  \subsection{Tits buildings}
  
  Before stating Schneider--Stuhler's result, \cite[\S 3, Theorem 1]{SS}, let us recall the following definition.
  
  \begin{df}
   Let $d\ge 1$ be an integer. For $1\le i\le d$, we denote by $\cl T_\bullet^{(i)}$ the \textit{topological Tits $i$-building} of $\GL_{d+1}(K)$, i.e. the simplicial profinite set defined by
  $$\cl T_j^{(i)}:=\{\text{flags } W_0\subseteq \cdots \subseteq W_j \text{ in } (K^{d+1})^*: 1\le \dim_K W_r\le i \;\text{  for all } 0\le r\le j\}$$
  endowed with its natural profinite topology, with face/degeneracy maps given by omitting/doubling one vector subspace in a flag.\footnote{We warn the reader that a topological Tits building differs from a \textit{Bruhat-Tits building}. See also the discussion before \cite[\S 3, Lemma 3]{SS}.}
  \end{df}

  \begin{theorem}\label{mainSS}
   Let $K$ be a finite extension of $\Qq_p$. Let $A$ be a condensed ring, and let $H^\bullet:\RigSm_K\to \Mod_A^{\cond}$ be a cohomology theory that satisfies the axioms of Schneider--Stuhler (\ref{axioms}). Let $d\ge 1$ be an integer, and $G=\GL_{d+1}(K)$. Then, for all $i\ge 0$, there exists a natural isomorphism of $\underline{G}$-modules in $\Mod_A^{\cond}$
   $$H^i(\Hh_K^d)\cong
   \begin{cases} 
   A & \mbox{if }i=0 \\
   \underline{\Hom}(\widetilde H^{i-1}(|\cl T_{\bullet}^{(i)}|, \Zz), A) & \mbox{if }1\le i\le d \\
   0 & \mbox{if } i>d
   \end{cases}
   $$
   where $\widetilde H$ denotes the reduced cohomology, $|\cl T_{\bullet}^{(i)}|$ denotes the geometric realization of the topological Tits $i$-building of $G$, and $\widetilde H^{i-1}(|\cl T_{\bullet}^{(i)}|, \Zz)$ is regarded as a discrete condensed abelian group.
  \end{theorem}

  \begin{proof}
  By axiom \listref{axiom:4}, considering the relative cohomology exact sequence
  \begin{equation}\label{relcoh}
   \cdots\to H^i(\Pp_K^d)\to H^i(\Hh_K^d)\to H^{i+1}(\Pp_K^d, \Hh_K^d)\to H^{i+1}(\Pp_K^d)\to \cdots
  \end{equation}
  we can reduce to compute $H^{i}(\Pp_K^d, \Hh_K^d)$.
  Given $n\in \Nn$, we begin by studying $H^i(\Pp_K^d, U_n)$, with $U_n$ as in (\ref{Un}). Note that we can write
  $$U_n=\bigcap_{H\in \cl H_n} U_n(H)$$
  where $U_n(H):=\Pp_K^d\setminus N_n(H)$ is an open polydisk in the affine space $\Pp_K^d\setminus H$. Then, we can reduce to study the strongly convergent spectral sequence\footnote{By the proof of \cite[\S 2, Proposition 6, Lemma 7]{SS} and the discussion afterwards, for a 
  sheaf cohomology theory $H^\bullet: \RigSm_K\to \Mod_A^{\cond}$, given $X\in \RigSm_K$, and $U_1, \ldots, U_m$ a finite number of open subvarieties of $X$, if we set $U:=U_1\cap\cdots \cap U_m$, then, there exists a strongly convergent spectral sequence
 $$E_1^{r, s}=\bigoplus_{1\le i_0, \ldots, i_{-r}\le m} H^s(X, U_{i_0}\cup \cdots \cup U_{i_{-r}})\implies H^{r+s}(X, U).$$}
  \begin{equation}\label{definitive}
   E_{1, (n)}^{-j, i}:=\bigoplus_{(H_0, \ldots, H_j)\in \cl H_n^{j+1}}H^{i}(\Pp_K^d, U_n(H_0)\cup \cdots \cup U_n(H_j))\implies H^{i-j}(\Pp_K^d, U_n).
  \end{equation}
  We first compute $H^{i}(\Pp_K^d, U_n(H_0)\cup \cdots \cup U_n(H_j))$. Given hyperplanes $H_0, \ldots, H_j\in \cl H_n$, we denote by $\rk(H_0, \ldots, H_j)$ the rank of the following $\cl O_K$-module  (i.e. its minimal number of generators over $\cl O_K$): $\sum_{r=0}^j (\cl O_K/\varpi^n)\ell_{H_r}\subseteq (\cl O_{K}^{d+1})^*/\varpi^n$,  where $\ell_{H_r}\in (\cl O_K^{d+1})^*$ is a unimodular representative among the linear forms defining $H_r$.
  
  Then, one checks that, given $H_0, \ldots, H_j\in \cl H_n$, for  $m=\rk(H_0, \ldots, H_j)-1$, there exists a locally trivial fibration
  $U_n(H_0)\cup \cdots \cup U_n(H_j)\to \Pp_K^m$, whose fibers are open polydisks of dimension $d-m$ (see \cite[\S 1, Lemma 5, Proposition 6]{SS}). Using this observation, and considering the relative cohomology sequence associated to the pair $(\Pp_K^d, U_n(H_0)\cup\cdots\cup U_n(H_j))$, by axioms \listref{axiom:1} and \listref{axiom:4}, we have that
  $$H^i(\Pp_K^d, U_n(H_0)\cup\cdots\cup U_n(H_j))=
  \begin{cases}
  A &\text{if $i$ even and } 2\rk(H_0, \ldots, H_j)\le i\le 2d \\
  0 &\text{otherwise}
  \end{cases}
  $$
 (see \cite[\S 3, Lemma 1]{SS}). Then, the first page of spectral sequence (\ref{definitive}) can be rewritten as
 \begin{equation}\label{1818}
  E_{1, (n)}^{-j, i}=\bigoplus_{\substack{(H_0, \ldots, H_j)\in \cl H_n^{j+1}\\
 \rk(H_0, \ldots, H_j)\le i/2}} A
 \end{equation}
 if $i\in[2, 2d]$ is even, and $E_{1, (n)}^{-j, i}=0$ otherwise. 
 Next, to have a cleaner picture of the differentials of (\ref{1818}), we introduce, for $1\le k\le d$, the simplicial set $Y_\bullet^{(n, k)}$ defined by
  $$Y_j^{(n, k)}:=\{(H_0, \ldots, H_j)\in \cl H_n^{j+1}: \rk (H_0, \ldots, H_j)\le k\}$$
  with face/degeneracy maps given by omitting/doubling one hyperplane in a tuple. In addition, we denote by $\cl K(Y_\bullet^{(n, k)}, A)$ the chain complex on $Y_\bullet^{(n, k)}$ with coefficients in $A$. Then, the terms  and the differentials of the first page of the spectral sequence (\ref{definitive}) identify with
 $$E_{1, (n)}^{-j, i}=\cl K(Y_{j}^{(n, i/2)}, A)
 $$
 if $i\in[2, 2d]$ is even, and $E_{1, (n)}^{-j, i}=0$ otherwise.  
 Passing to the second page of the spectral sequence,
 $$E_{2, (n)}^{-j, i}=H_j(Y_{\bullet}^{(n, i/2)}, A)
 $$
 if $i\in[2, 2d]$ is even, and $E_{2, (n)}^{-j, i}=0$ otherwise. 
 At this point, to compute $H^{i}(\Pp_K^d, \Hh_K^d)$, we want to take the inverse limit of (\ref{definitive}) over $n\in \Nn$. For this, we will need the following observations.
  \begin{rem}\label{remo} Let $F$ be a non-archimedean local field as in axiom \listref{axiom:3}, i.e. such that $A$ is a condensed $F$-algebra. Clearly, the functor $-\otimes_{F} A$ is exact on finite-dimensional condensed $F$-vector spaces. One deduces the following points.
  \begin{enumerate}[(i)]
   \item \label{remo:1} Note that the complex $\cl K(Y_\bullet^{(n, k)}, A)$ can be written as $\cl K(Y_\bullet^{(n, k)}, A)=\cl K(Y_\bullet^{(n, k)}, F)\otimes_{F} A$, where the terms of $\cl K(Y_\bullet^{(n, k)}, F)$ are finite-dimensional condensed $F$-vector spaces. Then, we have that
  $H_j(Y_{\bullet}^{(n, k)}, A)=H_j(Y_{\bullet}^{(n, k)}, F)\otimes_{F} A$,
  where $H_j(Y_{\bullet}^{(n, k)}, F)$ is finite-dimensional.
  \item \label{remo:2} From the spectral sequence (\ref{definitive}) it follows that $H^{i}(\Pp_K^d, U_n)=W_n\otimes_{F}A$, with $W_n$ a finite-dimensional condensed $F$-vector space.
 \end{enumerate}
 \end{rem}
 Now, we note that we have a morphism of spectral sequences
 $$
 \xymatrix@C-=0.6cm@R-=0.6cm{
 E_{2, (n+1)}^{-j, i} \ar[d] \ar@{}[r]^(.22){}="a"^(.5){}="b" \ar@{=>} "a";"b" &  H^{i-j}(\Pp_K^d, U_{n+1}) \ar[d] \\
 E_{2, (n)}^{-j, i} \ar@{}[r]^(.22){}="a"^(.5){}="b" \ar@{=>} "a";"b"          & H^{i-j}(\Pp_K^d, U_{n}) 
 }
 $$
 where the left arrow is induced by the map of simplicial sets $Y_\bullet^{(n+1, k)}\to Y_{\bullet}^{(n, k)}$, and the right arrow is induced by the inclusion $U_n\subset U_{n+1}$. 
 Therefore, for varying $n\in \Nn$, we obtain an inverse system of spectral sequences;  by Remark \ref{remo}\listref{remo:1}, and Lemma \ref{ML}, passing to the limit, we get the following spectral sequence
  \begin{equation}\label{2059}
   E_2^{-j, i}\implies \varprojlim_n H^{i-j}(\Pp_K^d, U_n)
  \end{equation}
 where $E_2^{-j, i}=\varprojlim_n H_j(Y_{\bullet}^{(n, i/2)}, A)$ if $i\in[2, 2d]$ is even, and $E_{2}^{-j, i}=0$ otherwise. Moreover, by (the proof of) \cite[\S 2, Proposition 4]{SS}, for all $i\ge 0$, we have the following natural exact sequence
  $$0\to R^1\varprojlim_n H^{i-1}(\Pp_K^d, U_n)\to H^i(\Pp_K^d, \Hh_K^d)\to \varprojlim_n H^i(\Pp_K^d, U_n)\to 0.$$
  Hence, by Remark \ref{remo}\listref{remo:2}, and Lemma \ref{ML}, we have that
  $$\varprojlim_n H^i(\Pp_K^d, U_n)=H^i(\Pp_K^d, \Hh_K^d).$$ 
  In conclusion, the spectral sequence (\ref{2059}) can be rewritten as
  \begin{equation}\label{spectral12}
   E_2^{-j, i}=\left\{\begin{array}{lr}
                 \varprojlim_n H_j(Y_{\bullet}^{(n, i/2)}, A) &\text{if } i\in[2, 2d]\cap 2\Zz \\
                 0 & \text{otherwise}
                \end{array}\right\}
  \implies H^{i-j}(\Pp_K^d, \Hh_K^d).
  \end{equation}
  Next, we study the second page $E_2^{-j, i}$ of the obtained spectral sequence (\ref{spectral12}). 
  For this, we introduce, for $1\le k\le d$, the simplicial profinite set $Y_\bullet^{(k)}$ defined by
  $$Y_j^{(k)}:=\left\{(H_0, \ldots, H_j)\in \cl H^{j+1}: \dim_K \left(\sum\nolimits_{r=0}^j K\ell_{H_r}\right)\le k\right\}$$
  endowed with the topology induced from the natural profinite one on $\cl H^{j+1}$, and with face/degeneracy maps given by omitting/doubling one hyperplane in a tuple.\footnote{
  Note that $Y_\bullet^{(k)}=\varprojlim_n Y_\bullet^{(n, k)}$, in fact, for $(H_0, \ldots, H_j)\in \cl H^{j+1}$, we have $\dim_K (\sum\nolimits_{r=0}^j K\ell_{H_r})\le k$ if and only if  $\rk(\sum_{r=0}^j (\cl O_K/\varpi^n)\ell_{H_r})\le k$ for all $n\in \Nn$. } 
  
  The definition above is motivated by the following result.
    
  \begin{lemma}\label{YT}\
  \begin{enumerate}[(i)]
   \item \label{YT:1} For each $i\ge 0$, and $1\le k\le d$, there is a natural isomorphism of abelian groups\footnote{Here, $|-|$ denotes the geometric realization.}
   $$H^i(|Y_\bullet ^{(k)}|, \Zz)\cong H^i(|\cl T_\bullet^{(k)}|, \Zz).$$
   \item \label{YT:2} We have $H^i(|\cl T_\bullet^{(k)}|, \Zz)=0$ for $i\neq 0, k-1$, and \footnote{Here, $\LC(-, \Zz)$ denotes the locally constant functions with values in $\Zz$.}
   $$H^0(|\cl T_\bullet^{(k)}|, \Zz)=
   \begin{cases}
    \LC(\Pp((K^{d+1})^*), \Zz) & \text{if } k=1 \\
    \Zz & \text{if } k>1.
   \end{cases}
   $$
   \item \label{YT:3} For each $j\ge 0$, we have a natural isomorphism in $\Mod_{A}^{\cond}$
   $$\varprojlim_n H_j(Y_\bullet^{(n, k)}, A)\cong\underline{\Hom}(H^j(|\cl T_\bullet^{(k)}|, \Zz), A)$$
   where $H^j(|\cl T_\bullet^{(k)}|, \Zz)$ is regarded as discrete condensed abelian group.
  \end{enumerate}
  \end{lemma}
  \begin{proof}
   Part \listref{YT:1} is \cite[\S 3, Proposition 5]{SS}, and part \listref{YT:2} is \cite[\S 3, Proposition 6, (iii)]{SS}.
   For part \listref{YT:3}, we adapt the argument of \cite[\S 3, Proposition 6, (i)]{SS}. Given a field $F$ as in axiom \listref{axiom:3}, by the universal coefficient theorem we have, for each $n\in \Nn$, a natural exact sequence in $\Vect_{F}^{\cond}$
   \begin{equation}\label{above}
    0 \to \underline{\Ext}^1(H^{j+1}(Y_\bullet^{(n, k)}, \Zz), F)\to H_j(Y_\bullet^{(n, k)}, F)\to \underline{\Hom}(H^j(Y_\bullet^{(n, k)}, \Zz), F)\to 0.
   \end{equation}
    We recall that $H^j(Y_\bullet^{(n, k)}, \Zz)$ is a finite module in $\Mod_{\Zz}^{\cond}$,  in particular $H^j(Y_\bullet^{(n, k)}, \Zz)\otimes_{\Zz}F$ is a finite-dimensional condensed $F$-vector space, for all $j\ge 0$. We deduce that the left term of (\ref{above}) vanishes, and thus we have an isomorphism 
    \begin{equation}\label{vani}
     H_j(Y_\bullet^{(n, k)}, F)\cong \underline{\Hom}(H^j(Y_\bullet^{(n, k)}, \Zz), F).
    \end{equation}
   Then, recalling Remark \ref{remo}, applying the functor $-\otimes_{F} A$ to  (\ref{vani}), and passing to the inverse limit over $n\in \Nn$, we have
   \begin{equation}\label{lllast}
    \varprojlim_n H_j(Y_\bullet^{(n, k)}, A)\cong \underline{\Hom}(\varinjlim\nolimits_n H^j(Y_\bullet^{(n, k)}, \Zz), A)
   \end{equation}
   where we used that $\underline{\Hom}(H^j(Y_\bullet^{(n, k)}, \Zz), F)\otimes_{F} A=\underline{\Hom}(H^j(Y_\bullet^{(n, k)}, \Zz), A)$,  which can be checked taking a finite presentation of the module $H^j(Y_\bullet^{(n, k)}, \Zz)$ in $\Mod_{\Zz}^{\cond}$. Now, for all $j\ge 0$, we have $\varinjlim_n H^j(Y_\bullet^{(n, k)}, \Zz)=H^j(|Y_\bullet^{(k)}|, \Zz)$ (see the proof of \cite[\S 3, Lemma 2]{SS}), therefore, the statement of part \listref{YT:3} follows from (\ref{lllast}) and part \listref{YT:1}.
  \end{proof}

  Now, we come back to the spectral sequence (\ref{spectral12}). In view of Lemma \ref{YT}, it can we rewritten as
  $$E_2^{-j, i}=\left\{\begin{array}{lr}
                \underline{\Hom}(H^j(|\cl T_\bullet^{(i/2)}|, \Zz), A) &\text{if } i\in[2, 2d]\cap 2\Zz,\; j\in \{0, \frac{i}{2}-1\} \\
                 0 & \text{otherwise}
                \end{array}\right\}
  \implies H^{i-j}(\Pp_K^d, \Hh_K^d).$$
 
 We note that on the second page of the spectral sequence above there are no differentials between the non-zero terms $E_2^{-j, i}$ with $j=i/2-1$. In conclusion, this spectral sequence degenerates and computes $H^i(\Pp_K^d, \Hh_K^d)$. Then, Theorem \ref{mainSS} follows from the exact sequence (\ref{relcoh})  (see \cite[\S 3, Lemma 7]{SS} for the details); the $\underline{G}$-equivariance of the isomorphism in the statement is a consequence of the remarks after \cite[\S 3, Theorem 1]{SS}. 
 \end{proof}

 \subsection{Generalized Steinberg representations}\label{gsr}
  
  As we will recall, it is possible to give a clearer representation-theoretic interpretation to the result of Theorem \ref{mainSS}.  In fact, the cohomology of the geometric realization of the topological Tits buildings of $\GL_{d+1}(K)$ is related with a generalization of the Steinberg representations. See also \cite[\S 5.2]{CDN1}.
  
 \begin{notation}\label{maintain}
  Let $d\ge 1$ be an integer, and $G=\GL_{d+1}(K)$. Let $e_1^*, \ldots, e_{d+1}^*$ be the canonical basis of $(K^{d+1})^*$. Let $\Delta:=\{1, \ldots, d\}$ and $I\subseteq \Delta$ a subset. We denote by $P_I$ the corresponding parabolic subgroup of $G$, i.e. the stabilizer in $G$ of the flag
  $\tau_I$ defined by $W_{i_r}\subsetneq \cdots \subsetneq W_{i_1}\subsetneq W_{i_0}$,
  where $\Delta\setminus I=\{i_0<\cdots< i_r\}$ and $W_{i_k}:=\sum_{i=i_k+1}^{d+1} Ke_i^*$. Moreover, we denote $X_I:=G/P_I$.
 \end{notation}
 
 \begin{df}\label{speciali}
  Let $M$ be an abelian group, and let $I\subseteq \Delta$. The \textit{locally constant special representations of} $G$ (associated to $I$) with coefficients in $M$ are defined by
  $$\Sp_I(M):=\frac{\LC(X_I, M)}{\sum_{j\in \Delta\setminus I}\LC(X_{{I}\cup \{j\}}, M)}$$
  where $\LC(-, M)$ denotes the locally constant functions with values in $M$.
  
  For $0\le i\le d$, we denote $$\Sp_i(M):=\Sp_{\{1, \ldots, d-i\}}(M)$$
  and for $i>d$, we set $\Sp_i(M)=0$.
 \end{df}
 
 \begin{rem}
  We have that $\Sp_I(M)$ is a smooth $G$-module. Moreover, we observe that there is a natural isomorphism $\Sp_I(M)\cong \Sp_I(\Zz)\otimes_{\Zz} M$. 
 \end{rem}
 
 \begin{rem}
  For $I=\emptyset$, we have that $\Sp_I(M)$ is the locally constant Steinberg representation of $G$ with coefficients in $M$.
 \end{rem}
 
 \begin{prop}\label{titsp}
  For all $1\le i\le d$, we have a natural isomorphism
  $$\widetilde H^{i-1}(|\cl T_\bullet^{(i)}|, \Zz)\cong \Sp_i(\Zz)$$
  where $\widetilde H$ denotes the reduced cohomology.
 \end{prop}
 \begin{proof}
  See \cite[\S 4, Lemma 1]{SS} and \cite[Proposition 5.6]{CDN1}.
 \end{proof}

   \begin{theorem}\label{mainSSS}
   Let $K$ be a finite extension of $\Qq_p$. Let $A$ be a condensed ring, and let $H^\bullet:\RigSm_K\to \Mod_A^{\cond}$ be a cohomology theory that satisfies the axioms of Schneider--Stuhler (\ref{axioms}). Then, for all $i\ge 0$, there exists a natural isomorphism of $\underline{G}$-modules in $\Mod_A^{\cond}$
   $$H^i(\Hh_K^d)\cong \underline{\Hom}(\Sp_i(\Zz), A)$$
   where $\Sp_i(\Zz)$ is regarded as a discrete condensed abelian group.
  \end{theorem}
  \begin{proof}
   This is a consequence of Theorem \ref{mainSS} combined with Proposition \ref{titsp}.
  \end{proof}
  
  \begin{rem}\label{spdual}  In the case $A=F$ is a non-archimedean local field, we can describe in classical topological terms the internal dual $\underline{\Hom}(\Sp_I(\Zz), F)$ that appears in Theorem \ref{mainSSS} (here, we maintain Notation \ref{maintain}).
  
  First, we observe that we can endow $\Sp_I(F)$ with a natural structure of topological $F$-vector space (see also \cite[\S 5.2.2]{CDN1}). In fact, the quotient space $X_I=G/P_I$ is a profinite set; more precisely, we can write $X_I=\varprojlim_n X_{n, I}$ as a countable inverse limit of finite sets $X_{n, I}$, along surjective transition maps. Since each $\LC(X_{n, I}, F)$ is a finite-dimensional $F$-vector space, it has a natural topology coming from the topology on $F$. We equip $\LC(X_I, F)=\varinjlim_n \LC(X_{n, I}, F)$ with the direct limit topology,\footnote{Recall that, given $\{W_n\}$ a countable direct system of locally compact Hausdorff topological $K$-vector spaces, whose transitions maps are immersions, then, the $F$-vector space $\varinjlim_n W_n$, endowed with the direct limit topology, is a topological $F$-vector space (see e.g. \cite[Theorem 4.1]{Hirai}).} and $\Sp_I(F)$ with the induced quotient topology.  Note that we can write $\Sp_I(F)=\varinjlim_n V_n$ as a countable direct limit of finite-dimensional topological $F$-vector spaces $V_n$, along (closed) immersions; then, by Example \ref{LFspace}, we have $\underline{\Sp_I(F)}=\varinjlim_n \underline{V_n}$. We deduce that the internal dual $\underline{\Hom}(\Sp_I(\Zz), F)$ identifies with
  $$\underline{\Hom}_{F}(\underline{\Sp_I(F)}, F)=\varprojlim_n \underline{\Hom}_F(\underline{V_n}, F)=\varprojlim_n \underline{\Hom_{\cont, F}(V_n, F)}=\underline{\Hom_{\cont, F}(\Sp_I(F), F)}$$
  where the dual space $\Sp_I(F)^*:=\Hom_{\cont, F}(\Sp_I(F), F)$ is endowed with the weak topology. We observe that $\Sp_I(F)^*$ is an $F$-Fréchet space, being a countable inverse limit of finite-dimensional topological $F$-vector spaces.
 \end{rem}

  \section{\textbf{Pro-\'etale period sheaves}}\label{petsheaves}
  \sectionmark{}

  In this section, first we recall definitions and basic results on the pro-\'etale period sheaves, and we check that the condensed pro-étale cohomology groups (Definition \ref{Hc}) of the period sheaves are compatible with the ``standard topology'' on the relative period rings. Then, in \S \ref{cls}, we study, in the condensed setting, a Cartan--Leray type spectral sequence for the cohomology of such pro-étale sheaves. Lastly, in \S \ref{rfun}, via the relative Fargues--Fontaine curves, we provide a geometric proof of the relative fundamental exact sequences of $p$-adic Hodge theory. \medskip
  
 \subsection{Preliminaries}
 
  We start with the following definition.
  
  \begin{df}\label{gyu}
    Let $X$ be an analytic adic space over $\Spa(\Zz_p, \Zz_p)$.\footnote{We remind the reader that all the analytic adic spaces are assumed to be $\kappa$-small (see \S \ref{conve}).} We define the \textit{integral $\pet$-structure sheaf} $\widehat{\cl O}_X^+$ and the \textit{$\pet$-structure sheaf} $\widehat{\cl O}_X$ as the sheaves on $X_{\pet}$ satisfying respectively $$\widehat{\cl O}_X^+(Y):=\cl O^+_{Y^{\sharp}}(Y^\sharp),\;\;\;\;\;\;\;\; \widehat{\cl O}_X(Y):=\cl O_{Y^{\sharp}}(Y^\sharp)$$
   for all perfectoid spaces $Y\in X_{\pet}$.
  \end{df}

  We note that, thanks to \cite[Theorem 8.7]{Scholze3}, $\widehat{\cl O}_X^+$ and $\widehat{\cl O}_X$ are indeed sheaves.

 \begin{rem}[Comparison with \cite{Scholze}] Let $X$ be an analytic adic space over $\Spa(\Qq_p, \Zz_p)$ that is either locally noetherian or perfectoid. Let $\nu: X_{\pet}\to X_{\ett}$ be the natural morphism of sites. As in \cite{Scholze}, one can define the following sheaves on $X_{\pet}$: the sheaves $\cl O_X^+=\nu^* \cl O_{X_{\ett}}^+$ and $\cl O_{X}=\nu^* \cl O_{X_{\ett}}$, the \textit{completed integral structure sheaf} $\widehat{\cl O}_X^+=\varprojlim_n \cl O_X^+/p^n$, and the \textit{completed structure sheaf} $\widehat{\cl O}_X=\widehat{\cl O}_X^+[1/p]$. Under the current assumptions on $X$, thanks to \cite[Lemma 2.7]{MW}, the latter two sheaves can be respectively identified with the ones of Definition \ref{gyu}.
 \end{rem}

  Following Scholze, \cite[\S 6]{Scholze}, we are ready to define the de Rham pro-étale period sheaves, that are the sheaf-theoretic version of the classical de Rham period rings of Fontaine.
  
 \begin{df}\label{periodsheaves} Let $X$ be an analytic adic space over $\Spa(\Qq_p, \Zz_p)$.  The following are defined to be sheaves on $X_{\pet}$.
 \begin{enumerate}[(i)] 
   \item The \textit{tilted integral $\pet$-structure sheaf} $\widehat{\cl O}_X^{\flat +}=\varprojlim_{\varphi} \widehat{\cl O}_{X}^+/p$, where the inverse limit is taken along the Frobenius map $\varphi$.
  \item The sheaves $\Aa_{\inf}=W(\widehat{\cl{O}}_X^{\flat +})$ and $\Bb_{\inf}=\Aa_{\inf}[1/p]$. We have a morphism of pro-\'etale sheaves $\theta: \Aa_{\inf}\to\widehat{\cl{O}}^+_{X}$ that extends to $\theta: \Bb_{\inf}\to \widehat{\cl{O}}_{X}$.
  \item We define the \textit{positive de Rham sheaf} $\Bb_{\dR}^+=\varprojlim_{n\in \Nn} \Bb_{\inf}/(\ker\theta)^n$, with filtration given by $\Fil^r \Bb_{\dR}^+=(\ker \theta)^r \Bb_{\dR}^+$.
  \item\label{periodsheaves:last} Let $t$ be a generator of $\Fil^1 \Bb_{\dR}^+$.\footnote{Such a generator exists locally on $X_{\pet}$, it is a non-zero-divisor and unique up to unit, by \cite[Lemma 6.3]{Scholze}.} We define the \textit{de Rham sheaf} $\Bb_{\dR}=\Bb_{\dR}^+[1/t]$, with filtration $\Fil^r \Bb_{\dR}=\sum_{j\in \Zz} t^{-j} \Fil^{r+j}\Bb_{\dR}^+$.
 \end{enumerate}
\end{df}

 Now, we recall the definition of the pro-\'etale sheaf-theoretic version of the ring $B$ introduced by Fargues and Fontaine in their work on the \textit{fundamental curve of $p$-adic Hodge theory}, \cite[\S 1.6]{FF}. See also \cite[\S 8]{LeBras1}. \medskip
 
 In the following, we denote by $v(-)$ the valuation on $\cl O_{C^\flat}$ defined as follows: for $x\in \cl O_{C^\flat}$, we define $v(x)$ as the $p$-adic valuation of $x^{\sharp}\in \cl O_{C}$

 \begin{df}\label{I}
   Let $X$ an analytic adic space over $\Spa(C, \cl O_C)$.  Let $I=[s, r]$ be an interval of $(0, \infty)$ with rational endpoints, and let $\alpha, \beta\in \cl O_{C^{\flat}}$ with valuation $v(\alpha)=1/r$ and $v(\beta)=1/s$. We define the following sheaves on $X_{\pet}$
  $$\Aa_{\inf, I}=\Aa_{\inf}\left[\frac{p}{[\alpha]}, \frac{[\beta]}{p}\right],\;\;\;\; \Aa_I=\varprojlim_n\Aa_{\inf, I}/p^n,\;\;\;\;\Bb_I=\Aa_I[1/p].$$
  Moreover, we define the sheaf on $X_{\pet}$
  $$\Bb=\varprojlim _{I\subset (0, \infty)} \Bb_I$$
  where $I$ runs over all the compact intervals of $(0, \infty)$ with rational endpoints.
 \end{df}
 
 \begin{rem}\label{frr}
  Let $I=[s, r]$ be an interval of $(0, \infty)$ with rational endpoints. The Frobenius $\varphi$ on $\Aa_{\inf}$ induces an isomorphism
  $$\varphi: \Bb_I\overset{\sim}{\to} \Bb_{I/p}$$
  where $I/p:=[s/p, r/p]$ (cf. \cite[\S 1.6]{FF}). In particular, the Frobenius $\varphi$ extends to an automorphism of $\Bb$.
 \end{rem}

  Next, we recall that, on the affinoid perfectoid spaces over $\Spa(C, \cl O_C)$, the period sheaves defined above are given by the expected relative period rings. \medskip
 
 Let $(R, R^+)$ be an affinoid perfectoid algebra over $(C, \cl O_C)$, and let $(R^\flat, R^{\flat +})$ be its tilt.\footnote{Here, $R^+$ is endowed with the $p$-adic topology, and $R^{\flat +}=\varprojlim_{x\mapsto x^p}R^+$ with the induced inverse limit topology.} We define $\Aa_{\inf}(R, R^+)=W(R^{\flat +})$ endowed with the inverse limit topology, $\Bb_{\inf}(R, R^+)=\Aa_{\inf}(R, R^+)[1/p]$ with the direct limit topology, and $\Bb_{\dR}^+(R, R^+)=\varprojlim_{n\in \Nn} \Bb_{\inf}(R, R^+)/(\ker\theta)^n$ with the inverse limit topology.
 Given $\xi \in A_{\inf}=\Aa_{\inf}(C, \cl O_C)$ a generator of the kernel of $\theta: A_{\inf}\to \cl O_C$,\footnote{We recall that such an element $\xi$ generates also the kernel of $\theta:\Aa_{\inf}(R, R^+)\to R^+$ and it is a non-zero-divisor in $\Aa_{\inf}(R, R^+)$ (see \cite[Lemma 6.3]{Scholze}).} we set $\Bb_{\dR}(R, R^+)=\Bb_{\dR}^+(R, R^+)[1/\xi]$ and endow it with the direct limit topology.
 Moreover, for $I\subset (0, \infty)$ a compact interval with rational endpoints, in a similar fashion we define $\Aa_{\inf, I}(R, R^+)$, $\Aa_{I}(R, R^+)$, $\Bb_I(R, R^+)$, $\Bb(R, R^+)$ and endow them with the topologies induced by the one on $\Aa_{\inf}(R, R^+)$. 
 
 \begin{rem}
  Recall that the topology defined above gives to $\Bb_{\dR}^+(R, R^+)$ a structure of $\Qq_p$-Fréchet algebra. Moreover, for $I\subset (0, \infty)$ a compact interval with rational endpoints, we have that $\Bb_I(R, R^+)$ is a $\Qq_p$-Banach algebra, and $\Bb(R, R^+)$ is a $\Qq_p$-Fréchet algebra (see \cite[\S 1.6]{FF}).
 \end{rem}

 \begin{prop}\label{affinoidsections}
  Let $Z=\Spa(R, R^+)$ be an affinoid perfectoid space over $\Spa(C, \cl O_C)$.
  \begin{enumerate}[(i)]
   \item \label{affinoidsections:1} We have $\widehat{\cl{O}}_Z^{+}(Z)=R^+$ and $\widehat{\cl{O}}_Z^{\flat +}(Z)=R^{\flat +}$. Moreover, $H_{\pet}^i(Z, \widehat{\cl{O}}_Z^{\flat +})$ (resp. $H_{\pet}^i(Z, \widehat{\cl{O}}_Z^{+})$) is almost zero, for all $i>0$, with respect to the almost setting defined by the ring $\cl O_C$ (resp. $\cl O_{C^\flat}$) and its ideal of topologically nilpotent elements.
   \item \label{affinoidsections:2} For $\mathbf{A}\in \{\Aa_{\inf}, \Aa_{\inf, I},  \Aa_I\}$, we have $\mathbf{A}(Z)=\mathbf{A}(R, R^+)$, and $H^i_{\pet}(Z, \mathbf{A})$ is almost zero, for all $i>0$, with respect to the almost setting defined by the ring $A_{\inf}$ and its ideal $([p^{\flat}]^{1/p^n})_{n\ge 1}$.
   \item \label{affinoidsections:3} For $\mathbf{B}\in \{\Bb_{\dR}^+, \Bb_{\dR}, \Bb_I, \Bb, \Bb_I[1/t], \Bb[1/t]\}$ we have
   $\mathbf{B}(Z)=\mathbf{B}(R, R^+)$, and $H^i_{\pet}(Z, \mathbf{B})$ vanishes for all $i>0$.
  \end{enumerate}
 \end{prop}
 \begin{proof}
 Part \listref{affinoidsections:1} is \cite[Lemma 4.10, Lemma 5.10]{Scholze} (see also \cite[Proposition 8.8]{Scholze3} for a more general statement). The other statements are a corollary. The statement for $\mathbf{A}=\Aa_{\inf}$ and $\mathbf{B}\in \{\Bb_{\dR}^+, \Bb_{\dR}\}$ is proven in \cite[Theorem 6.5]{Scholze}. The statement for $\mathbf{A}\in \{\Aa_{\inf, I},  \Aa_I\}$ and $\mathbf{B}=\Bb_I$ follows from (the proof of) \cite[Proposition 8.3]{LeBras1}.
  In order to deduce the statement for $\mathbf{B}=\Bb$, we reduce to the previous case, using \cite[Lemma 3.18]{Scholze}, together with Lemma \ref{R1BI} below. The statement for $\mathbf{B}\in \{\Bb_I[1/t], \Bb[1/t]\}$ follows from the latter two cases using that $|Z|$ is quasi-compact and quasi-separated.
 \end{proof}
 
 We used the following lemma.
 
  \begin{lemma}[{\cite[Lemma 3.5]{LeBras2}}]\label{R1BI}
   Let $Z=\Spa(R, R^+)$ be an affinoid perfectoid space over $\Spa(C, \cl O_C)$. Then, for all $j>0$, we have
  $$R^j\varprojlim_I \Bb_I(Z)=0$$
  where  $I$ runs over all the compact intervals of $(0, \infty)$ with rational endpoints.
 \end{lemma}
 \begin{proof}
  For any compact intervals $I\subset J\subset (0, \infty)$ with rational endpoints, the natural map $\Bb_J(R^+, R)\to \Bb_I(R^+, R)$ is a morphism of $\Qq_p$-Banach spaces with dense image. Then, the statement follows from the topological Mittag-Leffler lemma \cite[Remarques 13.2.4]{Groth} applied to the countable inverse system $\{\Bb_I(R^+, R)\}_I$.
 \end{proof}

 Therefore, given $X$ an analytic adic space over $\Spa(C, \cl O_C)$ and $\cl F$ any of the sheaves on $X_{\pet}$ of Proposition \ref{affinoidsections}, denoting by $\cl B$ the basis of $X_{\pet}$ consisting of the affinoid perfectoid spaces $U\in X_{\pet}$, one has that the values of $\cl F$ on $\cl B$ have a natural structure of topological abelian groups. Hence, following Remark \ref{summary:2}, we associate to $\cl F$ a sheaf $\underline{\cl F}$ on $X_{\pet}$ with values in $\CondAb$ such that, for $U\in \cl B$, we have $\underline{\cl F}(U)=\underline{\cl F(U)}$.\footnote{More precisely, Remark \ref{summary:2} applies to the \textit{big} pro-étale site, \cite[Definition 8.1, (i)]{Scholze3}.} The following result shows that the condensed cohomology groups $H_{\pet}^i(X, \underline{\cl F})$ agree with the ones of Definition \ref{Hc}.

 \begin{cor}[{\cite[Corollary 6.6]{Scholze}}]\label{profinite}
  Let $X$ be an analytic adic space over $\Spa(C, \cl O_C)$. Let $\cl F$ be any of the pro-\'etale sheaves on $X$ of Proposition \ref{affinoidsections}. For any $U\in X_{\pet}$ affinoid perfectoid, and $S$ profinite set, we have $$\cl F(U\times S)=\mathscr{C}^0(S, \cl F(U)).$$
  In particular, for all $i\ge0$, we have
  \begin{equation}\label{compatible}
   H^i_{\petc}(X, \cl F)=H_{\pet}^i(X, \underline{\cl F})
  \end{equation}
  as condensed abelian groups.\footnote{In the case $\cl F=\mathbf{B}$ of Proposition \ref{affinoidsections}\listref{affinoidsections:3}, we have that \ref{compatible} also holds as condensed $\Qq_p$-vector spaces.}
 \end{cor}
 \begin{proof}
  We note that for any profinite set $S=\varprojlim_i S_i$ written as a cofiltered limit of finite sets $S_i$, and any discrete topological space $T$, we have
  $$\mathscr{C}^0(S, T)=\varinjlim_i \mathscr{C}^0(S_i, T)$$
  (see e.g. \cite[Lemma 3.16]{MW}). We deduce that, for any $n\ge 1$, we have $$\widehat{\cl{O}}^{+}(U\times S)/p^n=\mathscr{C}^0(S, \widehat{\cl{O}}^{+}(U)/p^n).$$
  Then, the statement follows from Proposition \ref{affinoidsections} with the help of Lemma \ref{bspro}.
 \end{proof}

 \begin{example}\label{profiniteex}
  By Corollary \ref{profinite}, the pro-étale period sheaf $\Bb$ on $\Spa(C, \cl O_C)_{\pet}$ agrees with the condensed $\Qq_p$-algebra $\underline{B}$, where $B=\Bb(C, \cl O_C)$ is Fargues--Fontaine's period ring regarded as a $\Qq_p$-Fréchet algebra.
 \end{example}


 \subsection{Cartan--Leray spectral sequence}\label{cls} In this subsection, using the solid formalism and the results of Appendix \ref{ccg} on condensed group cohomology, we show that for the pro-étale cohomology groups of the period sheaves we have a Cartan--Leray spectral sequence. \medskip

 We start with the following definition.
 
 \begin{df}
   Let $G$ be a profinite group. We say that a map $f:\widetilde X\to X$ of analytic adic spaces over $\Zz_p$ is a \textit{pro-étale $G$-torsor} if the associated map of diamonds $f^\diamondsuit$ is endowed with an action of $G$ on $\widetilde X^{\diamondsuit}$ over $X^{\diamondsuit}$ such that quasi-pro-étale locally on $X^{\diamondsuit}$ there is a $G$-equivariant isomorphism of diamonds $\widetilde X^{\diamondsuit}\cong X^{\diamondsuit}\times G$.\footnote{Here, $G$ is regarded as a pro-étale sheaf on $\Perf_{\kappa}$ via the functor  sending a perfectoid space $Z$ to the continuous functions $\mathscr{C}^0(|Z|, G)$.} In this case, we call $G$ the \textit{Galois group} of $f$.
 \end{df}

 \begin{prop}\label{cart-ler}
  Let $G$ be a profinite group. Let $\widetilde X\to X$ be a pro-étale $G$-torsor of analytic adic spaces over $\Spa(C, \cl O_C)$. Let $\cl F$ be any of the period sheaves on $X_{\pet}$ of Proposition \ref{affinoidsections}. Then, we have a Cartan--Leray spectral sequence
  $$E_2^{i, j}=H^i_{\underline{\cond}}(G, H^j_{\petc}(\widetilde X, \cl F))\implies H^{i+j}_{\petc}(X, \cl F).$$
  In particular, if $\widetilde X$ is affinoid perfectoid and $\cl F$ is any of the period sheaves of Proposition \ref{affinoidsections}\listref{affinoidsections:3}, we have an isomorphism in $D(\CondAb)$
  $$R\Gamma_{\underline{\cond}}(G, \cl F(\widetilde X))\overset{\sim}{\to} R\Gamma_{\petc}(X, \cl F).$$
 \end{prop}
  \begin{proof} Consider the \v{C}ech-to-cohomology spectral sequence
  \begin{equation}\label{aok}
   E_1^{i, j}=H_{\petc}^j(\widetilde X_i, \cl F)\implies H_{\petc}^{i+j}(X, \cl F)
  \end{equation}
  where, for $i\ge 1$, $\widetilde X_i$ denotes the $i$-fold fibre product of $\widetilde X$ over $X$. We note that $\widetilde X_i\cong \widetilde X\times G^{i-1}$ regarded as diamonds.
  
  We claim that for any analytic adic space $Y$ over $\Spa(C, \cl O_C)$ and any profinite set $S$ (like $G^{i-1}$), we have a natural isomorphism in $D(\CondAb)$
  \begin{equation}\label{starr}
   R\Gamma_{\petc}(Y\times S, \cl F)\simeq R\underline{\Hom}(\Zz[S],  R\Gamma_{\petc}(Y, \cl F)).
  \end{equation}
  We note that $R\Gamma_{\petc}(Y, \cl F)\in D(\Solid)$: in fact, using (\ref{compatible}), it suffices to observe that the sheaf $\underline{\cl F}$ takes values in $\Solid$, and to recall that the subcategory $\Solid\subset \CondAb$ is stable under all limits and colimits. Then, by \cite[Corollary 6.1, (iv)]{Scholzecond}, we have a natural isomorphism in  $D(\CondAb)$
  \begin{equation}\label{0101}
   R\underline{\Hom}(\Zz[S]^{\solidif},  R\Gamma_{\petc}(Y, \cl F))\overset{\sim}{\to} R\underline{\Hom}(\Zz[S],  R\Gamma_{\petc}(Y, \cl F)).
  \end{equation}
   Now, assuming first $Y$ qcqs, by \cite[Lemma 7.18]{Scholze3} we can pick a simplicial $\pet$-hypercover $U_\bullet\to Y$ such that all $U_n$ are strictly totally disconnected perfectoid spaces. By pro-étale descent we have
   $$R\Gamma_{\petc}(Y, \cl F)=R\lim_{[n]\in \Delta}R\Gamma_{\petc}(U_n, \cl F)$$
   where $\Delta$ denotes the simplicial category. Hence, using that $R\underline{\Hom}(\Zz[S],  -)$ commutes with derived limits, we can reduce to showing (\ref{starr}) with $Y$ a strictly totally disconnected perfectoid space over $\Spa(C, \cl O_C)$. In the latter case, we have the following natural identifications 
   \begin{equation}\label{amino}
    R\Gamma_{\petc}(Y\times S, \cl F)=\underline{\cl F(Y\times S)}=\underline{\mathscr{C}^0(S, \cl F(Y))}=\underline{\Hom}(\Zz[S],  \underline{\cl F(Y)}).
   \end{equation}
   Here,  for the first identification in (\ref{amino}) we used Corollary \ref{profinite} and the fact that, for any profinite set $S'$, we have $H_{\pet}^i(Y\times S', \cl F)=0$ for all $i>0$: to show this, using the same reduction steps in the proof of Proposition \ref{affinoidsections} and the references therein, together with \cite[Proposition 2.13]{MW}, we can reduce to checking that $H_{\ett}^i(Y\times S', \cl O^+/p)=0$ for all $i>0$, which holds true observing that $Y\times S'$ is a strictly totally disconnected perfectoid space (see e.g. \cite[Lemma 7.19]{Scholze3}) in particular any of its étale covers splits; for the latter identification in (\ref{amino}) we used that, for any profinite set $S'$, $\mathscr{C}^0(S', \mathscr{C}^0(S,\cl F(Y)))\cong \mathscr{C}^0(S'\times S, \cl F(Y))$. Then, in the case $Y$ is a strictly totally disconnected perfectoid space over $\Spa(C, \cl O_C)$, (\ref{starr}) holds true observing that, by the isomorphism (\ref{0101}), using that $\Zz[S]^{\solidif}$ is an internally projective object in $\Solid$ (Proposition \ref{intproj}), the right-hand side of  (\ref{starr}) identifies with  $\underline{\Hom}(\Zz[S]^{\solidif},  \underline{\cl F(Y)})\cong\underline{\Hom}(\Zz[S],  \underline{\cl F(Y)})$.
   
   Combining (\ref{starr}) and (\ref{0101}), and using again that, for any profinite set $S$, $\Zz[S]^{\solidif}$ is an internally projective object in $\Solid$, we have that $$H_{\petc}^j(\widetilde X_i, \cl F)\cong \underline{\Hom}(\Zz[G^{i-1}], H_{\petc}^j(\widetilde X, \cl F)).$$ Then, the statement follows from the spectral sequence (\ref{aok}) and Proposition \ref{cond=cont}\listref{cond=cont:1}.
  \end{proof}

 \subsection{Relative fundamental exact sequences via the Fargues--Fontaine curve}\label{rfun}

 In this subsection, we state the relative fundamental exact sequences of rational $p$-adic Hodge theory, and we provide a geometric proof via the relative Fargues--Fontaine curves. \medskip
 
 \begin{notation}\label{notff}
 Let $S=\Spa(R, R^+)$ be an affinoid perfectoid space over $\Ff_p$. We denote
  $$Y_{\FF, S}:=\Spa(W(R^+), W(R^+))\setminus V(p[p^{\flat}]).$$
  We recall that $Y_{\FF, S}$ defines an analytic adic space over $\Qq_p$, \cite[Proposition II.1.1]{FS}. The $p$-th power Frobenius on $R^+$ induces an automorphism $\varphi$ of $Y_{\FF, S}$ whose action is free and totally discontinuous, \cite[Proposition II.1.16]{FS}. The \textit{Fargues--Fontaine curve relative to $S$} (and $\Qq_p$) will be denoted by
  \begin{equation}\label{FFdef}
   \FF_S:=Y_{\FF, S}/\varphi^{\Zz}.
  \end{equation}
  For $S$ an affinoid perfectoid space over $\Ff_p$, and $I=[s, r]\subset (0, \infty)$ an interval with rational endpoints, we define the open subset
 \begin{equation}
  Y_{\FF, S, I}:=\{|p|^r\le |[p^\flat]|\le |p|^s\}\subset Y_{\FF, S}.
 \end{equation}
 We note that $Y_{\FF, S, I}$ is an affinoid space, as it is a rational open subset of $\Spa(W(R^+), W(R^+))$. \medskip
  
 We denote by $\Bun(\FF_S)$ the category of vector bundles on $\FF_S$.
 Let $\breve{\Qq}_p:=W(\overline{\Ff}_p)[1/p]$, and let $\sigma$ be the automorphism of $\breve{\Qq}_p$ induced by the $p$-th power Frobenius on $\overline{\Ff}_p$. We denote by $\Isoc_{\overline{\Ff}_p}$ the category of isocrystals over $\overline{\Ff}_p$, i.e. the category of pairs $(V, \varphi)$ with $V$ a finite-dimensional $\breve{\Qq}_p$-vector space and $\varphi$ a $\sigma$-semilinear automorphism of $V$. Recall that we have a natural exact $\otimes$-functor
  $$\Isoc_{\overline{\Ff}_p}\to \Bun(\FF_S), \;\;\;\; (V, \varphi)\mapsto \cl E(V, \varphi).$$ 
 For $n\in \Zz$, we denote $\cl O_{\FF_S}(n):=\cl E(\breve{\Qq}_p, p^{-n}\sigma)$. \medskip
 
 We fix a compatible system $(1, \varepsilon_p, \varepsilon_{p^2}, \ldots)$ of $p$-th power roots of unity in $\cl O_C$, which defines an element $\varepsilon\in \cl O_C^{\flat}$. We denote by $[\varepsilon]\in A_{\inf}=\Aa_{\inf}(C, \cl O_C)$ its Teichm\"uller lift,  and $t=\log[\varepsilon]\in B=\Bb(C, \cl O_C)$.
 \end{notation}

 The following lemma will be repeatedly used in the sequel.
 
 \begin{lemma}\label{repp}
  Let $S^{\sharp}$ be an affinoid perfectoid space over $\Spa(C, \cl O_C)$, and let $S=(S^{\sharp})^\diamondsuit$.
  Let $I=[s, r]\subset (0, \infty)$ be an interval with rational endpoints. Then,  we have
  $$\Bb_I(S^{\sharp})=\cl O(Y_{\FF, S, I}),\;\;\;\;\;\;\;\;\Bb(S^{\sharp})=\cl O(Y_{\FF, S}).$$
 \end{lemma}
 \begin{proof}
  The statement follows from Proposition \ref{affinoidsections}\listref{affinoidsections:3}.
 \end{proof}

 \begin{prop}\label{strfs}
  Let $S$ be an affinoid perfectoid space over $\Ff_p$ with an untilt $S^{\sharp}$ over $C$.  We denote $\iota_{S^{\sharp}}:S^{\sharp}\hookrightarrow \FF_S$ the natural closed immersion. For $i\ge 1$, we have a short exact sequence of sheaves on $\FF_S$
 \begin{equation}\label{llog}
  0\to \cl O_{\FF_S}\overset{\cdot t^i}{\to}\cl O_{\FF_S}(i)\to \iota_{S^{\sharp}, *}(\Bb_{\dR}^+(S^{\sharp})/\Fil^i)\to 0
 \end{equation}
 where $\Fil^i\Bb_{\dR}^+(S^{\sharp})=t^i \Bb_{\dR}^+(S^{\sharp})$.
 \end{prop}
 \begin{proof}
  For $i=1$, the exact sequence (\ref{llog}) follows from \cite[Proposition II.2.3]{FS}. The general case follows by induction on $i\ge 1$, considering the following commutative diagram (cf. the proof of \cite[Proposition 8.22]{Colmez}):
  \begin{center}
  \begin{tikzcd}
   & & 0 \arrow[d]& 0 \arrow[d]&  \\
   0 \arrow[r]& \cl O_{\FF_S} \arrow[-,double line with arrow={-,-}]{d}\arrow[r, "\cdot t^i"]& \cl O_{\FF_S}(i) \arrow[d, "\cdot t"]\arrow[r]& \iota_{S^{\sharp}, *}(\Bb_{\dR}^+(S^{\sharp})/\Fil^i) \arrow[d, "\cdot t"]\arrow[r]& 0 \\
   0 \arrow[r]& \cl O_{\FF_S} \arrow[r, "\cdot t^{i+1}"]& \cl O_{\FF_S}(i+1) \arrow[d]\arrow[r]& \iota_{S^{\sharp}, *}(\Bb_{\dR}^+(S^{\sharp})/\Fil^{i+1}) \arrow[d]\arrow[r]& 0 \\
    & & \iota_{S^{\sharp}, *}(\widehat{\cl O}(S^{\sharp})) \arrow[d]\arrow[-,double line with arrow={-,-}]{r}& \iota_{S^{\sharp}, *}(\Bb_{\dR}^+(S^{\sharp})/\Fil^{1}) \arrow[d]&  \\
    & & 0 & 0 &  \\
  \end{tikzcd}
  \end{center}
 \end{proof}

 The following fundamental exact sequences summarize the relevant relations between the various rational period sheaves.

 \begin{prop}\label{suites} Let $X$ an analytic adic space over $\Spa(C, \cl O_C)$. Let $i\ge 0$ be an integer.
 We have the following exact sequences of sheaves on $X_{\pet}$
 \begin{equation}\label{suite2}
  0\to \Bb^{\varphi=p^i}\to \Bb \xrightarrow{\varphi p^{-i}-1} \Bb\to 0 
 \end{equation}
 
 \begin{equation}\label{suite3}
 0\to \Qq_p(i)\to \Bb^{\varphi=p^i}\to \Bb_{\dR}^+/\Fil^i\Bb_{\dR}^+\to 0.
 \end{equation}
 \end{prop}
 \begin{proof}
 By Lemma \ref{wfund}\listref{wfund:1} it suffices to prove the stated exact sequences over $S$ for any $S\in X_{\pet}$ w-contractible affinoid perfectoid space. In the following, we fix such an $S$.
 
 For the exact sequence (\ref{suite2}), we consider the vector bundle $\cl O_{\FF_S}(i)$ over the Fargues--Fontaine curve $\FF_S$. We have
 \begin{equation}\label{ffff}
  R\Gamma(\FF_S, \cl O_{\FF_S}(i))=[H^0(Y_{\FF, S}, \cl O_{Y_{\FF, S}})\xrightarrow{\varphi p^{-i}-1}H^0(Y_{\FF, S}, \cl O_{Y_{\FF, S}})].
 \end{equation}
 Then, we want to show that $H^1(\FF_S, \cl O_{\FF_S}(i))=0$. For $i>0$, this is \cite[Proposition II.2.5(i)]{FS}.
 For $i=0$, by \cite[Proposition II.2.5(ii)]{FS}, we have an isomorphism
 \begin{equation}\label{localsis}
  R\Gamma_{\pet}(S, \Qq_p)\overset{\sim}{\to} R\Gamma(\FF_S, \cl O_{\FF_S})
 \end{equation}
 and, using Lemma \ref{wfund}\listref{wfund:2}, we have
 \begin{equation}\label{vanishlb}
  H^1_{\pet}(S, \Qq_p)=0
 \end{equation}
 as desired. 
 
 Next, we show the exact sequence (\ref{suite3}). The case $i=0$ follows recalling that by (\ref{ffff}) we have $H^0(\FF_S, \cl O_{\FF_S})=H^0(Y_{\FF, S}, \cl O_{Y_{\FF, S}})^{\varphi=1}$. In the case $i\ge 1$, we use Proposition \ref{strfs}: taking cohomology of the short exact sequence (\ref{llog}) of sheaves on $\FF_S$, and using again (\ref{localsis}) combined with (\ref{vanishlb}), we obtain the desired statement.
 \end{proof}
  
 \begin{rem}
 By \cite[Proposition 7.8]{LeBras1} the vanishing (\ref{vanishlb}) also holds true for $S=\Spa(R, R^\circ)^\diamondsuit$ for any sympathetic $C$-algebra $R$.
 \end{rem}

 \begin{cor}\label{profundexact}
  Let $X$ be an analytic adic space over $\Spa(C, \cl O_C)$. We have the following exact sequences of sheaves on $X_{\pet}$
  \begin{equation}\label{fundexacttard}
     0\to \Bb_e\to \Bb[1/t]\overset{\varphi-1}{\to} \Bb[1/t]\to 0
  \end{equation}
  
  \begin{equation}\label{fundexact}
   0\to \Qq_p\to \Bb_e\to \Bb_{\dR}/\Bb_{\dR}^+\to 0
  \end{equation}
  where $\Bb_e:=\Bb[1/t]^{\varphi=1}$.
 \end{cor}
 \begin{proof}
 For $i\ge 0$, multiplying (locally) by $t^{-i}$ the exact sequence (\ref{suite3}), we obtain the exact sequence
 \begin{equation}\label{fili}
  0\to \Qq_p \to (t^{-i}\Bb)^{\varphi=1}\to \Fil^{-i}\Bb_{\dR}/\Bb_{\dR}^+\to 0
 \end{equation}
 where we used that $\varphi(t)=pt$. Then, the exact sequence (\ref{fundexact}) follows taking the filtered colimit over $i\ge 0$ of the exact sequences (\ref{fili}). Similarly, for $i\ge 0$, from the exact sequence (\ref{suite2}), we obtain the exact sequence
 \begin{equation}\label{fili2}
  0\to (t^{-i}\Bb)^{\varphi=1}\to t^{-i}\Bb\overset{\varphi-1}{\to}t^{-i}\Bb\to 0
 \end{equation}
 and (\ref{fundexacttard}) follows taking the filtered colimit over $i\ge 0$ of (\ref{fili2}).
 \end{proof}

 \begin{rem}\label{remonI}
  Let us keep the hypotheses of Proposition \ref{suites}. Then, proceeding as in the proof of (\ref{suite2}), using the presentation of the Fargues--Fontaine curves $\FF_S$ as the quotient of $Y_{\FF, S, [1, p]}$ via the identification $\varphi: Y_{\FF, S, [1, 1]}\cong Y_{\FF, S, [p, p]}$, we also have the following exact sequence of sheaves on $X_{\pet}$
  \begin{equation}\label{suite2bis}
  0\to \Bb_{[1, p]}^{\varphi=p^i}\to \Bb_{[1, p]} \xrightarrow{\varphi p^{-i}-1} \Bb_{[1, 1]}\to 0. 
 \end{equation}
 Using such presentation of the Fargues--Fontaine curves, we can similarly deduce exact sequences of sheaves analogous to (\ref{suite3}), (\ref{fundexacttard}), and (\ref{fundexact}).
 \end{rem}

\section{\textbf{Solid base change for coherent cohomology}}\label{coherent}
\sectionmark{}

 In this section, for $X$ a rigid-analytic variety over $K$, we endow the cohomology groups of a coherent $\cl O_{X}$-module over $X$ with a structure of condensed $K$-vector space. Then, in the category of condensed $K$-vector spaces, we prove a base change result for the coherent cohomology of connected, paracompact, rigid-analytic varieties defined over $K$. \medskip
 
  Following \S \ref{conve}, all of our rigid-analytic varieties will be assumed to be $\kappa$-small. Moreover, we introduce the following convention and notation.

  \begin{convention}\label{convqs}
  All rigid-analytic varieties will be assumed to be quasi-separated.
 \end{convention}

 \begin{notation}\label{notasolid}
   We denote by $\Vect_K^{\cond}$ the category of $\kappa$-condensed $K$-vector spaces, and by $\Vect_K^{\ssolid}\subset \Vect_K^{\cond}$ the symmetric monoidal subcategory of $\kappa$-solid $K$-vector spaces, endowed with the tensor product $\solid_K$ (and often the prefix ``$\kappa$'' is tacit).  See \S \ref{ssk} and Proposition \ref{solidA}.\footnote{See also Remark \ref{setcut} and Remark \ref{solidcut} on set-theoretic bounds.}
 \end{notation}

 \subsection{Coherent cohomology in the condensed world}\label{cohcondw} 
 
 In the following, for a rigid-analytic variety $X$ over $K$, we denote by $X_{\an}$ its analytic site, and by $X_{\ett}$ its étale site. \medskip
 
 As we will need to compute coherent cohomology both on the analytic site and the étale site, we start by  recalling the following definition.
 
 \begin{df}\label{situa}
  Let $X$ be a rigid-analytic variety defined over $K$, and let $\cl F$ be a coherent $\cl O_X$-module over $X$. We denote by $\cl F_{\ett}$ the associated $\cl O_{X_{\ett}}$-module over $X_{\ett}$, given by sending an affinoid étale $f:U\to X$ to the sections $\Gamma(U, f^*\cl F\otimes_{f^*\cl O_X}\cl O_U)$. 
 \end{df}
 
 By \cite[Proposition 3.2.2]{DeJong}, $\cl F_{\ett}$ is indeed a sheaf on $X_{\ett}$.

  \begin{rem}\label{nnew}
  In the situation of Definition \ref{situa}, both $\cl F$ and $\cl F_{\ett}$ are naturally sheaves of topological $K$-vector spaces. More precisely, the values of $\cl F$ (resp. $\cl F_{\ett}$) on an admissible open affinoid $U\subseteq X$ (resp. on an affinoid étale $U\to X$) have a natural structure of $K$-Banach space: in fact, $\cl F(U)$ is a finite module over the $K$-Banach algebra $\cl O_{U}(U)$.  
  \end{rem}
  
 Then, following \S \ref{summary}, we now study the coherent cohomology of rigid-analytic varieties from the point of view of condensed mathematics.
 
 \begin{df}\label{FsolidA}
  Let $X$ be a rigid-analytic variety defined over $K$. Let $\cl F_{\an}$ be a coherent $\cl O_{X_{\an}}$-module over $X_{\an}$, and let $\cl F_{\ett}$ denote the associated $\cl O_{X_{\ett}}$-module over $X_{\ett}$. Let $\tau\in \{\an, \ett\}$, and let $\cl B_{\tau}^{\aff}$ denote the basis for the site $X_{\tau}$ consisting of all $U\in X_{\tau}$ which are affinoid. 
  \begin{enumerate}[(i)]
   \item We denote by $\underline{\cl F_{\tau}}$ the presheaf on $X_{\tau}$, with values in $\Vect_K^{\cond}$, extending the presheaf
   $$(\cl B_{\tau}^{\aff})^{\opp}\to \Vect_K^{\cond}: U\mapsto \underline{\cl F_{\tau}(U)}.$$
   \item  Given $A\in \Vect_K^{\ssolid}$ flat,\footnote{See Definition \ref{platit}.} we denote by $\underline{\cl F_{\tau}}\solid_K A$ the presheaf on $X_{\tau}$, with values in $\Vect_K^{\cond}$, extending the following presheaf\footnote{We warn the reader that it might \textit{not} be true, in general, that for all $V\in X_{\tau}$ we have $(\underline{\cl F_{\tau}}\solid_K A)(V)=\underline{\cl F_{\tau}}(V)\solid_K A$, since the tensor product $\solid_K$ does not commute with all limits.}
   $$(\cl B_{\tau}^{\aff})^{\opp}\to \Vect_K^{\cond}: U\mapsto \underline{\cl F_{\tau}(U)}\solid_K A.$$   
  \end{enumerate}
 \end{df}
 
 The next result shows that the presheaves defined above are indeed sheaves, and translates in our setting some classical results on coherent cohomology.

 \begin{lemma}\label{condensedtate}\label{afterTate}\label{straight}
  Let $X$ be a rigid-analytic variety defined over $K$. Let $\cl F_{\an}$ be a coherent $\cl O_{X_{\an}}$-module over $X_{\an}$, and let $\cl F_{\ett}$ denote the associated $\cl O_{X_{\ett}}$-module over $X_{\ett}$.
  \begin{enumerate}[(i)]
   \item \label{condensedtate:1} The presheaf $\underline{\cl F_{\tau}}$ is a sheaf on $X_{\tau}$ with values in $\Vect_K^{\ssolid}$. Moreover, if $X$ is an affinoid space, then we have that $H^i_{\tau}(X, \underline{\cl F_{\tau}})=0$ for all $i>0$.
   \item \label{condensedtate:2} For $A\in \Vect_K^{\ssolid}$ flat, the presheaf $\underline{\cl F_{\tau}}\solid_K A$ is a sheaf on $X_{\tau}$ with values in $\Vect_K^{\ssolid}$. Moreover, if $X$ is an affinoid space, then we have that $H^i_{\tau}(X, \underline{\cl F_{\tau}}\solid_K A)=0$ for all $i>0$.
  \end{enumerate}  
 \end{lemma}
 \begin{proof}
  First, we recall that, by Proposition \ref{corbello} (combined with Remark \ref{precision}), for any $K$-Banach space $V$, we have $\underline{V}\in \Vect_K^{\ssolid}$. Then, from Remark \ref{nnew} we deduce that $\underline{\cl F_\tau}$ and  $\underline{\cl F_{\tau}}\solid_K A$ take values in $\Vect_K^{\ssolid}$, since the abelian subcategory $\Vect_K^{\ssolid}\subset \Vect_K^{\cond}$ is stable under all limits and colimits, Proposition \ref{fevre}\listref{fevre:1}.
  
  For part \listref{condensedtate:1}, it suffices to show  the exactness of the \v{C}ech complex
  \begin{equation}\label{cco}
  0\to \underline{\cl F_\tau(U)}\to \prod_{i\in I}\underline{\cl F_\tau(U_i)}\to\prod_{i, j\in I}\underline{\cl F_\tau(U_i\times_U U_j)}\to \cdots
  \end{equation}
 for any $U\in \cl B_{\tau}^{\aff}$ and $\cl U=\{U_i\to U\}_{i\in I}$ a finite covering of $U$ in $\cl B_{\tau}^{\aff}$. By Tate's acyclicity theorem, and \cite[Proposition 3.2.5]{DeJong}, the \v{C}ech cohomology $\check{H}^i(\cl U, \cl F_\tau)$ vanishes for all $i>0$. Therefore, recalling Remark \ref{nnew}, the statement follows from Lemma \ref{acyclic}.
 
 For part \listref{condensedtate:2}, we observe that, since $A$ is a flat solid $K$-vector space, the complex (\ref{cco}) remains acyclic after applying the functor $-\solid_K A$ (and then we use that $-\solid_K A$ commutes with finite products).
 \end{proof}
  
  \begin{rem}\label{condvsolid}
  We keep the notation of Definition \ref{FsolidA}.
  \begin{enumerate}[(i)]
   \item\label{condvsolid:1} By Remark \ref{summary:2}, for all $i\ge 0$, we have $H^i_{\tau}(X, \underline{\cl F_\tau})(*)=H^i_{\tau}(X, \cl F_\tau)$ as $K$-vector spaces, and the condensed cohomology group $H^i_{\tau}(X, \underline{\cl F_\tau})$ does not depend on our choice of the cardinal $\kappa$ (here, we use Remark \ref{precision}, and the fact that $X$ is assumed to be $\kappa$-small).
   \item\label{condvsolid:2} We note that the sheaf cohomology complex $R\Gamma_{\tau}(X, \underline{\cl F_\tau})\in D(\Vect_K^{\cond})$ lies in $D(\Vect_K^{\ssolid})$. Equivalently, by Proposition \ref{fevre}\listref{fevre:1}, $H^i_{\tau}(X, \underline{\cl F_\tau})$ lies in $\Vect_K^{\ssolid}$ for all $i$. This is true by Verdier's hypercovering theorem, recalling that the subcategory $\Vect_K^{\ssolid}\subset \Vect_K^{\cond}$ is stable under all limits and colimits.
   \item\label{condvsolid:3} Let $A$ be a solid $K$-algebra. By Proposition \ref{solidA} (and Remark \ref{solidcut}), we can endow the condensed ring $A$ with an analytic ring structure $(A, \Zz)_{\solidif}$ such that $\Mod_A^{\ssolid}:=\Mod_{(A, \Zz)_{\solidif}}^{\cond}$ is the category of $A$-modules in $\Vect_K^{\ssolid}$. Then, assuming that $A$ is flat as a solid $K$-vector space, the argument of point \listref{condvsolid:2} shows that, since the sheaf $\underline{\cl F_\tau}\solid_K A$ (recall Lemma \ref{condensedtate}\listref{condensedtate:2}) has values in $\Mod_A^{\ssolid}$, the complex $R\Gamma_{\tau}(X, \underline{\cl F_\tau}\solid_K A)\in D(\Vect_K^{\cond})$ lies in $D(\Mod_A^{\ssolid})$.
  \end{enumerate}  
  The present remark also holds replacing $\cl F_\tau$ with a bounded below complex of such sheaves on $X_{\tau}$.
  \end{rem}

   In the condensed setting, we also have a version of Kiehl's acyclicity theorem for Stein spaces. First, let us recall the following definition.
  
  \begin{df}\label{stein}
   A rigid-analytic variety $X$ over $K$ is called a \textit{Stein space} if it has an increasing admissible affinoid covering $\{U_j\}_{j\in \Nn}$ such that $U_j\Subset U_{j+1}$, i.e. the inclusion $U_j\subset U_{j+1}$ factors over the adic compactification of $U_j$, for every $j\in \Nn$. We call $\{U_j\}_{j\in \Nn}$ a \textit{Stein covering} of $X$.
  \end{df}
 
  \begin{lemma}\label{A&B} 
  Let $X$ be a Stein space over $K$, and let $\cl F$ be a coherent $\cl O_{X}$-module over $X$. Then, $H^i(X, \underline{\cl F})=0$ for all $i>0$.
  \end{lemma}
  \begin{proof}
   By Lemma \ref{condensedtate}\listref{condensedtate:1}, choosing a countable admissible affinoid covering $\cl U$ of $X$, we have that $H^i(X, \underline{\cl F})=\check{H}^i(\cl U, \underline{\cl F})$. Then, the statement follows from Kiehl's acyclicity theorem, \cite[Satz 2.4]{Kiehl},\footnote{Note that \textit{loc. cit.} is stated for any quasi-Stein space (see \cite[Definition 2.3]{Kiehl}), and in particular it holds for Stein spaces (Definition \ref{stein}).} and Lemma \ref{acyclic} (recalling that a countable product of $K$-Fréchet spaces is a $K$-Fréchet space).
  \end{proof}

  Next, we give some examples by defining a structure of condensed $K$-vector space on the de Rham cohomology groups.
 
  Let $X$ be a smooth rigid-analytic variety over $K$. Denote by $\Omega_X^\bullet$ the de Rham complex of $X$, given by
  $$\Omega_X^\bullet:=[\cl O_X\overset{d}{\to}\Omega^1_X\overset{d}{\to}\cdots \overset{d}{\to}\Omega^m_X\overset{d}{\to}\cdots].$$
  \begin{df}\label{condDR}
   We define the complex of $D(\Vect_K^{\cond})$
   $$R\Gamma_{\dRc}(X):=R\Gamma(X, \underline{\Omega_X^\bullet})$$
   whose $i$-th cohomology, for $i\ge 0$, is the \textit{condensed de Rham cohomology group} $H^i_{\dRc}(X)$.
  \end{df}
  
   Now, we want to compare the definition above with some other natural ways of putting a structure of topological $K$-vector space on the de Rham cohomology groups. Let us start observing that if $H_{\dR}^i(X)$ is a finite-dimensional $K$-vector space, then it has a natural topology coming from the topology on $K$.
  
  \begin{lemma}\label{propclass}
   Let $X$ be a smooth proper rigid-analytic variety. Then, for all $i\ge 0$, we have
   $$H^i_{\dRc}(X)=\underline{H_{\dR}^i(X)}.$$
  \end{lemma}
  \begin{proof}
   Since $H_{\dR}^i(X)$ is a finite-dimensional $K$-vector space, \cite{Kiehl0}, and any $K$-linear map between finite-dimensional $K$-vector spaces is continuous, we have that it is isomorphic, as a topological $K$-vector space, to the \v{C}ech hypercohomology $\check{H}^i(\cl U, \Omega_X^\bullet)$ of a finite admissible affinoid covering $\cl U$ of $X$, endowed with the natural quotient topology. By Lemma \ref{condensedtate}\listref{condensedtate:1}, we have $H^i_{\dRc}(X)=\check{H}^i(\cl U, \underline{\Omega_X^\bullet})=\underline{\check{H}^i(\cl U, \Omega_X^\bullet)}$, where in the last step we used Lemma \ref{acyclic}.
  \end{proof}

  \begin{rem}\label{steinfrechet}
  If $X$ is a smooth Stein space, then the global sections $\Omega^i(X)$, $i\ge 0$, have a natural structure of $K$-Fr\'echet spaces: in fact, given a Stein covering $\{U_j\}_{j\in \Nn}$ of $X$, each space $\Omega^i(U_j)$ is $K$-Banach, therefore the spaces $\Omega^i(X)=\varprojlim_j \Omega^i(U_j)$ endowed with the inverse limit topology are $K$-Fr\'echet, and such structure does not depend on the choice of $\{U_j\}_{j\in \Nn}$.  By \cite[Corollary 3.2]{GK}, all the differentials $d: \Omega^{i-1}(X)\to \Omega^{i}(X)$ have closed image, hence, endowing $\Omega^i(X)^{d=0}\subseteq \Omega^i(X)$ with the subspace topology, and $H_{\dR}^i(X)=\Omega^i(X)^{d=0}/d\Omega^{i-1}(X)$ with the induced quotient topology, we obtain a $K$-Fr\'echet space structure on the latter.
 \end{rem}

  \begin{lemma}\label{drstein}
  Let $X$ be a smooth Stein space over $K$. 
  Then, for all $i\ge 0$, we have
  $$H^i_{\dRc}(X)=\underline{\Omega^i(X)^{d=0}/d\Omega^{i-1}(X)}.$$
  \end{lemma}
  \begin{proof}
  First note that, by definition, for all $i\ge0$, we have $\underline{\Omega^i_X}(X)=\lim_{U\in \cl B_{\an}^{\aff}}\underline{\Omega^i(U)}=\underline{\Omega^i(X)}$, with the structure of $K$-Fréchet space on $\Omega^i(X)$ as in Remark \ref{steinfrechet}. 
  
  Then, by Lemma \ref{A&B},  we have that $H_{\dRc}^i(X)=\underline{\Omega^i(X)}^{d=0}/d\underline{\Omega^{i-1}(X)}$ and the statement follows from Lemma \ref{acyclic}.\footnote{Note that $d \underline{\Omega^{i-1}(X)}=\underline{\Omega^{i-1}(X)/\Omega^{i-1}(X)^{d=0}}=\underline{d\Omega^{i-1}(X)}$, where in the first step we used Lemma \ref{acyclic}, and in the second one we used the open mapping theorem for $K$-Fréchet spaces.} 
  \end{proof}

  \begin{rem}\label{vsCN}
   Let $X$ be a smooth affinoid over $K$. Then, by Lemma \ref{condensedtate}\listref{condensedtate:1}, for all $i\ge 0$, we have that $H^i_{\dRc}(X)=\underline{\Omega^i(X)}^{d=0}/\underline{d\Omega^{i-1}(X)}$. But, in general, $H^i_{\dRc}(X)\neq \underline{\Omega^i(X)^{d=0}/d\Omega^{i-1}(X)}$: in fact, the subspace $d\Omega^{i-1}(X)\subseteq \Omega^i(X)^{d=0}$ can be non-closed,\footnote{For example, let $\Dd$ be the $1$-dimensional closed unit disk over $\Qq_p$, then, one can check that the subspace $d\cl O(\Dd)$ of $\Omega^1(\Dd)$ is not closed.} and then we can refer to Example \ref{card}. Indeed, the quotient topology on the de Rham cohomology groups is not the ``correct'' one in this case (cf. Remark \ref{vs}).
  \end{rem}

 \subsection{Paracompact rigid-analytic varieties}\label{paracsec}

 Before stating and proving the main result of this section, namely Theorem \ref{Wbasechange}, we first need to recall the notion of paracompact rigid-analytic variety, and give some examples.
 
 \begin{df}\label{parac}
  A (quasi-separated)\footnote{Recall Convention \ref{convqs}.} rigid-analytic variety $X$  over $K$ is \textit{paracompact} if it admits an admissible locally finite affinoid covering, i.e. there exists an admissible covering $\{U_i\}_{i\in I}$ of $X$ by affinoid subspaces such that for each index $i\in I$ the intersection $U_i\cap U_j$ is non-empty for at most finitely many indices $j\in I$.
 \end{df}
 
 \begin{rem}\label{paracompactvscountable}\
 \begin{enumerate}[(i)]
  \item\label{paracompactvscountable:1}  A paracompact rigid-analytic variety over $K$ is taut, \cite[Definition 5.1.2, Lemma 5.1.3]{Huberbook}, and it is the admissible disjoint union of connected paracompact rigid-analytic varieties of countable type, i.e. having a countable admissible affinoid covering, \cite[Lemma 2.5.7]{DeJong}. 
  \item Conversely, given $X$ a taut rigid-analytic variety over $K$ that is of countable type, then $X$ is paracompact. In order to see this, recall that there is an equivalence between the category of Hausdorff strictly $K$-analytic Berkovich spaces and the category of taut rigid-analytic varieties over $K$, \cite[Proposition 8.3.1]{Huberbook}, and, under such equivalence, Hausdorff, paracompact\footnote{Here, we mean \textit{paracompact} as a topological space.} strictly $K$-analytic Berkovich spaces correspond to paracompact rigid-analytic varieties over $K$, \cite[Theorem 1.6.1]{Berk}. Now, if $X^{\Berk}$ corresponds to $X$ under such equivalence, then $X^{\Berk}$ is a locally compact Hausdorff topological space that is a countable union of compact subspaces; we deduce that $X^{\Berk}$ is paracompact, and so $X$ is a paracompact rigid-analytic variety.
  \item  Let us also recall that there are examples of (non-taut) separated rigid-analytic varieties over $K$ of countable type that are not paracompact, \cite[Remarks 4.4]{LP}.
 \end{enumerate}
 \end{rem}

 \begin{rem}\label{openinpara}
  Assume $K/\Qq_p$ finite. Any admissible open of a paracompact rigid-analytic variety $X$ over $K$ is paracompact. In fact, the Hausdorff strictly $K$-analytic Berkovich space $X^{\Berk}$ associated to $X$ is metrizable, hence any of its subspaces is metrizable, in particular paracompact. In order to see that $X^{\Berk}$ is metrizable, we recall that, since $K$ has a countable dense subfield (namely, the algebraic closure of the field of rational numbers $\Qq$ in $K$), $X^{\Berk}$ is locally metrizable (see \cite[\S 2]{CL}), and a Hausdorff, paracompact, topological space that is locally metrizable is metrizable.\footnote{ We want to stress the importance of the assumptions on the base field $K$ in this remark: it would not hold for a base field with uncountable residue field, e.g. $\Qq_p(\!(t)\!)$, (see \cite[Proposition 4.3]{LP} for a counterexample).}
 \end{rem}

 \begin{examples}\label{list}
  A lot of interesting rigid-analytic varieties over $K$ are paracompact. Here is a list of examples: any separated rigid-analytic variety over $K$ of dimension 1, \cite{LP}; the rigid analytification of a separated scheme of finite type over $K$, \cite[\S 8.4, Proposition 7]{Bosch}; a Stein space over $K$; an admissible open of a quasi-compact (and quasi-separated) rigid-analytic variety over $K$, if $K/\Qq_p$ is finite (by Remark \ref{openinpara}).
 \end{examples}

 \begin{rem}\label{last?}
 Let $X$ be a paracompact rigid-analytic variety over $K$, and let $\cl F$ a coherent $\cl O_X$-module $\cl F$. If $X$ has finite dimension $n$, by \cite[Corollary 2.5.10]{DeJong} and Lemma \ref{condensedtate}\listref{condensedtate:1}, we have that $H^i(X, \underline{\cl F})=0$ for all $i>n$.
  \end{rem}

 \subsection{Solid base change} We are ready to state and prove the main result of this section.

 \begin{theorem}\label{Wbasechange}
  Let $X$ be a connected, paracompact, rigid-analytic variety defined over $K$. Let $\cl F^\bullet$ be a bounded below complex of sheaves of topological $K$-vector spaces whose terms are coherent $\cl O_X$-modules. Let $A$ be a $K$-Fréchet algebra, regarded as a condensed $K$-algebra. Then, we have a natural isomorphism in $D(\Vect_K^{\ssolid})$
  $$R\Gamma(X, \underline{\cl F}^\bullet)\dsolid_K A\overset{\sim}{\to} R\Gamma(X, \underline{\cl F}^\bullet \solid_K A).$$
 \end{theorem}

 \begin{proof}
  First of all, we note that, by Corollary \ref{frescoflat}, $A$ is flat as a solid $K$-vector space. Moreover, the natural map $R\Gamma(X, \underline{\cl F}^\bullet)\to R\Gamma(X, \underline{\cl F}^\bullet\solid_K A)$  induces a morphism
  \begin{equation}\label{derivedbasechange}
   R\Gamma(X, \underline{\cl F}^\bullet)\dsolid_K A\to R\Gamma(X, \underline{\cl F}^\bullet\solid_K A).
  \end{equation}
  in $D(\Vect_K^{\ssolid})$. In fact, by Remark \ref{condvsolid}\listref{condvsolid:3}, the complex $R\Gamma(X, \underline{\cl F}^\bullet\solid_K A)$ lies in $D(\Mod_A^{\ssolid})$.
  
  Now, we show that (\ref{derivedbasechange}) is a quasi-isomorphism. We suppose first that $X$ is affinoid, and consider the hypercohomology spectral sequence
  $$E_1^{i, j}=H^j(X, \underline{\cl F}^i\solid_K A)\implies H^{i+j}(X, \underline{\cl F}^\bullet\solid_K A).$$
  By Lemma \ref{condensedtate}\listref{condensedtate:2}, we have that $E_1^{i, j}=0$ for $j>0$. Hence, the spectral sequence degenerates and it gives the quasi-isomorphism 
  \begin{equation}\label{firsthalf}
   \underline{\cl F^\bullet(X)}\solid_K A\simeq R\Gamma(X, \underline{\cl F}^\bullet\solid_K A).
  \end{equation}
  Since $R\Gamma(X, \underline{\cl F}^\bullet)=\underline{\cl F^\bullet(X)}$, recalling that $A$ is a flat solid $K$-vector space, we have that
  \begin{equation}\label{secondhalf}
   R\Gamma(X, \underline{\cl F}^\bullet)\dsolid_K A=\underline{\cl F^\bullet(X)}\dsolid_K A=\underline{\cl F^\bullet(X)}\solid_K A.
  \end{equation}
  Thus, putting together (\ref{firsthalf}) and (\ref{secondhalf}), we have shown that (\ref{derivedbasechange}) is a quasi-isomorphism for $X$ affinoid.  
  For $X$ quasi-compact (and quasi-separated), we can reduce to the affinoid case by covering $X$ by a finite number of admissible affinoid subspaces. 
  In general, for $X$ as in the statement, by Remark \ref{paracompactvscountable}\listref{paracompactvscountable:1}, we can choose a quasi-compact admissible covering $\{U_n\}_{n\in \Nn}$ of $X$ such that $U_n\subseteq U_{n+1}$. By Corollary \ref{ids} and Theorem \ref{nuclearbanach}\listref{nuclearbanach:1},\footnote{We apply Theorem \ref{nuclearbanach}\listref{nuclearbanach:1} for $F=\Qq_p$ and $A=K$, using Corollary \ref{ova}.} each complex $R\Gamma(U_n, \underline{\cl F}^\bullet)$ is representable by a complex of nuclear $K$-vector spaces:\footnote{For $U$ an admissible open of $X$ we denote $R\Gamma(U, \underline{\cl F}^\bullet):=R\Gamma(U, \underline{\cl F}^\bullet|_U)$.} in fact, for $V\in \cl B_{\an}^{\aff}$, $R\Gamma(V, \underline{\cl F}^\bullet)$ is representable by a complex of $K$-Banach spaces; for $U$ a quasi-compact admissible open of $X$, which is the colimit of a finite, full and complete subcategory $\{V_i\}_{i\in \cl I}$ of $\cl B_{\an}^{\aff}$, we have that $R\Gamma(U, \underline{\cl F}^\bullet)=R\lim_{\cl I}R\Gamma(V_i, \underline{\cl F}^\bullet)$, and the claim follows from the Bousfield--Kan formula for the derived limits.\footnote{See \cite[Chapter XI]{BK}, or \cite[Appendix to Lecture VIII]{Scholzeanalytic} for a more recent reference.} Then, by Corollary \ref{commlim}, we have
  $$R\Gamma(X, \underline{\cl F}^\bullet)\dsolid_K A=R\varprojlim_n \left(R\Gamma(U_n, \underline{\cl F}^\bullet)\dsolid_K A\right)\simeq R\varprojlim_nR\Gamma(U_n, \underline{\cl F}^\bullet\solid_K A)=R\Gamma(X, \underline{\cl F}^\bullet\solid_K A)$$
  which is what we wanted. 
 \end{proof}

 \begin{cor}\label{fundcor}
  Under the hypotheses of Theorem \ref{Wbasechange},  we have the isomorphism in $\Vect_K^{\ssolid}$
  $$H^i(X, \underline{\cl F}^\bullet)\solid_K A\cong H^i(X, \underline{\cl F}^\bullet \solid_K A)$$
  for all $i\in \Zz$.
 \end{cor}
 \begin{proof}
  We want to show that, for all $i\in \Zz$, we have $H^i(R\Gamma(X, \underline{\cl F}^\bullet)\dsolid_K A)=H^i(X, \underline{\cl F}^\bullet)\solid_K A$. Considering the spectral sequence
  $$E_2^{j, i}=H^j(H^i(X, \underline{\cl F}^\bullet)\dsolid_K A)\implies H^{i+j}(R\Gamma(X, \underline{\cl F}^\bullet)\dsolid_K A).$$
  it suffices to note that $H^i(X, \underline{\cl F}^\bullet)\dsolid_K A$ is concentrated in degree 0, recalling that, by Corollary \ref{frescoflat}, $A$ is a flat solid $K$-vector space.
 \end{proof}

 \begin{rem}\label{vs}
  Let $X$ be a smooth affinoid over $K$. We note that, taking for example $\cl F^\bullet=\Omega_X^\bullet$ and $A=K$, Corollary \ref{fundcor} is trivially false, in general, if we work instead in the category of locally convex $K$-vector spaces, with the completed projective tensor product $\widehat \otimes_K$,\footnote{See Footnote \ref{projj} for a reminder about the definition of $\widehat \otimes_K$.} and put on $H^i_{\dR}(X)$ its natural locally convex quotient topology. In fact, $H^i_{\dR}(X)\widehat \otimes_K K$ is the Hausdorff completion of $H^i_{\dR}(X)$, by definition of $\widehat \otimes_K$, but $H^i_{\dR}(X)$ can be non-Hausdorff (cf. Remark \ref{vsCN}).
 \end{rem}

 \section{\textbf{Pro-\'{e}tale cohomology of $\Bb_{\dR}$ and $\Bb_{\dR}^+$}}\label{sectionBdr}
  \sectionmark{}

 In this section, we study the geometric pro-\'etale cohomology of $\Fil^r\Bb_{\dR}$ on any connected, paracompact, smooth rigid-analytic variety defined over $K$. As a corollary, we also show that the geometric pro-étale cohomology of $\Bb_{\dR}$ satisfies the axioms of Schneider--Stuhler (\ref{axioms}). \medskip

 Throughout this section, we maintain the notations and conventions of \S \ref{petsheaves} and \S \ref{coherent}. In particular, all rigid-analytic varieties will be assumed to be quasi-separated (Convention \ref{convqs}). \medskip

  \subsection{The comparison with the de Rham cohomology} 
  
  Let us begin by fixing the notation.
  
  \begin{notation}
   We denote $B_{\dR}^+=\Bb_{\dR}^+(C, \cl O_C)$, and we write $B_{\dR}=\Bb_{\dR}(C, \cl O_C)$ for Fontaine’s field of $p$-adic periods, which we equip with the filtration $\Fil^r B_{\dR}=t^rB_{\dR}^+$ for $r\in \Zz$, induced from the filtration of Definition \ref{periodsheaves}\listref{periodsheaves:last}. 
   If the context is clear we omit the underscores from the notation of the associated condensed rings.
  \end{notation}

  \begin{rem}\label{dRisind}
  Note that, by Lemma \ref{bspro} (and Example \ref{LFspace}), we have $\underline{B_{\dR}}=\varinjlim_{j\in \Nn}t^{-j}\underline{B_{\dR}^+}$ as condensed $K$-vector spaces, and each $t^{-j}B_{\dR}^+$ is a $K$-Fréchet space. 
  In particular, we deduce that $\underline{B_{\dR}}$ is a quasi-separated condensed $K$-vector space, being a filtered colimit along injections of quasi-separated condensed $K$-vector spaces (Remark \ref{ovoo}), and a flat solid $K$-vector space, being a filtered colimit of $K$-Fréchet spaces (Corollary \ref{frescoflat}).
 \end{rem}
 
 \begin{df} Let $X$ be a smooth rigid-analytic variety over $K$. Given a filtered $\cl O_X$-module with integrable connection $(\cl E, \nabla, \Fil^\bullet)$, \cite[Definition 7.4]{Scholze}, we denote by 
  $$\DR_X^{\cl E}:=[\cl E\overset{\nabla}{\to}\cl E\otimes_{\cl O_X}\Omega_X^1\overset{\nabla}{\to}\cdots \overset{\nabla}{\to}\cl E\otimes_{\cl O_X}\Omega_X^m\overset{\nabla}{\to}\cdots]$$
  its de Rham complex, which we equip with the filtration given by 
  $$\Fil^r\DR_X^{\cl E}:= [\Fil^r\cl E\overset{\nabla}{\to}\Fil^{r-1}\cl E\otimes_{\cl O_X}\Omega_X^1\overset{\nabla}{\to}\cdots \overset{\nabla}{\to}\Fil^{r-m}\cl E\otimes_{\cl O_X}\Omega_X^m\overset{\nabla}{\to}\cdots]$$
  for $r\in \Zz$.
  \begin{enumerate}[(i)]
   \item  We define the complex of $D(\Vect_K^{\cond})$
    $$R\Gamma_{\dRc}(X, \cl E):=R\Gamma(X, \underline{\DR_X^{\cl E}})$$
    whose $i$-th cohomology, for $i\ge 0$, is the \textit{condensed de Rham cohomology group $H^i_{\dRc}(X, \cl E)$ with coefficients in $\cl E$}.
    \item We define the complex of $D(\Vect_K^{\cond})$
   $$R\Gamma_{\dRc}(X_{B_{\dR}}, \cl E):=R\Gamma(X, \underline{\DR_X^{\cl E}}\solid_K B_{\dR})$$
   and we endow it with the filtration induced from the tensor product filtration.
  \end{enumerate}  
 \end{df}

 \begin{rem}
  By Remark \ref{condvsolid}, the complexes $R\Gamma_{\dRc}(X, \cl E)$ and $R\Gamma_{\dRc}(X_{B_{\dR}}, \cl E)$ lie in $D(\Vect_K^{\ssolid})$.
 \end{rem}
 
  We will prove the following theorem, which generalizes results of Scholze \cite[Theorem 7.11]{Scholze}, and Le Bras \cite[Proposition 3.17]{LeBras1}.
 
 \begin{theorem}\label{BDR}\label{BDR+}
   Let $X$ be a smooth rigid-analytic variety defined over $K$. Let $\Ll$ be a de Rham $\Qq_p$-local system on $X_{\pet}$, with associated  filtered $\cl O_X$-module with integrable connection $(\cl E, \nabla, \Fil^\bullet)$. Denote $\Mm_{\dR}:=\Ll\otimes_{\Qq_p}\Bb_{\dR}$. 
   \begin{enumerate}[(i)]
   \item\label{BDR:1}
   We have a $\underline{\mathscr{G}_K}$-equivariant, compatible with filtrations, natural isomorphism in $D(\Vect_K^{\ssolid})$
   $$ R\Gamma_{\petc}(X_C, \Mm_{\dR})\simeq R\Gamma_{\dRc}(X_{B_{\dR}}, \cl E).$$
   \item\label{BDR:2}  Assume that $X$ is connected and paracompact.
   Then, for each $r\in \Zz$, we have a $\underline{\mathscr{G}_K}$-equivariant isomorphism in $D(\Vect_K^{\ssolid})$
   $$R\Gamma_{\petc}(X_C, \Fil^r\Mm_{\dR})\simeq\Fil^r(R\Gamma_{\dRc}(X, \cl E)\dsolid_K B_{\dR}).$$ 
   \end{enumerate}
  \end{theorem}
  \medskip
  
  In order to recall the notion of being \textit{de Rham} for a $\Qq_p$-local system on $X_{\pet}$, we first need to recall the definition of the pro-étale period sheaf $\cl O \Bb_{\dR}$, following \cite{Scholze2} and \cite[Definition 2.2.10]{DLLZ}. \medskip
  
  In the following, we let $X$ be a smooth analytic adic space over $\Spa(K, \cl O_K)$.

 \begin{df}\label{precc} Let $\nu:X_{\pet}\to X_{\ett}$ denote the natural morphism of sites. We define the following sheaves on $X_{\pet}$.
 \begin{enumerate}[(i)]
  \item For $m\ge 1$, the sheaf of differentials $\Omega^m_X=\nu^*\Omega_{X_{\ett}}^m$.
  \item We define the sheaf $\cl O\Bb_{\dR}^+$ as the sheafification of the presheaf that to $Z^\diamondsuit \in X_{\pet}$, with $Z=\Spa(R_\infty, R_\infty^+)\to X$ a map, from an affinoid perfectoid space $Z$, that can be written as a cofiltered limit of étale maps $\Spa(R_i, R_i^+)\to X$, $i\in I$, along a $\kappa$-small index category $I$,\footnote{Recall that the set of all such $Z^\diamondsuit \in X_{\pet}$ forms a basis of $X_{\pet}$ (see e.g. \cite[Lemma 2.6]{MW}).} associates the following direct limit
  $$\varinjlim_i \varprojlim_n \left(( R_i^+\widehat\otimes_{W(k)}\Aa_{\inf}(R_\infty, R_\infty^+) )[1/p] \right)/(\ker \theta)^n.$$
  Here, $\widehat\otimes$ denotes the $p$-adic completion of the tensor product, and 
  $$\theta:(R_i^+\widehat\otimes_{W(k)}\Aa_{\inf}(R_\infty, R_\infty^+))[1/p]\to R_\infty$$
  is the tensor product of the maps $R_i^+\to R_\infty^+$ and $\theta: \Aa_{\inf}(R_\infty, R_\infty^+)\to R_\infty^+$. We define a filtration on $\cl O\Bb_{\dR}^+$ by setting $\Fil^r\cl O\Bb_{\dR}^+=(\ker\theta)^r\cl O\Bb_{\dR}^+$.
  \item Let $t$ be a generator of $\Fil^1 \Bb_{\dR}^+$. We define the sheaf $\cl O\Bb_{\dR}$ as the completion of the sheaf $\cl O\Bb_{\dR}^+[1/t]$ with respect to the filtration defined by $\Fil^r \cl (\cl O\Bb_{\dR}^+[1/t])=\sum_{j\in \Zz} t^{-j} \Fil^{r+j}\cl O\Bb_{\dR}^+$, and we equip $\cl O\Bb_{\dR}$ with the induced filtration.
 \end{enumerate}
 \end{df}
 
 \begin{rem}
  As observed in \cite[Remark 2.2.11]{DLLZ} the definition of $\cl O\Bb_{\dR}$ given above corrects the one of \cite{Scholze}, \cite{Scholze2}, as the latter is not complete with respect to the filtration.\footnote{We thank David Hansen for having brought this point to our attention.} The fact that $\cl O\Bb_{\dR}$ is complete with respect to its filtration will be crucial in the proof of Corollary \ref{frompoincare}.
 \end{rem}

 \begin{prop}[{\cite[Theorem 7.6]{Scholze}}]\label{7.6}
  The functor from filtered $\cl O_X$-modules with integrable connection to $\Bb_{\dR}^+$-local systems on $X_{\pet}$
  \begin{equation}\label{ffio}
  (\cl E, \nabla, \Fil^\bullet)\mapsto \Fil^0(\cl E\otimes_{\cl O_X}\cl O\Bb_{\dR})^{\nabla=0}
  \end{equation}
  is fully faithful.
 \end{prop}

 \begin{df}
  We say that a $\Qq_p$-local system $\Ll$ on $X_{\pet}$ is \textit{de Rham} if the $\Bb_{\dR}^+$-local system $\Ll\otimes_{\Qq_p}\Bb_{\dR}^+$ is associated to a filtered module with integrable connection $(\cl E, \nabla, \Fil^\bullet)$, i.e. it lies in the essential image of the functor (\ref{ffio}).
 \end{df}

 One of the main ingredients that we will use to prove Theorem \ref{BDR} is the following version of the Poincar\'{e} lemma, due to Scholze.
 
  \begin{prop}[{\cite[Corollary 6.13]{Scholze}}, {\cite[Corollary 2.4.2]{DLLZ}}]\label{poincare}
  Let $X$ be a smooth analytic adic space over $\Spa(K, \cl O_K)$ of dimension $n$. Then, we have an exact sequence of sheaves on $X_{\pet}$
  \begin{equation*}\label{poincexact}
   0\to \Bb_{\dR}^+\to \cl O\Bb_{\dR}^+\overset{\nabla}{\to}\cl O\Bb_{\dR}^+\otimes_{\cl O_X}\Omega_X^1\overset{\nabla}{\to}\cdots \overset{\nabla}{\to}\cl O\Bb_{\dR}^+\otimes_{\cl O_X}\Omega_X^n\to 0
  \end{equation*}
  where $\nabla: \cl O\Bb_{\dR}^+\to\cl O\Bb_{\dR}^+\otimes_{\cl O_X}\Omega_X^1$ is the unique $\Bb_{\dR}^+$-linear connection extending the differential $d: \cl O_X\to \Omega_X^1$.
  Moreover, we have an exact sequence as above replacing $\Bb_{\dR}^+$ with $\Bb_{\dR}$ and $\cl O\Bb_{\dR}^+$ with $\cl O\Bb_{\dR}$, and for $r\in \Zz$ we have compatible exact sequences of sheaves on $X_{\pet}$
  \begin{equation*}\label{poincexactgen}
   0\to \Fil^r\Bb_{\dR}\to \Fil^{r}\cl O\Bb_{\dR}\overset{\nabla}{\to}\Fil^{r-1}\cl O\Bb_{\dR}\otimes_{\cl O_X}\Omega_X^1\overset{\nabla}{\to}\cdots \overset{\nabla}{\to}\Fil^{r-n}\cl O\Bb_{\dR}\otimes_{\cl O_X}\Omega_X^n\to 0.
  \end{equation*}
 \end{prop}
 
  We will also need the following lemma.
 
 \begin{lemma}\label{grcond}
  Let $X=\Spa(R, R^+)$ be an affinoid adic space of finite type over $\Spa(K, \cl O_K)$ with an \'etale map $$X\to \Tt_K^n:=\Spa(K\langle T_1^{\pm 1}, \ldots, T_n^{\pm 1}\rangle, \cl O_K\langle T_1^{\pm 1}, \ldots, T_n^{\pm 1}\rangle)$$ that can be written as a composite of of finite \'etale maps and rational embeddings. Let $S\in *_{\kappa{\text -}\pet}$. Then, for any $j\in \Zz$, we have
  $$H_{\pet}^i(X_C\times S, \gr^j\cl O\Bb_{\dR})=
   \begin{cases} \mathscr{C}^0\left(S, R\widehat\otimes_K C(j)\right)& \mbox{if }i=0 \\ 0 & \mbox{if }i>0 \end{cases}$$
   where $(j)$ denotes a Tate twist.\footnote{See Footnote \ref{projj} for the definition of completed projective tensor product $\widehat \otimes_K$.}
 \end{lemma}
 \begin{proof}
  This is \cite[Proposition 6.16]{Scholze} in the case $S=*$. For the general case, it suffices to slightly expand the argument of \textit{loc. cit.} as follows. By twisting, we can suppose $j=0$. Denote by $$\widetilde \Tt_C^n:=\Spa(C\langle T_1^{\pm 1/p^\infty}, \ldots, T_n^{\pm 1/p^\infty}\rangle, \cl O_C\langle T_1^{\pm 1/p^\infty}, \ldots, T_n^{\pm 1/p^\infty}\rangle)$$
  the affinoid perfectoid $n$-torus over $C$, and let $\widetilde X_C:=X_C\times_{\Tt_C^n}\widetilde\Tt_C^n=\Spa(R_{\infty}, R_{\infty}^+)$. We recall that $\widetilde \Tt_C^n\to\Tt_C^n$ is a pro-(finite étale) $\Zz_p(1)^n$-cover. 
  Then, since a version of Proposition \ref{affinoidsections} and Corollary \ref{profinite} holds true for the pro-étale sheaf $\gr^0 \cl O\Bb_{\dR}$, considering the Cartan--Leray spectral sequence associated to the affinoid perfectoid $\Zz_p(1)^n$-cover $\widetilde X_C\times S\to X_C\times S$, we have, for $i\ge 0$,
  $$H^i_{\pet}(X_C\times S, \gr^0 \cl O\Bb_{\dR})\cong H^i_{\cont}(\Zz_p^n, \gr^0\cl O\Bb_{\dR}(\widetilde X_C\times S)).$$
  We note that $\widetilde X_C\times S=\Spa(\mathscr{C}^0(S, R_{\infty}), \mathscr{C}^0(S, R_{\infty}^+))$. Now, the proof of \textit{loc. cit.} shows that $H_{\pet}^i(X_C\times S, \gr^0\cl O\Bb_{\dR})=0$ if $i>0$, and it is isomorphic to $\mathscr{C}^0(S, R)\widehat\otimes_K C$ if $i=0$.\footnote{Here, $\mathscr{C}^0(S, R)$ is endowed with the sup-norm.} Then, we obtain the statement observing that $\mathscr{C}^0(S, R)\widehat\otimes_K C=\mathscr{C}^0(S, R\widehat\otimes_K C)$ (which follows from \cite[Corollary 10.5.4]{Garcia}).
 \end{proof}

 In the following statement, we will keep using the notations introduced in \S \ref{condcohgroups} (in particular, see Remark \ref{cohcond}).
 Moreover, given $X$ a rigid-analytic variety defined over $K$, $\cl F$ an $\cl O_{X_{\ett}}$-module over $X_{\ett}$ that is locally free of finite rank, and $A$ a $\overline{K}$-algebra such that $\underline{A}$ is a flat solid $K$-vector space,\footnote{Our main cases of interest will be $A\in\{t^mB_{\dR}^+/t^nB_{\dR}^+:-\infty\le m<n\le \infty\}$.} in addition to Definition \ref{FsolidA}, by abuse of notation we will also denote by $\underline{\cl F}\solid_K \underline{A}$ the sheaf with values in $\Vect_K^{\ssolid}\subset \Vect_K^{\cond}$ regarded on $X_{C, \ett}$ via the equivalence of topoi\footnote{As observed in \cite[\S 3.2]{LeBras1}, this equivalence follows from Elkik's approximation theorem, \cite{Elkik}. Compare with \cite[Lemma 2.5]{LZ}.}
 \begin{equation}\label{eqtop}
  X_{C, \ett}^\sim\cong \varprojlim_{K'/K \text{ finite}} X_{K', \ett}^\sim.
 \end{equation}
 
 \begin{cor}\label{frompoincare}
  Let $X$ be a smooth rigid-analytic variety over $K$. Let $(\cl E, \nabla, \Fil^\bullet)$ be a filtered $\cl O_X$-module with integrable connection, with associated $\Bb_{\dR}^+$-local system 
  $$\Mm_{\dR}^+:=\Fil^0(\cl E\otimes_{\cl O_X}\cl O \Bb_{\dR})^{\nabla=0}$$ 
  and let $\Mm_{\dR}:=\Mm_{\dR}^+[1/t]$. Let us denote by $\lambda: X_{\pet}/X_C\cong X_{C, \pet} \to X_{C, \ett, \cond}$ the natural morphism of sites. Then, we have a natural quasi-isomorphism of complexes of sheaves on $X_{C, \ett}$ with values in $\Vect_K^{\cond}$ which is compatible with filtrations
  \begin{equation}\label{afterpoincare0}  
   (R\lambda_*\Mm_{\dR})^{\blacktriangledown}\simeq\underline{\DR_X^{\cl E}} \solid_K \underline{B_{\dR}}
  \end{equation}
  where the right-hand side is endowed with the tensor product filtration.
 \end{cor}
 
 \begin{proof}
  We may assume that $X$ is connected of dimension $n$. Then, from Proposition \ref{poincare}, we have an exact sequence of sheaves on $X_{C, \pet}\cong X_{\pet}/X_C$
 $$0\to\Mm_{\dR}\to \cl E\otimes_{\cl O_X}\cl O \Bb_{\dR}\overset{\nabla}{\to}\cl E\otimes \Omega_X^1\otimes_{\cl O_X}\cl O \Bb_{\dR}\overset{\nabla}{\to}\cdots \overset{\nabla}{\to}\cl E\otimes \Omega_X^n\otimes_{\cl O_X}\cl O \Bb_{\dR}\to 0$$
  which remains exact after taking $\Fil^r$.
  In particular, this induces a quasi-isomorphism between $(R\lambda_*\Mm_{\dR})^{\blacktriangledown}$ and the complex\footnote{$\PN$ stands for ``Poincaré''.} 
  $$\PN_X^{\cl E}:=R\lambda_*\left[\cl E\otimes_{\cl O_X}\cl O \Bb_{\dR}\to\cl E\otimes \Omega_X^1\otimes_{\cl O_X}\cl O \Bb_{\dR}\to\cdots\to\cl E\otimes \Omega_X^n\otimes_{\cl O_X}\cl O \Bb_{\dR}\right]^{\blacktriangledown}$$ 
  which is compatible with the natural filtrations.
  Now, we construct a natural morphism 
  \begin{equation}\label{natural}
   \underline{\DR_X^{\cl E}} \solid_K \underline{B_{\dR}}\to \PN_X^{\cl E}
  \end{equation}
  of filtered complexes of sheaves on $X_{C, \ett}$, which we claim to be a quasi-isomorphism. We note that it suffices to exhibit a natural morphism $\underline{\cl O_X}\solid_K \underline{B_{\dR}}\to (\lambda_*\cl O\Bb_{\dR})^{\blacktriangledown}$ of sheaves on $X_{C, \ett}$ (with values in $\Vect_K^{\cond}$) that is compatible with filtrations. Let $S\in *_{\kappa{\text -}\pet}$ be an extremally disconnected set, let $K'/K$ be a finite extension with residue field $k'$, and let $V\to X_{K'}$ be an étale morphism with $V=\Spa(R, R^+)$ affinoid. Moreover, let $\widetilde V_C=\Spa(R_{\infty}, R_{\infty}^+)\to V_C$ be a map from an affinoid perfectoid space, that can be written as a cofiltered limit of étale maps $\Spa(R_i, R_i^+)\to V_C$, $i\in I$, along a $\kappa$-small index category $I$. We observe that
  \begin{align*} 
  (\underline{\cl O_X}\solid_K \underline{B_{\dR}})(V)(S)&=\varinjlim\nolimits_{j\in \Nn}(\underline{R}\solid_{K'} t^{-j}\underline{B_{\dR}^+})(S)\\
  &=\varinjlim\nolimits_{j\in \Nn}\mathscr{C}^0(S, R\widehat \otimes_{K'} t^{-j}B_{\dR}^+)\\
  &=\varinjlim\nolimits_{j\in \Nn}\mathscr{C}^0(S, R)\widehat \otimes_{K'} t^{-j}B_{\dR}^+\\
  &=\left(\varprojlim\nolimits_{n\in \Nn}\left((\mathscr{C}^0(S,  R^+)\widehat\otimes_{W(k')} A_{\inf})[1/p]\right)/\xi^n\right)[1/t]
  \end{align*}
  where in the first step we used Remark \ref{dRisind} together with the fact that the tensor product $\solid_{K'}$ commutes with colimits, in the second one we used that filtered colimits of condensed $K$-vector spaces can be computed pointwise on extremally disconnected sets, and Proposition \ref{solidvsproj}, and, finally, in the third step we used \cite[Corollary 10.5.4]{Garcia}. Therefore,  recalling Definition \ref{precc}, we deduce that we have a natural map $$(\underline{\cl O_X}\solid_K \underline{B_{\dR}})(V)(S)\to \cl O\Bb_{\dR}(\widetilde V_C\times S)$$ since $\widetilde V_C\times S=\Spa(\mathscr{C}^0(S, R_{\infty}), \mathscr{C}^0(S, R_{\infty}^+))$. Then, one checks that the latter map descends to a natural map with target $\cl O\Bb_{\dR}(V_C\times S)$, and that it gives the desired morphism of sheaves.
  
  Now, observing that the filtration on both $\underline{\DR_X^{\cl E}} \solid_K \underline{B_{\dR}}$ and $\PN_X^{\cl E}$ is exhaustive and complete,\footnote{The filtration on $\underline{\DR_X^{\cl E}} \solid_K \underline{B_{\dR}}$ is complete since the filtration on $B_{\dR}$ has such property: in fact, in order to prove the vanishing of the derived limit $R\varprojlim_r \Fil^r(\underline{\DR_X^{\cl E}} \solid_K \underline{B_{\dR}})$, arguing degreewise (\cite[Tag 015E]{Thestack}),  we can use Corollary \ref{commlim}, recalling that $B_{\dR}$ is filtered by $K$-Fréchet spaces, which are nuclear $K$-vector spaces by Corollary \ref{frechetnuc}.} by \cite[Lemma 5.2, (1)]{BMS2} it suffices to show that the natural morphism (\ref{natural}) is a quasi-isomorphism on graded pieces, i.e. that, for all $j\in \Zz$, the induced morphism $$\gr^j (\underline{\DR_X^{\cl E}} \solid_K \underline{B_{\dR}}) \to \gr^j\PN_X^{\cl E}$$ is a quasi-isomorphism. Since the morphism (\ref{natural}) is also compatible with respect to the naive filtration on both sides, it suffices to show that, for any $\cl O_{X_{\ett}}$-module $\cl F$ over $X_{\ett}$ that is locally free of finite rank, and for all $j\in \Zz$, the natural morphism
  \begin{equation}\label{local}
   \underline{\cl F}\solid_K\gr^j \underline{B_{\dR}}\overset{\sim}{\to}R\lambda_*(\nu^* \cl F\otimes_{\cl O_X}\gr^{j}\cl O\Bb_{\dR})^{\blacktriangledown}
  \end{equation}
  is a quasi-isomorphism of complexes of sheaves on $X_{C, \ett}$. Let us consider the commutative diagram of morphisms of sites
 \begin{center}
   \begin{tikzcd}
   X_{C, \pet} \arrow[r, "\lambda"] \arrow[d, "\varepsilon'"'] 
     & X_{C, \ett, \cond} \arrow[r, "\pi"]
     & X_{C, \ett} \arrow[d, "\varepsilon"]
    \\ 
    X_{\pet} \arrow[rr, "\nu"] 
     &
     & X_{\ett}    
  \end{tikzcd}
  \end{center}
 where $\pi$ is the natural projection and $\varepsilon$, $\varepsilon'$ are the base change morphisms. Let $\upsilon:=\varepsilon\circ \pi$, then, by the projection formula, for all $j\in \Zz$, we have\footnote{Here, the sheaf $\nu^* \cl F\otimes_{\cl O_X}\gr^{j}\cl O\Bb_{\dR}$ is regarded on $X_{\pet}/X_C\cong X_{C, \pet}$, hence it should be read as the sheaf $\varepsilon'^*(\nu^* \cl F\otimes_{\cl O_X}\gr^{j}\cl O\Bb_{\dR})$.}
  $$R\lambda_*(\nu^* \cl F\otimes_{\cl O_X}\gr^{j}\cl O\Bb_{\dR})=\upsilon^*\cl F\otimes_{\upsilon^*\cl O_X } R\lambda_*(\gr^{j}\cl O\Bb_{\dR}).$$
  Now, note that we can check (\ref{local}) locally on $X_{C, \ett}$.  Since $X$ is smooth, by \cite[Lemma 5.2]{Scholze}, we can assume that $X$ is affinoid and there exists an étale map $X\to \Tt_K^n$ that can be written as a composite of finite \'etale maps and rational embeddings. By Lemma \ref{grcond} and Proposition \ref{solidvsproj}, we have, for all $j\in \Zz$, $R\lambda_*(\gr^j \cl O\Bb_{\dR})^{\blacktriangledown}=\underline{\cl O_X}\solid_K \underline{C(j)} $.  By twisting we can assume $j=0$, hence, it remains to show that we have $$(\upsilon^*\cl F)^{\blacktriangledown}\otimes_{(\upsilon^*\cl O_X)^{\blacktriangledown}}(\underline{\cl O_X}\solid_K \underline{C})=\underline{\cl F}\solid_K \underline{C}$$ as sheaves on $X_{C, \ett}$.\footnote{Note that, by Lemma \ref{acyclic}, for all $j\in \Zz$, we have $\underline{C(j)}=t^{j}\underline{B_{\dR}^+}/t^{j+1}\underline{B_{\dR}^+}=\gr^j\underline{B_{\dR}}$.} For this,  we first observe that, by Proposition \ref{solidvsproj}, we have $\underline{\cl O_X}\solid_K \underline{C}=\underline{\cl O_{X_C}}$ and $\underline{\cl F}\solid_K \underline{C}=\underline{\cl F_C}$. Then, let $S\in *_{\kappa{\text -}\pet}$ be an extremally disconnected set, let $K'/K$ be a finite extension, let $V\to X_{K'}$ be an étale morphism with $V=\Spa(R, R^+)$ affinoid, and let $M$ be the projective module of finite rank over $R$  corresponding to $\cl F_{K'}|_V$ (endowed with its natural $K'$-Banach space structure). Then, 
  $$((\upsilon^*\cl F)^{\blacktriangledown}\otimes_{(\upsilon^*\cl O_X)^{\blacktriangledown}}\underline{\cl O_{X_C}})(V_C)(S)=\varinjlim_{L/K'\text{ finite}}(M_L\otimes_{R_L}\mathscr{C}^0(S, R\widehat\otimes_{K'} C))=\mathscr{C}^0(S, M\widehat\otimes_{K'} C)=(\underline{\cl F_C})(V_C)(S)$$
  which is what we wanted.  
 \end{proof}

 Before proving Theorem \ref{BDR}, we need another preliminary result (cf. \cite[Lemma 3.1]{LZ}).
 
 \begin{lemma}\label{bc}
   Let $X$ be an affinoid space defined over $K$, and let $\cl F$ be an $\cl O_{X_{\ett}}$-module over $X_{\ett}$ that is locally free of finite rank. Then, for $A\in\{t^mB_{\dR}^+/t^nB_{\dR}^+:-\infty\le m<n\le \infty\}$ regarded in $\Mod_{K}^{\ssolid}$, we have
   $$H_{\ett}^i(X_C, \underline{\cl F}\solid_K A)=
   \begin{cases} \underline{\cl F(X)}\solid_K A & \mbox{if }i=0 \\ 0 & \mbox{if }i>0. \end{cases}$$
 \end{lemma}
 \begin{proof}
  By twisting, it suffices to prove the statement for $A\in\{B_{\dR}, B_{\dR}^+, B_{\dR}^+/t^j:j\ge 1\}$. 
  
  Let us start from $A=B_{\dR}^+/t=C$. By Proposition \ref{solidvsproj}, we have $\underline{\cl F}\solid_K C=\underline{\cl F_C}$ as sheaves on $X_{C, \ett}$, and the result simply follows from the acyclicity of $\underline{\cl F_C}$ (cf. Lemma \ref{condensedtate}).
  In order to prove the statement for $A=B_{\dR}^+/t^j$, $j\ge 1$, we can proceed by induction on $j$, using the following exact sequence\footnote{The exactness of this sequence can be checked using the flatness of $K$-Banach spaces for the tensor product $\solid_K$, Corollary \ref{ids}.} of sheaves on $X_{C, \ett}$
  $$0\to \underline{\cl F}\solid_K C(j)\to \underline{\cl F}\solid_K(B_{\dR}^+/t^{j+1})\to \underline{\cl F}\solid_K(B_{\dR}^+/t^j)\to 0.$$
  For $A=B_{\dR}^+$, observing that $B_{\dR}^+=\varprojlim_{j\in \Nn}B_{\dR}^+/t^j$, we deduce the statement from \cite[Lemma 3.18]{Scholze}.\footnote{We note that \textit{loc. cit.} is stated for sheaves on a site with values in abelian groups, however it holds also for sheaves with values in $\CondAb$.}
  Finally, for $A=B_{\dR}$, it suffices to recall that $B_{\dR}=\varinjlim_{j\in \Nn}t^{-j}B_{\dR}^+$ and use that, since $|X_C|$ is quasi-compact and quasi-separated, the cohomology commutes with direct limit of sheaves on $X_{C, \ett}$ with values in $\CondAb$.
 \end{proof}

  \begin{proof}[Proof of Theorem \ref{BDR}] The natural morphism of sites $\lambda: X_{C, \pet} \to X_{C, \ett, \cond}$ factors as $$X_{C, \pet}\overset{\mu}{\to} X_{C, \pet, \cond}\to X_{C, \ett, \cond}.$$ 
  Then, by Lemma \ref{pet=petcond}, we have
  $$R\Gamma_{\petc}(X_C, \Mm_{\dR})\simeq R\Gamma_{\pet, \cond}(X_C, R\mu_*\Mm_{\dR})\simeq R\Gamma_{\ett, \cond}(X_C, R\lambda_* \Mm_{\dR})= R\Gamma_{\ett}(X_C, (R\lambda_* \Mm_{\dR})^{\blacktriangledown}).$$
  Let $\varepsilon: X_{C, \ett}\to X_{\ett}$ be the base change morphism. By Corollary \ref{frompoincare}, Lemma \ref{bc}, and Lemma \ref{condensedtate}\listref{condensedtate:2}, we have a natural quasi-isomorphism which is compatible with filtrations
  $$
  R\Gamma_{\petc}(X_C, \Mm_{\dR})\simeq R\Gamma_{\ett}(X, R\varepsilon_*(R\lambda_* \Mm_{\dR})^{\blacktriangledown})\simeq R\Gamma_{\ett}(X, \underline{\DR_X^{\cl E}}\solid_K B_{\dR})\simeq R\Gamma_{\an}(X, \underline{\DR_X^{\cl E}}\solid_K B_{\dR}).
  $$
  This finishes the proof of part \listref{BDR:1}. For part \listref{BDR:2}, we assume that $X$ is connected and paracompact. It remains to show that, for each $r\in \Zz$, the natural morphism
  $$\Fil^r(R\Gamma_{\dRc}(X, \cl E)\dsolid_K B_{\dR})\to R\Gamma(X, \Fil^r(\underline{\DR_X^{\cl E}}\solid_K B_{\dR}))$$
  is a quasi-isomorphism. This follows from Theorem \ref{Wbasechange}, observing that the complex $\DR_X^{\cl E}$ is bounded (since $X$ has finite dimension), and recalling that $\Fil^jB_{\dR}=t^jB_{\dR}^+$ for $j\in \Zz$ with $B_{\dR}^+$ a $K$-Fréchet algebra.
  \end{proof}
  
  Let us make some comments on Theorem \ref{BDR} and deduce some corollaries.
  
   \begin{rem}\label{quasi-comp}
   Let $X$ be a quasi-compact (and quasi-separated) smooth rigid-analytic variety defined over $K$. Retaining the notation of Theorem \ref{BDR}, we have the following $\underline{\mathscr{G}_K}$-equivariant isomorphism in $D(\Vect_K^{\ssolid})$
   \begin{align*}
    R\Gamma_{\petc}(X_C, \Mm_{\dR})&=\colim_{j\in \Nn}R\Gamma_{\petc}(X_C, \Fil^{-j}\Mm_{\dR})\\
    &\simeq\colim_{j\in \Nn}\Fil^{-j}(R\Gamma_{\dRc}(X, \cl E)\dsolid_K B_{\dR}) \\
    &=R\Gamma_{\dRc}(X, \cl E)\dsolid_K B_{\dR}
   \end{align*}
   where in the first step we used that $|X_C|$ is quasi-compact and quasi-separated, in the second one  Theorem \ref{BDR}\listref{BDR:2}, and in the third step we used that $\dsolid_K$ commutes with colimits, together with the fact that $\Vect_K^{\ssolid}$ satisfies Grothendieck's axiom (AB5).
   In particular, in this case, for all $i\ge 0$, we have a $\underline{\mathscr{G}_K}$-equivariant isomorphism in $\Vect_K^{\ssolid}$
   \begin{equation}\label{afterkunneth}
    H_{\petc}^i(X_C, \Mm_{\dR})\cong H_{\dRc}^i(X, \cl E)\solid_K B_{\dR}.
   \end{equation}
   as it follows recalling that, by Remark \ref{dRisind}, $B_{\dR}$ is a flat solid $K$-vector space (see also the proof of Corollary \ref{fundcor}).
  \end{rem}
  
  \begin{rem}\label{prop}
    Let $X$ be a smooth proper rigid-analytic variety over $K$. For any filtered $\cl O_X$-module with integrable connection $(\cl E, \nabla, \Fil^\bullet)$, by the finiteness of coherent cohomology of proper rigid-analytic varieties, \cite{Kiehl0}, the same argument of Lemma \ref{propclass} shows that $H_{\dRc}^i(X, \cl E)=\underline{H_{\dR}^i(X, \cl E)}$. Then, in this case, the right-hand side of (\ref{afterkunneth}) is given by $\underline{H_{\dR}^i(X, \cl E)\otimes_K B_{\dR}}$. Hence, the isomorphism (\ref{afterkunneth}) recovers \cite[Theorem 7.11]{Scholze}.
  \end{rem}

  \begin{rem}\label{affstein}
   Let $X$ be a smooth affinoid or Stein space over $K$. By Theorem \ref{BDR}\listref{BDR:2} for $\cl E=\cl O_X$, recalling that $K$-Fréchet spaces are flat solid $K$-vector spaces (Corollary \ref{frescoflat}), for $i\ge 0$, we have that
  $$H^i_{\petc}(X_C, \Bb_{\dR}^+)= \colim\nolimits_j \left(H^i(X, \underline{\Omega_X^{\ge j}})\solid_K t^{-j}B_{\dR}^+\right).$$
  We note that $H^i(X, \underline{\Omega_X^{\ge j}})$ vanishes for $i<j$, moreover, $H^i(X, \underline{\Omega_X^{\ge i}})=\underline{\Omega^i(X)}^{d=0}$ and, by Lemma \ref{condensedtate} and Lemma \ref{A&B}, respectively, we have $H^i(X, \underline{\Omega_X^{\ge j}})=H^i_{\dRc}(X)$ for $i>j$. Therefore, we have the following $\underline{\mathscr{G}_K}$-equivariant exact sequence in $\Vect_K^{\ssolid}$
  \begin{equation}\label{explB_dR^+}
    0\to H^i_{\dRc}(X)\solid_K t^{-i+1}B_{\dR}^+ \to H^i_{\petc}(X_C, \Bb_{\dR}^+)\to \underline{\Omega^i(X)}^{d=0}\solid_K C(-i)\to 0.
  \end{equation}
  In the case when $X$ is a smooth Stein space,  by Lemma \ref{drstein}, we have that $H^i_{\dRc}(X)=\underline{H^i_{\dR}(X)}$,\footnote{Let us remark that we don't know whether, in the smooth Stein case and with non-trivial coefficients, the condensed de Rham cohomology group $H_{\dRc}^i(X, \cl E)$ comes from a topological $K$-vector space.} where $H^i_{\dR}(X)$  is endowed with its natural structure of $K$-Fréchet space as in Remark \ref{steinfrechet}. Then, by Proposition \ref{solidvsproj}, the exact sequence (\ref{explB_dR^+}) can be stated in classical topological terms, with the completed projective tensor product replacing the solid tensor product.
  \end{rem}

  Using Theorem \ref{BDR+}, one can express the geometric pro-étale cohomology with coefficients in $\Bb_{\dR}/\Bb_{\dR}^+$, appearing in the fundamental exact sequence of $p$-adic Hodge theory (\ref{fundexact}), in terms of differential forms. In the following special cases, this takes a particularly nice form, which also illustrates how, for a smooth rigid-analytic variety $X$ over $K$, the (non-)degeneration of the Hodge-de Rham spectral sequence is reflected in its geometric $p$-adic pro-étale cohomology.

 \begin{cor}\label{B/B+}
  Let $X$ be a smooth rigid-analytic variety over $K$. Let $i\ge 0$.
  \begin{enumerate}[(i)]
  \item \label{B/B+:1} If $X$ is proper,  we have a $\underline{\mathscr{G}_K}$-equivariant isomorphism in $\Vect_K^{\ssolid}$ 
   $$H^i_{\petc}(X_C, \Bb_{\dR}/\Bb_{\dR}^+)\cong(H^i_{\dRc}(X)\otimes_K B_{\dR})/\Fil^0.$$
  \item \label{B/B+:2} If $X$ is an affinoid space,  we have the following $\underline{\mathscr{G}_K}$-equivariant exact sequence in $\Vect_K^{\ssolid}$
    $$0\to H^i_{\dRc}(X)\solid_K B_{\dR}/t^{-i}B_{\dR}^+ \to H^i_{\petc}(X_C, \Bb_{\dR}/\Bb_{\dR}^+)\to \underline{\Omega^i(X)/\ker d}\solid_K C(-i-1)\to 0.$$
  \end{enumerate} 
 \end{cor}
  
 \begin{proof}
  Let $f:X_{C}\to \Spa(C, \cl O_C)$ denote the structure morphism, and let us consider the long exact sequence associated to the distinguished triangle $Rf_{\pet *}\Bb_{\dR}^+\to Rf_{\pet *}\Bb_{\dR}\to Rf_{\pet*}\Bb_{\dR}/\Bb_{\dR}^+$, i.e.
  $$\cdots\to H^i_{\petc}(X_C, \Bb_{\dR}^+)\overset{\alpha_i}\to H^i_{\petc}(X_C, \Bb_{\dR})\to H^i_{\petc}(X_C, \Bb_{\dR}/\Bb_{\dR}^+)\to \cdots$$
  which gives the short exact sequence
  \begin{equation}\label{ses}
   0\to \coker \alpha_i\to H^i_{\petc}(X_C, \Bb_{\dR}/\Bb_{\dR}^+)\to \ker\alpha_{i+1}\to 0.
  \end{equation}
  
  If $X$ is proper, recalling Remark \ref{prop}, by the degeneration of the Hodge-de Rham spectral sequence at the first page, \cite[Corollary 1.8]{Scholze}, we have that $\ker \alpha_{i}=0$, hence part \listref{B/B+:1}.
  
  If $X$ is an affinoid space, by Remark \ref{quasi-comp} and Remark \ref{affstein}, we have the following commutative diagram with exact rows
   \begin{center}
   \begin{tikzcd}
   0 \arrow[r] 
     &[-1em] H^i_{\dRc}(X)\solid_K t^{-i+1}B_{\dR}^+ \arrow[-,double line with arrow={-,-}]{d} \arrow[r] 
     &[-1em] H^i_{\petc}(X_C, \Bb_{\dR}^+) \arrow[d, "\alpha_{i}"] \arrow[r]      
     &[-1em] \underline{\Omega^i(X)}^{d=0}\solid_K C(-i) \arrow[d, "\beta_{i}"] \arrow[r] 
     &[-1em] 0  \\
   0 \arrow[r] 
     &[-1em] H^i_{\dRc}(X)\solid_K t^{-i+1}B_{\dR}^+ \arrow[r]           
     &[-1em] H^i_{\dRc}(X)\solid_K B_{\dR} \arrow[r]  
     &[-1em] H^i_{\dRc}(X)\solid_K B_{\dR}/t^{-i+1}B_{\dR}^+ \arrow[r]    
     &[-1em] 0
  \end{tikzcd}
  \end{center}
  where the top right arrow is given by the quotient by $t^{-i+1}B_{\dR}^+$, and the morphism $\beta_i$ is given by the natural projection $\underline{\Omega^i(X)}^{d=0}\to H^i_{\dRc}(X)$ on the first factor and the inclusion on the second one. By the snake lemma, we have
  $$\ker \alpha_i=\ker\beta_i=\underline{d\Omega^{i-1}(X)}\solid_K C(-i),$$
  $$\coker \alpha_i=\coker \beta_i=H^i_{\dRc}(X)\solid_K B_{\dR}/t^{-i}B_{\dR}^+$$
  that, together with the short exact sequence (\ref{ses}), imply part \listref{B/B+:2}.
  \end{proof}
 
  From Theorem \ref{BDR+} we can also deduce the following result.
  
  \begin{prop}\label{BdRSS}
   The cohomology theories 
   $$H_{\underline{\dR}}^\bullet:\RigSm_{K, \an}\to \Vect_{K}^{\cond},\;\;\;\;\;\;\;\;\; H_{\petc}^\bullet(-_C, \Bb_{\dR}):\RigSm_{K, \pet}\to \Mod_{B_{\dR}}^{\cond}$$ 
   satisfy the axioms of Schneider--Stuhler (\ref{axioms}).
  \end{prop}
  \begin{proof}
   The condensed de Rham cohomology $H_{\underline{\dR}}^\bullet$ is defined by the complex of sheaves $\underline{\Omega}^\bullet$ on the site $\RigSm_{K, \an}$.
   By Corollary \ref{profinite}, $H_{\petc}^\bullet(-_C, \Bb_{\dR})$ is defined by the complex of sheaves $R\varepsilon_*\varepsilon^*\underline{\Bb_{\dR}}$ on the site $\RigSm_{K, \pet}$, where we denote by $\varepsilon:\RigSm_{C, \pet}\to \RigSm_{K, \pet}$ the base change morphism.
  
  We begin by proving axiom \listref{axiom:1}.  Let $X$ be a smooth affinoid space defined over $K$. For the condensed de Rham cohomology, we need to show that the natural map
  $$\underline{\Omega^\bullet(X)}\to \underline{\Omega^\bullet(X\times\accentset{\circ}{\Dd}_K)}$$
  is a quasi-isomorphism. From the knowledge of axiom \listref{axiom:1} for the (classical) de Rham cohomology (see the discussion after \cite[\S 1, Lemma 2]{SS}), we know that $\Omega^\bullet(X)\overset{\sim}{\to}\Omega^\bullet(X\times\accentset{\circ}{\Dd}_K)$ is a quasi-isomorphism of complexes of $K$-Fréchet spaces; then, we can conclude using Lemma \ref{acyclic}.  Now, axiom \listref{axiom:1} for $H_{\petc}^\bullet(-_C, \Bb_{\dR})$ follows observing that, writing the 1-dimensional open unit disk $\accentset{\circ}{\Dd}$ over $\Qq_p$ as a strictly increasing admissible union of closed disks $\{\Dd_j\}_{j\in \Nn}$ of radius in $p^{\Qq}$, we have 
  \begin{align}
   R\Gamma_{\petc}(X_C\times \accentset{\circ}{\Dd}_C, \Bb_{\dR})&=R\varprojlim\nolimits_j R\Gamma_{\petc}(X_C\times {\Dd}_{j,C}, \Bb_{\dR}) \nonumber\\
   &\simeq R\varprojlim\nolimits_j (R\Gamma_{\dRc}(X\times {\Dd}_{j,K})\dsolid_K B_{\dR}) \label{labb1}\\
   &\simeq R\varprojlim\nolimits_j (R\Gamma_{\dRc}(X\times \accentset{\circ}{\Dd}_{j,K})\dsolid_K B_{\dR}) \label{labb2}\\
   &\simeq R\Gamma_{\dRc}(X)\dsolid_K B_{\dR} \label{labb3}\\
   &\simeq R\Gamma_{\petc}(X_C, \Bb_{\dR}) \label{labb4}
  \end{align}
  where in (\ref{labb1}) and (\ref{labb4}) we used Remark \ref{quasi-comp}, in (\ref{labb2}) we used that, for each $j\in \Nn$, the morphism $R\Gamma_{\dRc}(X\times {\Dd}_{j+1,K})\to R\Gamma_{\dRc}(X\times {\Dd}_{j,K})$ factors through $R\Gamma_{\dRc}(X\times \accentset{\circ}{\Dd}_{j+1,K})$, and in (\ref{labb3}) we used axiom \listref{axiom:1} for the condensed de Rham cohomology.

  Axioms \listref{axiom:2} and \listref{axiom:3} are satisfied. 
   In order to verify axiom \listref{axiom:4}, we first need to construct the respective cycle class maps for the cohomology theories in the statement; we observe that it suffices to do this at the level of the complexes of abelian sheaves underlying the respective defining complexes of sheaves (having values in condensed abelian groups), i.e. it suffices to define  $c^{\dR}:\Gg_m[-1]\to \Omega^\bullet$ and $c^{B_{\dR}}:\Gg_m[-1]\to R\varepsilon_*\varepsilon^*\Bb_{\dR}$. 
 The cycle class map $c^{\dR}$ is defined on $\RigSm_{K, \an}$ by 
 \begin{equation}
   c^{\dR}:\Gg_m[-1]\overset{d\log}{\longrightarrow} \Omega^{\ge 1}\longrightarrow \Omega^\bullet,
 \end{equation}
 and the cycle class map $c^{B_{\dR}}$ is defined on $\RigSm_{K, \pet}$ as the composite 
  \begin{equation}
     c^{B_{\dR}}: \Gg_m[-1]\longrightarrow R\varepsilon_*\varepsilon^*\Gg_m[-1]\longrightarrow R\varepsilon_*\varepsilon^* \Zz_p(1)\longrightarrow R\varepsilon_*\varepsilon^* \Bb_{\dR}
  \end{equation}
  where the first arrow is the adjunction morphism, and the middle arrow comes from the boundary map of the Kummer exact sequence of sheaves on $\RigSm_{K, \pet}$
  \begin{equation}\label{kummer}
   0\to\Zz_p(1)\to\varprojlim_{\times p}\Gg_m \to \Gg_m\to 0.
  \end{equation}
  Then, axiom \listref{axiom:4} for the cohomology theory $H_{\underline{\dR}}^\bullet$ follows from the knowledge of the same axiom for the (classical) de Rham cohomology, recalling that $H_{\underline{\dR}}^\bullet(\Pp_K^d)=\underline{H_{\dR}^\bullet(\Pp_K^d)}$ by Lemma \ref{propclass}. For the cohomology theory $H_{\petc}^\bullet(-_C, \Bb_{\dR})$, we can reduce to the latter case, using the isomorphism (\ref{afterkunneth}) for the projective space $\Pp_K^d$, if we check the compatibility of \textit{loc. cit.} with the cycle class maps $c^{\dR}$ and $c^{B_{\dR}}$, up to a sign. For this, we show more generally that, for any $X\in \RigSm_K$, the quasi-isomorphism (\ref{afterpoincare0}) constructed in Corollary \ref{frompoincare} (in the case of trivial coefficients) is compatible with the cycle class maps, up to a sign. Since it suffices to do it locally on $X$, we can assume that there exists an étale map $X\to \Tt_K^n$ that can be written as a composite of finite \'etale maps and rational embeddings. Then, recalling that (\ref{afterpoincare0}) is induced by the quasi-isomorphism $\Bb_{\dR}\overset{\sim}{\to}\Omega_X^{\bullet}\otimes_{\cl O_X} \cl O\Bb_{\dR}$ of complex of sheaves on $X_{\pet}$, given by Proposition \ref{poincare}, we need to check that the following diagram of complexes of sheaves on $X_{\pet}$ is commutative, up to quasi-isomorphisms, and up to a sign,
  \begin{center}
   \begin{tikzcd}
   \cl O_X^\times[-1] \arrow[r, "d\log"] 
     & \Omega_X^{\ge 1}\otimes_{\cl O_X} \cl O\Bb_{\dR} \arrow[r]
     &  \Omega_X^{\bullet}\otimes_{\cl O_X} \cl O\Bb_{\dR}
    \\ 
    \left[\Zz_p(1)\to\varprojlim_{\times p}\cl O_X^\times\right] \arrow[u, "\wr"] \arrow[d] \arrow[r] 
     & \left[\Bb_{\dR}\to \cl O\Bb_{\dR}\right]\arrow[u, "\wr", "\nabla"'] \arrow[d]
     & 
     \\ 
    \Zz_p(1) \arrow[r] 
     & \Bb_{\dR} \arrow[-,double line with arrow={-,-}]{r} 
     & \Bb_{\dR} \arrow[uu, "\wr"]       
  \end{tikzcd}
  \end{center}
  where the middle horizontal map sends an ``element'' $x$ of $\varprojlim_{\times p}\cl O_X^\times$ to $\log(V/[U])\in  \cl O\Bb_{\dR}$, with $V\in \cl O_X^\times$ and $U\in \widehat{\cl O}_X^{\flat\times}$ the respective ``elements'' defined by $x$. For this, it suffices to observe that the right square of the diagram above is anticommutative, and the top left square is commutative: in fact, using the Leibniz rule, and $d[U]=0$, we have that
  $\nabla\left(\log(V/[U])\right)=d\log(V)$.\footnote{A similar argument appears in the proof of \cite[Lemma 3.24]{ScholzeSurvey}.}
  \end{proof}

 \section{\textbf{Pro-\'etale cohomology of $\Bb_e$}}\label{bee} 
 \sectionmark{}  
 
 In this section, we show that the pro-\'etale cohomology with coefficients in $\Bb_e=\Bb[1/t]^{\varphi=1}$, defined in \S \ref{petsheaves}, satisfies the axioms of Schneider--Stuhler (\ref{axioms}). The crucial axiom to be proven is the homotopy invariance with respect to the 1-dimensional open unit disk.

 \subsection{D\'ecalage functor and Koszul complexes}\label{deckos}
 
 We begin by recalling the construction of the \textit{d\'ecalage functor}, due to Deligne and Berthelot--Ogus, and defined more generally by Bhatt--Morrow--Scholze on any ringed site (or topos), \cite[\S 6]{BMS1}. For our purposes, only the case of the ringed site $(*_{\kappa{\text -}\pet}, A)$, for $\kappa$ a cut-off cardinal as in \S \ref{conve}, and $A$ a $\kappa$-condensed ring, will be relevant.\medskip
 
 Let $A$ be a ($\kappa$-)condensed ring, and let $D(A):=D(\Mod_A^{\cond})$ denote the derived category of $A$-modules in $\CondAb$. Let $f$ be a \textit{non-zero-divisor in} $A$, i.e. a generator of a principal invertible ideal sheaf of $A$; we denote by $(f)=fA$ the invertible ideal sheaf of $A$ it generates. We say that a complex $M^\bullet$ of $A$-modules in $\CondAb$ is \textit{$f$-torsion-free} if the map $(f)\otimes_A M^i\to M^i$ is injective for all $i\in \Zz$, and, in this case, we denote by $fM^i$ its image.
 
 \begin{df}
  Let $M^\bullet$ be an $f$-torsion-free complex of $A$-modules in $\CondAb$. We denote by $\eta_f(M^\bullet)$ the subcomplex of $M^\bullet[1/f]$ defined by
  $$\eta_f(M^\bullet)^i:=\{x\in f^iM^i: dx \in f^{i+1}M^{i+1}\}.$$
 \end{df}
 
 The endo-functor $\eta_f(-)$, defined on the category of $f$-torsion-free complexes of $A$-modules in $\CondAb$, induces an endo-functor on $D(A)$ that kills $f$-torsion in cohomology. More precisely, the following result holds true.
 
 \begin{lemma}[{\cite[Lemma 6.4]{BMS1}}]\label{1239}
  Let $M^\bullet$ be an $f$-torsion-free complex of $A$-modules in $\CondAb$. Then, for all $i\in \Zz$, there is a canonical isomorphism
  $$H^i(\eta_f(M^\bullet))\cong (H^i(M^\bullet)/H^i(M^\bullet)[f])\otimes_A (f^i)$$
  where $H^i(M^\bullet)[f]$ denotes the $f$-torsion of $H^i(M^\bullet)$, and $(f^i)\subset A[1/f]$ denotes the invertible ideal sheaf generated by $f^i$.
 \end{lemma}

 Every complex of $A$-modules in $\CondAb$ is quasi-isomorphic to an $f$-torsion-free complex of $A$-modules in $\CondAb$, \cite[Lemma 6.1]{BMS1}. Then, from Lemma \ref{1239}, one can deduce the following key result.
 
 \begin{cor}[{\cite[Corollary 6.5]{BMS1}}]\label{decfunc}
  The functor $\eta_f$ from $f$-torsion-free complexes of $A$-modules in $\CondAb$ to $D(A)$ factors canonically over a functor
  $$L\eta_f:D(A)\to D(A)$$
  called the \textit{d\'ecalage functor}, that satisfies $H^i(L\eta_f(M))\cong (H^i(M)/H^i(M)[f])\otimes_A (f^i)$ functorially in $M$.
 \end{cor}
 
 \begin{rem}\label{obv}
 For any $M\in D(A)$, we have $$L\eta_f(M)[1/f]=M[1/f].$$
 \end{rem}

 \begin{rem}
  The d\'ecalage functor $L\eta_f$ is not exact (see \cite[Remark 6.6]{BMS1} for a counterexample).
 \end{rem}

 Next, we check that the d\'ecalage functor preserves ``solid complexes'', cf. \cite[Lemma 6.19]{BMS1}. In the following statement, given an analytic ring $(A, \cl M)$,\footnote{See \S \ref{ape} for the notation, and Remark \ref{solidcut} for the set-theoretic conventions.} we denote $D(A, \cl M):=D(\Mod_{(A, \cl M)}^{\cond})$ the derived category of $\Mod_{(A, \cl M)}^{\cond}$.
 
 \begin{lemma}
  Let $(A, \cl M)$ be an analytic ring. If $M\in D(A, \cl M)$, then $L\eta_f M\in D(A, \cl M)$.
 \end{lemma}
 \begin{proof}
  By Proposition \ref{fevre}\listref{fevre:2}, we know that $L\eta_f M\in D(A, \cl M)$ if and only if $H^i(L\eta_f M)\in \Mod_{(A, \cl M)}^{\cond}$ for all $i$. Then, the statement follows recalling that, by Corollary \ref{decfunc}, we have a (non-canonical) isomorphism $H^i(L\eta_f(M))\cong H^i(M)/H^i(M)[f]$, and observing that $H^i(M)[f]$ lies in $\Mod_{(A, \cl M)}^{\cond}$ since, by definition, it is the kernel of a morphism in $\Mod_{(A, \cl M)}^{\cond}$.
 \end{proof}
 
 A favorable property of the décalage functor $L\eta_f$ is that, in some suitable cases, it simplifies Koszul complexes, whose definition we now recall (translated in the condensed setting). See also \cite[\S 7]{BMS1}.
 
 \begin{df}\label{Kosz}
   Let $M$ be a condensed abelian group and let $f_i:M\to M$, $i=1, \ldots, n$, be $n$ endomorphisms of $M$ that commute with each other. We define the \textit{Koszul complex}
   $$\Kos_M(f_1, \ldots, f_n):=M\otimes_{\Zz[f_1, \ldots, f_n]}\bigotimes_{i=1}^n(\Zz[f_1, \ldots, f_n]\overset{f_i}{\to}\Zz[f_1, \ldots, f_n])$$
   as a complex of condensed abelian groups that sits in non-negative cohomological degrees.
 \end{df}

 \begin{lemma}\label{kill}
  Let $A$ be a condensed ring and $f, g_1, \ldots, g_m\in A$ be non-zero-divisors, such that each $g_i$ either divides $f$ or is divisible by $f$. Let $M^\bullet$ be an $f$-torsion-free complex of $A$-modules in $\CondAb$.
  \begin{enumerate}[(i)]
   \item \label{kill:1} If some $g_i$ divides $f$, then $\eta_f(M^\bullet\otimes_A\Kos_A(g_1, \ldots, g_m))$ is acyclic.
   \item \label{kill:2} If $f$ divides $g_i$ for all $i$, then
   $$\eta_f(M^\bullet\otimes_A\Kos_A(g_1, \ldots, g_m))\cong \eta_f M^\bullet\otimes_A \Kos_{A}(g_1/f, \ldots, g_m/f).$$
  \end{enumerate}
 \end{lemma}
 \begin{proof}
  See the proof of \cite[Lemma 7.9]{BMS1}.
 \end{proof}
 
 We will also need to relate Koszul complexes with condensed group cohomology. For this, we refer the reader to  Appendix \ref{ccg}, in particular Proposition \ref{contkosz}.

 \subsection{Homotopy invariance with respect to the open disk}
 
 We are ready to prove that pro-étale cohomology with coefficients in $\Bb_e=\Bb[1/t]^{\varphi=1}$ satisfies the homotopy invariance with respect to the 1-dimensional open unit disk.\medskip
 
 We refer the reader to \S \ref{petsheaves} for the definitions of the pro-étale period sheaves used here. In addition, we introduce the following notation. 
 
  \begin{notation}\label{notnot}
   As in Notation \ref{notff}, we fix a compatible system $(1, \varepsilon_p, \varepsilon_{p^2}, \ldots)$ of $p$-th power roots of unity in $\cl O_C$, which defines an element $\varepsilon\in \cl O_C^{\flat}$. We denote by $[\varepsilon]\in A_{\inf}=\Aa_{\inf}(C, \cl O_C)$ its Teichm\"uller lift and $\mu=[\varepsilon]-1\in A_{\inf}$. Furthermore, we let $\xi=\mu/\varphi^{-1}(\mu)\in A_{\inf}$ and $t=\log[\varepsilon]\in B=\Bb(C, \cl O_C)$.
   Given a compact interval $I\subset (0, \infty)$ with rational endpoints, we let $A_I=\Aa_I(C, \cl O_C)$ and $B_I=\Bb_I(C, \cl O_C)$, and if the context is clear we omit the underscores from the notation of the associated condensed rings.
   \end{notation}
   
   We will use several times the following fact.
  \begin{lemma}\label{reap}
   Let $I\subset[1/(p-1), \infty)$ a compact interval with rational endpoints. Then, the element $t\in B_I$ can be written as $t=\mu\cdot u$ for some unit $u\in B_I$.
  \end{lemma}
  \begin{proof}
   Since $I\subset [1/(p-1), \infty)$, we have $A_{\cris}\subset A_I$ (see e.g. \cite[\S 2.4.2]{CN1}). Then, the statement follows from \cite[Lemma 12.2, (iii)]{BMS1}.
  \end{proof}

   To illustrate the ideas needed to prove the main result of this section, namely Corollary \ref{smadonno}, we start by studying the geometric pro-étale cohomology of $\Bb_I[1/t]$ on the torus.\medskip
   
   We denote  by $$\Tt_C^n:=\Spa(R, R^+)=\Spa(C\langle T_1^{\pm 1}, \ldots, T_n^{\pm 1}\rangle, \cl O_C\langle T_1^{\pm 1}, \ldots, T_n^{\pm 1}\rangle)$$ 
   the $n$-torus over $C$, and we write  $$\widetilde \Tt_C^n:=\Spa(R_{\infty}, R_{\infty}^+)=\Spa(C\langle T_1^{\pm 1/p^\infty}, \ldots, T_n^{\pm 1/p^\infty}\rangle, \cl O_C\langle T_1^{\pm 1/p^\infty}, \ldots, T_n^{\pm 1/p^\infty}\rangle)$$
   for the affinoid perfectoid $n$-torus over $C$.
 \begin{lemma}\label{adjoint}
  Let $I$ be a compact interval of $(0, \infty)$ with rational endpoints. 
  \begin{enumerate}[(i)]
   \item\label{adjoint:1} We have a canonical identification 
  \begin{equation*}
   \Aa_I(\widetilde \Tt_C^n)=A_I\langle V_1^{\pm 1/p^{\infty}}, \ldots, V_n^{\pm 1/p^{\infty}}\rangle:=A_{\inf}\langle V_1^{\pm 1/p^{\infty}}, \ldots, V_n^{\pm 1/p^{\infty}} \rangle\widehat\otimes_{A_{\inf}}A_I
  \end{equation*}
  where $V_i:=[T_i^\flat]$, we denote by $\widehat\otimes_{A_{\inf}}$ the $p$-adically completed tensor product, and we write $A_{\inf}\langle V_1^{\pm 1/p^{\infty}}, \ldots, V_n^{\pm 1/p^{\infty}} \rangle$ for the $(p, [p^\flat])$-adic completion of $A_{\inf}[V_1^{\pm 1/p^{\infty}}, \ldots, V_n^{\pm 1/p^{\infty}}]$.
  \item\label{adjoint:2} We have
  $$A_I\langle V_1^{\pm 1/p^{\infty}}, \ldots, V_n^{\pm 1/p^{\infty}}\rangle=\widehat{\bigoplus_{(a_1, \ldots, a_n)\in \Zz[1/p]^n}} A_I\cdot \prod_{j=1}^n V_j^{a_j}$$
  where the completion over the direct sum is $p$-adic.
  \end{enumerate}

 \end{lemma}
 \begin{proof}
  By Proposition \ref{affinoidsections}, we have $\Aa_I(\widetilde \Tt_C^n)=\Aa_I(R_{\infty}, R_{\infty}^+)$, hence, unraveling the definition of $\Aa_I(R_{\infty}, R_{\infty}^+)$, for part \listref{adjoint:1} it suffices to identify $\Aa_{\inf}(R_{\infty}, R_{\infty}^+)$ with $A_{\inf}\langle V_1^{\pm 1/p^{\infty}}, \ldots, V_n^{\pm 1/p^{\infty}}\rangle$. We recall from \S \ref{petsheaves} that we have $\Aa_{\inf}(R_{\infty}, R_{\infty}^+)=W(R^{+\flat}_{\infty})$ and $R^{+\flat}_{\infty}=\cl O_C^\flat\langle (T_1^\flat)^{\pm 1/p^\infty}, \ldots, (T_n^\flat)^{\pm 1/p^\infty} \rangle$
  where the completion is $p^\flat$-adic. We deduce that $W(R^{+\flat}_{\infty})=W(\cl O_C^\flat)\langle V_1^{\pm 1/p^{\infty}}, \ldots, V_n^{\pm 1/p^{\infty}}\rangle$ where the completion is $(p, [p^\flat])$-adic, as desired.\footnote{One way to see this is to use that the Witt vectors $W(-)$ induce an equivalence of categories between the perfect $\Ff_p$-algebras and the category of $p$-adically complete, separated and $p$-torsion free $\Zz_p$-algebras $R$ such that $R/p$ is perfect, with quasi-inverse the tilting functor (see e.g. \cite[Remark 2.5, Example 2.6]{Bhatt}).} 
  
   For part \listref{adjoint:2} it suffices to observe that the $p$-adic topology is equivalent to the $(p, [p^\flat])$-adic topology on $A_I$: in fact, recalling Definition \ref{I}, $A_I$ is the $p$-adic completion of $A_{\inf, I}=\Aa_{\inf, I}(C, \cl O_C)$, and, for $I=[s, r]$, one has that $p$ divides $[p^\flat]^{1/s}$ in $A_{\inf, I}$. 
 \end{proof}

 \begin{rem}\label{BIact}
 Recall that the affinoid perfectoid pro-(finite étale) cover of the $n$-torus $\widetilde \Tt_C^n\to \Tt_C^n$ has Galois group $\Zz_p(1)^n$. The choice of $\varepsilon$ (Notation \ref{notnot}) gives an isomorphism $\Zz_p(1)^n\cong \Zz_p^n$. We denote $\Gamma:= \Zz_p^n$ and by $\gamma_1, \ldots, \gamma_n$  the canonical generators of $\Gamma$.  One can describe explicitly the Galois action of $\Gamma$ on the coordinates $(R_\infty, R_{\infty}^+)$ of $\widetilde \Tt_C^n$, as follows. We can write 
 $$R_{\infty}^+=\cl O_C\langle T_1^{\pm 1/p^{\infty}}, \ldots, T_n^{\pm 1/p^{\infty}}\rangle=\widehat{\bigoplus_{(a_1, \ldots, a_n)\in \Zz[1/p]^n}} \cl O_C\cdot \prod_{j=1}^n T_j^{a_j}$$
  where the completion over the direct sum is $p$-adic. Then, the action of $\Gamma$ on $R_{\infty}^+$ preserves the decomposition above and $\gamma_i$ acts on $\prod_{j=1}^n T_j^{a_j}$ via multiplication by $\varepsilon_p^{a_i}$.
 
 \end{rem}
 
 In the following, we maintain Notation \ref{notasolid}.
 
 \begin{prop}\label{torusB}
  Let $\Tt^n$ denote the $n$-dimensional torus defined over $\Qq_p$. For any compact interval $I\subset (0, \infty)$ with rational endpoints, we have a Frobenius-equivariant isomorphism in $D(\Vect_{\Qq_p}^{\ssolid})$
  $$R\Gamma_{\petc}(\Tt_C^n, \Bb_I[1/t])\simeq\underline{\Omega}^\bullet(\Tt^n)\dsolid_{\Qq_p} B_I[1/t].$$
 \end{prop}
 \begin{proof}
 We will show that, for any compact interval $I\subset (0, \infty)$ with rational endpoints, we have a Frobenius-equivariant isomorphism in $D(\Vect_{\Qq_p}^{\ssolid})$
 \begin{equation}\label{LI}
  L\eta_tR\Gamma_{\petc}(\Tt_C^n, \Bb_I)\simeq\underline{\Omega}^\bullet(\Tt^n)\dsolid_{\Qq_p} B_I.
 \end{equation}
 We note that the statement follows inverting $t$ in  (\ref{LI}), using that the tensor product $\dsolid_{\Qq_p}$ commutes with filtered colimits, and observing 
  that $(L\eta_tR\Gamma_{\petc}(\Tt_C^n, \Bb_I))[1/t]=R\Gamma_{\petc}(\Tt_C^n, \Bb_I[1/t])$ (recall Remark \ref{obv}, and use that $|\Tt_C^n|$ is quasi-compact and quasi-separated). 
 
 Fix a compact interval $I\subset (0, \infty)$ with rational endpoints. We compute $R\Gamma_{\petc}(\Tt_C^n, \Bb_I)$ using the Cartan--Leray spectral sequence relative to the pro-étale $\Gamma$-torsor $\widetilde \Tt^n_C\to \Tt^n_C$: by Proposition \ref{cart-ler}, we have a quasi-isomorphism
  $$R\Gamma_{\underline{\cond}}(\Gamma , \underline{\Bb_I}(\widetilde \Tt_C^n))\overset{\sim}{\to} R\Gamma_{\petc}(\Tt_C^n, \Bb_I).$$
  Thus, we need to study the complex $L\eta_t R\Gamma_{\underline{\cond}}(\Gamma , \underline{\Bb_I}(\widetilde \Tt_C^n)).$
  By Lemma \ref{adjoint}\listref{adjoint:1}, we can write
  \begin{equation}\label{summation}
   \Aa_I(\widetilde \Tt_C^n)=A_I(R) \oplus A_I(R_{\infty})^{\nonint},\;\;\;\;\;\;\;\Bb_I(\widetilde \Tt_C^n)=B_I(R) \oplus B_I(R_{\infty})^{\nonint}
  \end{equation}
  where $A_I(R):=A_I\langle V_1^{\pm 1}, \ldots, V_n^{\pm 1}\rangle$ denotes the ``integral part'', $A_I(R_{\infty})^{\nonint}$ denotes the ``non-integral part'' of $\Aa_I(\widetilde \Tt_C^n)$, and similarly for $\Bb_I(\widetilde \Tt_C^n)$ by inverting $p$.
   By Proposition \ref{contkosz}, we have
  $$R\Gamma_{\underline{\cond}}(\Gamma , \underline{\Bb_I}(\widetilde \Tt_C^n))\simeq \Kos_{\underline{B_I(R)}}(\gamma_1-1, \ldots, \gamma_n-1) \oplus \Kos_{\underline{B_I(R_{\infty})}^{\nonint}}(\gamma_1-1, \ldots, \gamma_n-1)$$
  Now, we first assume that $I\subset [1/(p-1), \infty)$.  
  \begin{itemize}
   \item We begin by showing that we have $L\eta_t\Kos_{\underline{B_I(R_{\infty})}^{\nonint}}(\gamma_1-1, \ldots, \gamma_n-1)=0$. For this, it suffices to prove that the multiplication by $t$ on the  complex $\Kos_{B_I(R_{\infty})^{\nonint}}(\gamma_1-1, \ldots, \gamma_n-1)$ is homotopic to 0.\footnote{In fact, this would imply that the cohomology groups of the complex of condensed $\Qq_p$-vector spaces $\Kos_{\underline{B_I(R_{\infty})}^{\nonint}}(\gamma_1-1, \ldots, \gamma_n-1)$ are annihilated by $t$, and hence the claim by Corollary \ref{decfunc}.} Let $\mu=[\varepsilon]-1$ as in Notation \ref{notnot}, and recall that, by Lemma \ref{reap}, we have that $t=\mu\cdot u$ for some unit $u\in B_I$. Then, it suffices to show that the multiplication by $\varphi^{-1}(\mu)=[\varepsilon]^{1/p}-1$ on $\Kos_{A_I(R_{\infty})^{\nonint}}(\gamma_1-1, \ldots, \gamma_n-1)$ is homotopic to $0$. This can be done as in the proof of \cite[Lemma 9.6]{BMS1}:  by Remark \ref{BIact} and Lemma \ref{adjoint}\listref{adjoint:2}, we have
   $$\Kos_{A_I(R_{\infty})^{\nonint}}(\gamma_1-1, \ldots, \gamma_n-1)=\widehat{\bigoplus_{(a_1, \ldots, a_n)}} \Kos_{A_I(R)}(\gamma_1[\varepsilon]^{a_1}-1, \ldots, \gamma_n[\varepsilon]^{a_n}-1)$$
   where the completion is $p$-adic, and the direct sum runs over $a_1, \ldots, a_n\in \Zz[1/p]\cap [0, 1)$ not all 0. Hence, we are reduced to showing that for $a_i\in \Zz[1/p]\cap (0, 1)$ the multiplication by $\varphi^{-1}(\mu)$ on the complex 
   \begin{center}
   \begin{tikzcd}
   A_I(R) \arrow[r, "\gamma_i\text{[}\varepsilon\text{]}^{a_i}-1"] &[1em] A_I(R)
  \end{tikzcd}
  \end{center}
  is homotopic to 0, i.e. we have to find $h$ that completes the following diagram
  \begin{center}
   \begin{tikzcd}
   A_I(R) \arrow[r, "\gamma_i\text{[}\varepsilon\text{]}^{a_i}-1"]\arrow[d, "\varphi^{-1}(\mu)"'] &[1em] A_I(R)\arrow[d, "\varphi^{-1}(\mu)"]\arrow[dl, dashed, "h"'] \\
   A_I(R) \arrow[r, "\gamma_i\text{[}\varepsilon\text{]}^{a_i}-1"] &[1em] A_I(R)
  \end{tikzcd}
  \end{center}
  Let us write $a_i=m/p^r$, with $m\in \Zz\setminus p\Zz$. Up to changing the choice of $(1, \varepsilon_p, \varepsilon_{p^2}, \ldots)$ in Notation \ref{notnot}, we can suppose $m=1$. Furthermore, since $\gamma_i[\varepsilon]^{1/p^r}-1$ divides $\gamma_i^{p^{r-1}}[\varepsilon]^{1/p}-1$, it suffices to show that the latter map is homotopic to 0. Then, one has to find $h$ such that 
  $$\gamma_i^{p^{r-1}}(h(x))[\varepsilon]^{1/p}-h(x)=\varphi^{-1}(\mu)x.$$
  For this, we refer the reader to  \cite[Lemma 9.6]{BMS1}.
   
   \item Next, we show that 
   $$L\eta_t\Kos_{\underline{B_I(R)}}(\gamma_1-1, \ldots, \gamma_n-1)\simeq \underline{\Omega}^\bullet(\Tt^n)\solid_{\Qq_p}B_I.$$ Since $\mu$ divides $\gamma_i-1$, i.e. $\gamma_i$ acts trivially on $B_I(R)/\mu$, and by Lemma \ref{reap} $t=\mu\cdot u$ for some unit $u\in B_I$, using Lemma \ref{kill}\listref{kill:2} we have
   $$L\eta_t\Kos_{\underline{B_I(R)}}(\gamma_1-1, \ldots, \gamma_n-1)\simeq \Kos_{\underline{B_I(R)}}
   \left(\frac{\gamma_1-1}{t}, \ldots, \frac{\gamma_n-1}{t}\right).$$
   By the proof of \cite[Lemma 12.4]{BMS1}, one has the following Taylor expansion in $B_I(R)$
   $$\gamma_i=\sum_{j\ge 0}\frac{t^j}{j!}\left(\frac{\partial}{\partial\log(V_i)}\right)^j$$
   from which we can write
   $$\frac{\gamma_i-1}{t}=\frac{\partial}{\partial\log(V_i)}\left(1+H\right), \; \text{  with  }\;  H:=\sum_{j\ge 1}\frac{t^{j}}{(j+1)!}\left(\frac{\partial}{\partial\log(V_i)}\right)^{j}.$$
   Note that $H$ is topologically nilpotent, using again that $A_{\cris}\subset B_I$ since $I\subset [1/(p-1), \infty)$, and \cite[Lemma 12.2, (ii)]{BMS1}. In particular, the factor $1+H$ is an automorphism of $B_I(R)$; moreover, the latter automorphism is Frobenius-equivariant recalling that $\varphi(t)=pt$ and $$\frac{\partial}{\partial\log(V_i)}\circ \varphi=p\left(\varphi\circ \frac{\partial}{\partial\log(V_i)}\right).$$  
   We deduce that we have a Frobenius-equivariant quasi-isomorphism
   $$\Kos_{\underline{B_I(R)}}\left(\frac{\gamma_1-1}{t}, \ldots, \frac{\gamma_n-1}{t}\right)\simeq \Kos_{\underline{B_I(R)}}\left(\frac{\partial}{\partial\log(V_1)}, \ldots, \frac{\partial}{\partial\log(V_n)}\right)\simeq \underline{\Omega}^\bullet(\Tt^n)\solid_{\Qq_p} B_I$$
   where in the last step we used Remark \ref{coinban}.\footnote{In fact, by \cite[Appendix B, proposition 5]{Bosch}, $B_I(R)$ is naturally isomorphic to the completed tensor product of Banach $\Qq_p$-algebras $R\widehat \otimes_{\Qq_p}B_I$. }
  \end{itemize}
  
  Then, putting the above points together, and using that $B_I$ is flat for the tensor product $\solid_{\Qq_p}$ by Corollary \ref{solidqs}, we obtain a quasi-isomorphism as in (\ref{LI}) for any given $I\subset [1/(p-1), \infty)$.
  Now, since the Frobenius $\varphi$ induces an isomorphism $\varphi: \Bb_I\overset{\sim}{\to}\Bb_{I/p}$ (Remark \ref{frr}), applying $\varphi^{N}$ to the latter quasi-isomorphism, for $N$ a sufficiently big positive integer, we deduce that (\ref{LI}) holds true for a general compact interval $I\subset(0, \infty)$ with rational endpoints, as desired.
 \end{proof}

 To reach the stated goal of this section, we will need in particular to study the pro-étale cohomology with coefficients in $\Bb_I[1/t]$ of the 1-dimensional unit \textit{closed disk} $\Dd_C$. For this, the general strategy we will use is similar to the one we have seen for the torus (Proposition \ref{torusB}), but there is one slight difference, which is explained in the following remark.
 
 \begin{rem}\label{remeh}
  If we denote by $\widetilde \Dd_C=\Spa(C\langle T^{1/p^{\infty}}\rangle, \cl O_C\langle T^{1/p^{\infty}}\rangle)$ the affinoid perfectoid unit closed disk over $C$, then the cover $\widetilde \Dd_C^\diamondsuit \to \Dd_C^\diamondsuit$ is \textit{not} pro-étale. But, it is quasi-pro-étale. Hence, recalling Definition \ref{qpet}, we can still study the pro-étale cohomology of $\Dd_C$ with coefficients in $\Bb_I[1/t]$ using the \v{C}ech-to-cohomology spectral sequence relative to such cover.
 \end{rem} 
 
 Thus, we need to better understand $\widetilde \Dd_{C, j}$, the $j$-fold fibre product of $\widetilde \Dd_C$ over $\Dd_C$. The following result is due to Le Bras.
 
 \begin{lemma}\label{muiomale}
  Let $\Dd_C=\Spa(C\langle T\rangle, \cl O_C\langle T\rangle)$ denote the 1-dimensional unit closed disk over $C$. Let $\widetilde \Dd_C=\Spa(C\langle T^{1/p^{\infty}}\rangle, \cl O_C\langle T^{1/p^{\infty}}\rangle)$ be the affinoid perfectoid unit closed disk over $C$. For $j\ge 1$, let $\widetilde \Dd_{C, j}$ denote the $j$-fold fibre product of $\widetilde \Dd_C$ over $\Dd_C$. Then, we have $$\widetilde \Dd_{C, j}\cong\Spa(R_j^+[1/p], R_j^+)$$ where\footnote{Here, we denote by $f|_{T=0}$ the evaluation of the function $f$ at $T=0$.} $$R_j^+:=\{f\in\mathscr{C}^0(\Zz_p^{j-1}, \cl O_C\langle T^{1/p^{\infty}}\rangle): f|_{T=0} \text{ {\normalfont is constant}}\}.$$
 \end{lemma}
 \begin{proof}
  This is a particular case of \cite[Lemme 3.30]{LeBras1}.
 \end{proof}

 \begin{rem}\label{59}
  Let $I$ be a compact interval of $(0, \infty)$ with endpoints. 
 Similarly to Lemma \ref{adjoint}\listref{adjoint:1}, from Lemma \ref{muiomale} we deduce that we have $$\Aa_I(\widetilde \Dd_{C, j})=\{f\in\mathscr{C}^0(\Zz_p^{j-1}, A_I\langle V^{1/p^{\infty}}\rangle): f|_{V=0} \text{ is constant}\}$$
 where $V:=[T^\flat]$, and $A_I\langle V^{1/p^{\infty}}\rangle:=A_{\inf}\langle V^{1/p^{\infty}} \rangle \widehat\otimes_{A_{\inf}} A_I$ (here, we denote by $\widehat\otimes_{A_{\inf}}$ the $p$-adically completed tensor product, and we write $A_{\inf}\langle V^{1/p^{\infty}} \rangle$ for the $(p, [p^\flat])$-adic completion of $A_{\inf}[V^{\pm 1/p^{\infty}}]$).
 \end{rem}

 \begin{prop}\label{pota}
   Let $\Dd$ denote the $1$-dimensional unit closed disk defined over $\Qq_p$. Let $X$ be a smooth affinoid space defined over $C$ with an étale map $X\to \Tt_C^n$ that factors as a composite of rational embeddings and finite étale maps. Then, for any compact interval $I\subset (0, \infty)$ with rational endpoints, we have a Frobenius-equivariant isomorphism in $D(\Vect_{\Qq_p}^{\ssolid})$
   $$R\Gamma_{\petc}(X\times \Dd_C, \Bb_I[1/t])\simeq R\Gamma_{\petc}(X, \Bb_I[1/t])\dsolid_{\Qq_p} \underline{\Omega}^\bullet(\Dd).$$
 \end{prop}
 \begin{proof}
 Following the proof of Proposition \ref{torusB}, it suffices to show that, for any compact interval $I\subset (0, \infty)$ with rational endpoints, we have a Frobenius-equivariant isomorphism in $D(\Vect_{\Qq_p}^{\ssolid})$
 \begin{equation*}
  L\eta_t R\Gamma_{\petc}(X\times \Dd_C, \Bb_I)\simeq L\eta_tR\Gamma_{\petc}(X, \Bb_I)\dsolid_{\Qq_p} \underline{\Omega}^\bullet(\Dd).
 \end{equation*}
 
 Fix a compact interval $I\subset (0, \infty)$ with rational endpoints.
 We compute $R\Gamma_{\petc}(X\times \Dd_C, \Bb_I)$ using the \v{C}ech-to-cohomology spectral sequence associated to the cover $\widetilde X\times \widetilde \Dd_C\to X\times \Dd_C$, where $\widetilde X:=X\times_{\Tt^n_C}\widetilde \Tt^n_C$ (see Remark \ref{remeh}).
  Recall that, by Proposition \ref{affinoidsections}\listref{affinoidsections:3}, for any affinoid perfectoid $Z$ over $\Spa(C, \cl O_C)$, we have $H_{\petc}^i(Z, \Bb_I)=0$ for all $i>0$. Moreover, for $j\ge 1$, denoting by $\widetilde X_j$ the $j$-fold fibre product of $\widetilde X$ over $X$, by Lemma \ref{muiomale} and Remark \ref{59}, we have that
  $$
   \Bb_I(\widetilde X_j\times \widetilde \Dd_{C, j})=\{ f\in\mathscr{C}^0(\Gamma^{j-1}\times\Zz_p^{j-1}, \Bb_I(\widetilde X_j)\langle V^{1/p^{\infty}}\rangle): f(\gamma, -)|_{V=0} \text{ is constant for all } \gamma \in \Gamma^{j-1}\}
  $$
  where $\Gamma=\Zz_p^n$, $V=[T^\flat]$, and $\Bb_I(\widetilde X_j)\langle V^{1/p^{\infty}}\rangle:=\Aa_I(\widetilde X_j)\langle V^{1/p^{\infty}}\rangle[1/p]$. Hence, $H_{\petc}^0(\widetilde X_j\times \widetilde \Dd_{C, j}, \Bb_I)$ fits into the exact sequence
  $$0\to \underline{\Hom}(\Zz[\Gamma^{j-1}\times\Zz_p^{j-1}], \underline{N_j})\to H_{\petc}^0(\widetilde X_j\times \widetilde \Dd_{C, j}, \Bb_I)\to \underline{\Hom}(\Zz[\Gamma^{j-1}], \underline{\Bb_I}(\widetilde X_j))\to 0$$
  where $N_j:=V\cdot \Bb_I(\widetilde X_j)\langle V^{1/p^{\infty}}\rangle$, and the last map is given by $f\mapsto f|_{V=0}$.
  Then, from Proposition \ref{cond=cont}\listref{cond=cont:1}, we deduce that we have a distinguished triangle as follows
  \begin{equation}\label{distingue}
  R\Gamma_{\underline{\cond}}(\Gamma\times \Zz_p, \underline{N})\to R\Gamma_{\petc}(X\times \Dd_C, \Bb_I)\to R\Gamma_{\underline{\cond}}(\Gamma, \underline{\Bb_I}(\widetilde X))
  \end{equation}
  where $N:=V\cdot \Bb_I(\widetilde X)\langle V^{1/p^{\infty}}\rangle$. 
  Next, we want to study $L\eta_t R\Gamma_{\petc}(X\times \Dd_C, \Bb_I)$.
  
  Let us begin by handling $L\eta_t R\Gamma_{\underline{\cond}}(\Gamma\times \Zz_p, \underline{N})$. We can write
  $N=N^{\intt}\oplus N^{\nonint}$,
  where $N^{\intt}:=V\Bb_I(\widetilde X)\langle V\rangle$ denotes the ``integral part'', and $N^{\nonint}$ the ``non-integral part'' of $N$.
  Denoting by $\gamma_1, \ldots, \gamma_{n+1}$ the canonical generators of $\Gamma\times \Zz_p$, by Proposition \ref{contkosz} we have that
  $$R\Gamma_{\underline{\cond}}(\Gamma\times \Zz_p, \underline{N})\simeq \Kos_{\underline{N}^{\intt}}(\gamma_1-1, \ldots, \gamma_{n+1}-1) \oplus \Kos_{\underline{N}^{\nonint}}(\gamma_1-1, \ldots, \gamma_{n+1}-1).$$
   In the following, we assume first that $I\subset [1/(p-1), \infty)$. Then, similarly to the proof of Proposition \ref{torusB}, one checks that the multiplication by $t$ on the complex $\Kos_{N^{\nonint}}(\gamma_1-1, \ldots, \gamma_{n+1}-1)$ is homotopic to 0,  and hence $$L\eta_t\Kos_{\underline{N}^{\nonint}}(\gamma_1-1, \ldots, \gamma_{n+1}-1)=0.$$ Moreover, by Corollary \ref{solidqs}, we have
  $$\Kos_{\underline{N}^{\intt}}(\gamma_1-1, \ldots, \gamma_{n+1}-1)\simeq M^\bullet\dsolid_{\Qq_p}\Kos_{V\underline{\Qq_p\langle V\rangle}}(\gamma_{n+1}-1)$$
  where $M^\bullet:=\Kos_{\underline{\Bb_I}(\widetilde X)}(\gamma_1-1, \ldots, \gamma_{n}-1)$. 
  Therefore, arguing again as in the proof of Proposition \ref{torusB}, by the proof of Lemma \ref{kill}\listref{kill:2}, we have a Frobenius-equivariant quasi-isomorphism
 $$L\eta_t\Kos_{\underline{N}^{\intt}}(\gamma_1-1, \ldots, \gamma_{n+1}-1) \simeq L\eta_tM^\bullet\dsolid_{\Qq_p} \Kos_{V\underline{\Qq_p\langle V\rangle}}\left(\frac{\partial}{\partial\log(V)}\right)$$
 where, by Proposition \ref{contkosz}, $M^\bullet\simeq R\Gamma_{\underline{\cond}}(\Gamma, \underline{\Bb_I}(\widetilde X))\simeq R\Gamma_{\petc}(X, \Bb_I)$.

  Putting everything together, and recalling the definition of the last arrow of (\ref{distingue}), we obtain the following Frobenius-equivariant quasi-isomorphism
  \begin{align*}
   L\eta_t R\Gamma_{\petc}(X\times \Dd_C, \Bb_I)&\simeq L\eta_t R\Gamma_{\petc}(X, \Bb_I)\dsolid_{\Qq_p} \Kos_{\underline{{\Qq_p}\langle V\rangle}}\left(\frac{\partial}{\partial\log(V)}\right) \\
   &\simeq L\eta_t R\Gamma_{\petc}(X, \Bb_I)\dsolid_{\Qq_p} \underline{\Omega}^\bullet(\Dd).
  \end{align*}
  Arguing as in the proof of Proposition \ref{torusB}, such quasi-isomorphism extends to a general compact interval $I\subset(0, \infty)$ with rational endpoints, as desired. 
 \end{proof}

 \begin{rem}\label{trivia}
  We note that Proposition \ref{pota} holds replacing the 1-dimensional unit closed disk $\Dd$ with any 1-dimensional closed disk $\Dd(\rho)$ over $\Qq_p$ of radius $\rho\in p^{\Qq}$.
 \end{rem}

 \begin{cor}\label{smadonno}
  Let $\accentset{\circ}{\Dd}$ denote the $1$-dimensional open unit disk defined over $\Qq_p$. Let $X$ be a smooth affinoid space defined over $C$. Then, the natural projection map $X\times \accentset{\circ}{\Dd}_C\to X$ induces a isomorphism in $D(\Vect_{\Qq_p}^{\ssolid})$
  $$R\Gamma_{\petc}(X, \Bb_e)\overset{\sim}{\to} R\Gamma_{\petc}(X\times \accentset{\circ}{\Dd}_C, \Bb_e).$$
 \end{cor}
 \begin{proof}
  Using \cite[Lemma 5.12]{Scholze}, we can reduce to the case in which there exists an étale map $X\to \Tt_C^n$ that factors as a composite of rational embeddings and finite étale maps. We write $\accentset{\circ}{\Dd}$ as a strictly increasing admissible union of closed disks $\{\Dd_j\}_{j\in \Nn}$ of radius in $p^{\Qq}$. Then, for any compact interval $I\subset (0, \infty)$ with rational endpoints, we have the following natural quasi-isomorphisms
  \begin{align}
  R\Gamma_{\petc}(X\times \accentset{\circ}{\Dd}_C, \Bb_I[1/t])&\simeq R\varprojlim\nolimits_j R\Gamma_{\petc}(X\times \Dd_{j, C}, \Bb_I[1/t]) \nonumber \\
   &\simeq R\varprojlim\nolimits_j \left(R\Gamma_{\petc}(X, \Bb_I[1/t])\dsolid_{\Qq_p} R\Gamma_{\underline{\dR}}(\Dd_j)\right) \label{lab3}\\
   &\simeq R\varprojlim\nolimits_j \left(R\Gamma_{\petc}(X, \Bb_I[1/t])\dsolid_{\Qq_p} R\Gamma_{\underline{\dR}}(\accentset{\circ}{\Dd}_j)\right) \label{lab4}\\
   &\simeq R\Gamma_{\petc}(X, \Bb_I[1/t]) \label{lab6}.
  \end{align}
  where the quasi-isomorphism (\ref{lab3}) follows from Proposition \ref{pota} (and Remark \ref{trivia}); in (\ref{lab4}) we used that, for each $j\in \Nn$, the morphism $R\Gamma_{\underline{\dR}}(\Dd_{j+1})\to R\Gamma_{\underline{\dR}}(\Dd_j)$ factors through $R\Gamma_{\underline{\dR}}(\accentset{\circ}{\Dd}_{j+1})$; and (\ref{lab6}) follows from axiom \listref{axiom:1} of Schneider--Stuhler for the condensed de Rham cohomology (Proposition \ref{BdRSS}). Then, the statement follows from the exact sequence 
  $$0\to \Bb_e\to  \Bb_{[1, p]}[1/t]\overset{\varphi-1}{\to}\Bb_{[1, 1]}[1/t]\to 0$$
  (see Remark \ref{remonI}).
 \end{proof}

 We are ready to show the following result.
 
 \begin{prop}\label{BeSS}
  The cohomology theory $$H_{\petc}^\bullet(-_C, \Bb_e):\RigSm_{K, \pet}\to \Mod_{B_e}^{\cond}$$ satisfies the axioms of Schneider--Stuhler (\ref{axioms}).
 \end{prop}
 \begin{proof}
  By Corollary \ref{profinite}, $H_{\petc}^\bullet(-_C, \Bb_e)$ is defined by the complex of sheaves $R\varepsilon_*\varepsilon^*\underline{\Bb_e}$ on the site $\RigSm_{K, \pet}$, where we denote by $\varepsilon:\RigSm_{C, \pet}\to \RigSm_{K, \pet}$ the base change morphism.
  
  Axiom \listref{axiom:1} follows from Corollary \ref{smadonno}. Axioms \listref{axiom:2} and \listref{axiom:3} are satisfied. For axiom \listref{axiom:4}, as in the proof of Proposition \ref{BdRSS}, we first construct the cycle class map $c^{B_e}:\Gg_m[-1]\to R\varepsilon_*\varepsilon^*\Bb_e$, with target the complex of sheaves of abelian groups underlying $R\varepsilon_*\varepsilon^*\underline{\Bb_e}$: we define it on $\RigSm_{K, \pet}$ as the composite 
  $$c^{B_e}:\Gg_m[-1]\longrightarrow R\varepsilon_*\varepsilon^*\Gg_m[-1]\longrightarrow R\varepsilon_*\varepsilon^*\Zz_p(1)\longrightarrow R\varepsilon_*\varepsilon^*\Bb_e(1)\cong R\varepsilon_*\varepsilon^*\Bb_e$$
  where the first arrow is the adjunction morphism,  the middle arrow comes from the boundary map of the Kummer exact sequence (\ref{kummer}), and the trivialization $\Bb_e(1)\cong \Bb_e$ is given by the choice of a compatible system of $p$-th power roots of unity in $\cl O_C$ (Notation \ref{notnot}).
  
  Now, we need to check that $H_{\petc}^i(\Pp^d_C, \Bb_e)=0$ for $i$ odd, or $i>2d$, and that, for all $0\le i\le d$, the map $B_e\to H_{\petc}^{2i}(\Pp^d_C, \Bb_e)$, defined in (\ref{cycleiso}) and induced by $c^{B_e}$, is an isomorphism. We note that, since $\Pp_K^d$ is a proper smooth rigid-analytic variety, by well-known results in (relative) $p$-adic Hodge theory, for $\mathbf{B}\in \{\Bb_{\dR}, \Bb_e\}$, for all $j\ge 0$ we have an isomorphism\footnote{For $\mathbf{B}=\Bb_{\dR}$, the isomorphism (\ref{wk}) follows from \cite[Theorem 8.8, (i)]{Scholze}. For $\mathbf{B}=\Bb_e$ it follows combining the proof of \textit{loc. cit.} with \cite[Proposition 3.4]{LeBras2} and the exact sequence (\ref{fundexacttard}).}
 \begin{equation}\label{wk}
   H^j_{\petc}(\Pp^d_C, \Qq_p)\otimes_{\Qq_p} \underline{\mathbf{B}}(C, \cl O_C)\overset{\sim}{\to} H^j_{\petc}(\Pp^d_C, \mathbf{B}).
 \end{equation}
 By the crystalline comparison theorem for proper smooth schemes over $\cl O_K$, \cite{Falcrys}, combined with the GAGA for étale cohomology, \cite[Theorem 3.2.10]{Huberbook}, we have  $H_{\petc}^i(\Pp^d_C, \Qq_p)=0$ for $i$ odd, or $i>2d$, and $H_{\petc}^{2i}(\Pp^d_C, \Qq_p)=\Qq_p$ for all $0\le i\le d$.
  Then, axiom \listref{axiom:4}  follows from the knowledge of the same axiom for $H_{\petc}^\bullet(-_C, \Bb_{\dR})$, which was shown in Proposition \ref{BdRSS}, and from the compatibility of the cycle class map $c^{B_{\dR}}$ of \textit{loc. cit.} with the cycle class map $c^{B_e}$, under the inclusion $\Bb_e\hookrightarrow \Bb_{\dR}$.
 \end{proof}

 \begin{rem}\label{precisegal}
  We note that, in the proof of Proposition \ref{BeSS}, we identified $\Bb_e(1)$ with $\Bb_e$, in order to make $H_{\petc}^\bullet(-_C, \Bb_e)$ satisfy the axiom \listref{axiom:4} of Schneider--Stuhler (\ref{axioms}). Thus, keeping track of the Galois action on the geometric pro-étale cohomology of $\Pp_K^d$ with coefficients in $\Bb_e$, by Theorem \ref{mainSSS} and Proposition \ref{BeSS}, we have the following statement: let $K$ be a finite extension of $\Qq_p$, denoting by $\Hh_C^d$ the base change to $C$ of the Drinfeld upper half-space $\Hh^d_K$, for all $i\ge 0$ we have a $\underline{G}\times \underline{\mathscr{G}_K}$-equivariant isomorphism in $\Vect_{\Qq_p}^{\cond}$
   $$H_{\petc}^i(\Hh_C^d, \Bb_e)\cong \underline{\Hom}(\Sp_i(\Zz), B_e)(-i).$$
 \end{rem}

 \section{\textbf{An application}}\label{appli}
 \sectionmark{}
 
 We are ready to reprove \cite[Theorem 4.12]{CDN1}.
 
 \begin{theorem}\label{azz}
  Let $K$ be a finite extension of $\Qq_p$. Given an integer $d\ge 1$, let $G=\GL_{d+1}(K)$. Let $\Hh^d_K$ be the Drinfeld upper half-space of dimension $d$ defined over $K$, and let $\Hh_C^d$ be its base change to $C$. For all $i\ge 0$, we have the following commutative diagram of $G\times \mathscr{G}_K$-Fr\'echet spaces over $\Qq_p$, with strictly exact rows
  \begin{center}
   \begin{tikzcd}
   0 \arrow[r] 
     &[-1em] \Omega^{i-1}(\Hh_C^d)/\ker d \arrow[-,double line with arrow={-,-}]{d} \arrow[r] 
     &[-1em] H^i_{\pet}(\Hh_C^d, \Qq_p(i)) \arrow[d] \arrow[r]      
     &[-1em] \Sp_i(\Qq_p)^* \arrow[d] \arrow[r] 
     &[-1em] 0  \\
   0 \arrow[r] 
     &[-1em] \Omega^{i-1}(\Hh_C^d)/\ker d \arrow[r]           
     &[-1em] \Omega^{i}(\Hh_C^d)^{d=0} \arrow[r]  
     &[-1em] \Sp_i(K)^*\widehat \otimes_K C \arrow[r]    
     &[-1em] 0
  \end{tikzcd}
  \end{center}
  where $(-)^*$ denotes the weak topological dual.\footnote{See \S \ref{gsr}.}
 \end{theorem}
 \begin{proof}
  By Corollary \ref{profundexact}, we have a commutative diagram of sheaves on $\Hh^d_{C, \pet}$ with exact rows
  \begin{center}
  \begin{tikzcd}
   0 \arrow[r] &[-1em] \Qq_p\arrow[r]\arrow[d] &[-1em] \Bb_e \arrow[r]\arrow[d] &[-1em] \Bb_{\dR}/\Bb_{\dR}^+ \arrow[r]\arrow[-,double line with arrow={-,-}]{d} &[-1em] 0 \\
   0 \arrow[r] &[-1em] \Bb_{\dR}^+\arrow[r]    &[-1em] \Bb_{\dR} \arrow[r]      &[-1em] \Bb_{\dR}/\Bb_{\dR}^+ \arrow[r]                                          &[-1em] 0
  \end{tikzcd}
  \end{center}
  Let $f: \Hh^d_C\to \Spa(C, \cl O_C)$ be the structure morphism. Then, applying the derived functor $Rf_{\pet*}$ to the diagram above, and taking the long exact sequences in cohomology, we obtain the following $\underline{G}\times \underline{\mathscr{G}_K}$-equivariant commutative diagram in $\Vect_{\Qq_p}^{\cond}$ with exact rows 
  \begin{center}
   \begin{tikzcd}
   \cdots \arrow[r] &[-1em] H_{\petc}^i(\Hh^d_C, \Qq_p) \arrow[r]\arrow[d] &[-1em] H_{\petc}^i(\Hh^d_C, \Bb_{e}) \arrow[r, "\alpha_i"]\arrow[d] &[-1em] H_{\petc}^i(\Hh^d_C, \Bb_{\dR}/\Bb_{\dR}^+)\arrow[r]\arrow[-,double line with arrow={-,-}]{d} &[-1em] \cdots \\
   \cdots \arrow[r] &[-1em]  H_{\petc}^i(\Hh^d_C, \Bb_{\dR}^+)\arrow[r]  &[-1em] H_{\petc}^i(\Hh^d_C, \Bb_{\dR}) \arrow[r, "\beta_i"]  &[-1em] H_{\petc}^i(\Hh^d_C, \Bb_{\dR}/\Bb_{\dR}^+) \arrow[r]  &[-1em] \cdots
  \end{tikzcd}
  \end{center}
  from which we obtain the commutative diagram with exact rows
  \begin{equation}\label{eheh}
  \begin{tikzcd}
   0 \arrow[r] &[-1em] \coker\alpha_{i-1} \arrow[r]\arrow[d] &[-1em]  H_{\petc}^i(\Hh^d_C, \Qq_p) \arrow[r]\arrow[d]  &[-1em] \ker \alpha_i \arrow[r]\arrow[d] &[-1em] 0 \\
   0 \arrow[r] &[-1em] \coker\beta_{i-1} \arrow[r]   &[-1em]  H_{\petc}^i(\Hh^d_C, \Bb_{\dR}^+)\arrow[r]  &[-1em]  \ker \beta_i \arrow[r] &[-1em] 0
  \end{tikzcd}
  \end{equation}
  First, we determine $\ker \beta_i$ and $\coker \beta_{i-1}$. By Theorem \ref{mainSSS} for the cohomology theories $H_{\dRc}^\bullet$ and $H_{\petc}^\bullet(-_C, \Bb_{\dR})$, which applies thanks to Proposition \ref{BdRSS}, we have the following compatible\footnote{The compatibility statement follows from the compatibility between the cycle class maps $c^{\dR}$ and $c^{B_{\dR}}$ proven in  Proposition \ref{BdRSS}.} $\underline{G}\times \underline{\mathscr{G}_K}$-equivariant isomorphisms in $\Vect_{\Qq_p}^{\cond}$
  $$H_{\dRc}^i(\Hh^d_K)\cong \underline{\Hom}(\Sp_i(\Zz), K),\;\;\;\;\;\;\;\; H_{\petc}^i(\Hh^d_C, \Bb_{\dR})\cong\underline{\Hom}(\Sp_i(\Zz), B_{\dR}).$$

  Then, since $\Hh^d_{K}$ is a Stein space, by Remark \ref{affstein} we have the following $\underline{G}\times \underline{\mathscr{G}_K}$-equivariant commutative diagram in $\Vect_{\Qq_p}^{\cond}$ with exact rows 
  \begin{center}
   \begin{tikzcd}
   0 \arrow[r] &[-1em]  H^i_{\dRc}(\Hh^d_K)\solid_K t^{-i+1}B_{\dR}^+ \arrow[r]\arrow[d, "\wr"] &[-1em]  H^i_{\petc}(\Hh^d_C, \Bb_{\dR}^+) \arrow[r]\arrow[d, "\gamma_i"]  &[-1em] \underline{\Omega^i(\Hh^d_K)}^{d=0}\solid_K C(-i) \arrow[r]\arrow[d, "\pi_i"] &[-1em] 0 \\
   0 \arrow[r] &[-1em] \underline{\Hom}(\Sp_i(\Zz),t^{-i+1}B_{\dR}^+) \arrow[r]   &[-1em]  H_{\petc}^i(\Hh^d_C, \Bb_{\dR})\arrow[r]  &[-1em]   \underline{\Hom}(\Sp_i(\Zz),B_{\dR}/t^{-i+1}B_{\dR}^+)\arrow[r] &[-1em] 0
  \end{tikzcd}
  \end{center}
  where the lower row is exact, and the left vertical arrow is an isomorphism, thanks to Remark \ref{soliddual}. Here, the morphism $\pi_i$ is defined as the composite
  $$\pi_i: \underline{\Omega^i(\Hh^d_K)}^{d=0}\solid_K C(-i)\twoheadrightarrow H^i_{\dRc}(\Hh^d_K)\solid_K C(-i)\cong \underline{\Hom}(\Sp_i(\Zz),C(-i))\hookrightarrow \underline{\Hom}(\Sp_i(\Zz),B_{\dR}/t^{-i+1}B_{\dR}^+).$$
  We deduce that $$\ker \beta_i=\im \gamma_i=\underline{\Hom}(\Sp_i(\Zz), t^{-i}B_{\dR}^+),\;\;\;\;\;\;\;\; \coker \beta_{i-1}=\ker \gamma_i=\underline{\Omega^{i-1}(\Hh^d_K)/\ker d}\solid_K C(-i).$$
   Next, we determine $\ker \alpha_i$ and $\coker \alpha_{i-1}$. By Theorem \ref{mainSSS} for the cohomology theory $H_{\petc}^\bullet(-_C, \Bb_e)$, which applies thanks to Proposition \ref{BeSS}, we have the following $\underline{G}\times \underline{\mathscr{G}_K}$-equivariant commutative diagram in $\Vect_{\Qq_p}^{\cond}$ with exact rows 
   \begin{center}
   \begin{tikzcd}
   0 \arrow[r] 
     &[-1em] \underline{\Hom}(\Sp_i(\Zz), B_e(-i)) \arrow[d, "\delta_i"] \arrow[r, "\sim"] 
     &[-1em] H_{\petc}^i(\Hh^d_C, \Bb_e) \arrow[d, "\alpha_i"] \arrow[r]      
     &[-1em] 0 \arrow[d] \arrow[r] 
     &[-1em] 0  \\
   0 \arrow[r] 
     &[-1em] \underline{\Hom}(\Sp_i(\Zz), B_{\dR}/t^{-i}B_{\dR}^+) \arrow[r]           
     &[-1em] H_{\petc}^i(\Hh^d_C, \Bb_{\dR}/\Bb_{\dR}^+) \arrow[r]  
     &[-1em] \underline{\Omega^i(\Hh^d_{K})/\ker d}\solid_K C(-i-1) \arrow[r]    
     &[-1em] 0
  \end{tikzcd}
  \end{center}
   where we recalled  Remark \ref{precisegal}. Now, using the fundamental exact sequence (\ref{fundexact}) over $\Spa(C, \cl O_C)_{\pet}$ together with Corollary \ref{profinite}, we have a short exact sequence
   $$0\to \underline{\Hom}(\Sp_i(\Zz), \Qq_p(-i))\to  \underline{\Hom}(\Sp_i(\Zz),  B_e(-i))\to  \underline{\Hom}(\Sp_i(\Zz), B_{\dR}/t^{-i}B_{\dR}^+)\to 0$$
   where we observed that $R\underline{\Hom}(\Sp_i(\Zz), \Qq_p(-i))$ is concentrated in degree 0 (see again Remark \ref{soliddual}).
   In other words,  $\coker \delta_i=0$, and, by Remark \ref{spdual}, $\ker \delta_i=\underline{\Sp_i(\Qq_p)^*(-i)}$, where the weak topological dual $\Sp_i(\Qq_p)^*$ is a $\Qq_p$-Fréchet space.
   Therefore, by the snake lemma applied to the diagram above, we have 
  $$\ker\alpha_i=\ker\delta_i=\underline{\Sp_i(\Qq_p)^*(-i)}$$
  $$\coker \alpha_{i-1}=\underline{(\Omega^{i-1}(\Hh_C^d)/\ker d)(-i)}.$$
  Hence, putting everything together, from the diagram (\ref{eheh}), twisting by $(i)$ we obtain the following $\underline{G}\times \underline{\mathscr{G}_K}$-equivariant commutative diagram in $\Vect_{\Qq_p}^{\cond}$ with exact rows 
  \begin{center}
   \begin{tikzcd}
   0 \arrow[r] 
     &[-1em] \underline{\Omega^{i-1}(\Hh_C^d)/\ker d} \arrow[-,double line with arrow={-,-}]{d} \arrow[r] 
     &[-1em] H^i_{\petc}(\Hh_C^d, \Qq_p(i)) \arrow[d] \arrow[r]      
     &[-1em] \underline{\Sp_i(\Qq_p)^*} \arrow[d] \arrow[r] 
     &[-1em] 0  \\
   0 \arrow[r] 
     &[-1em] \underline{\Omega^{i-1}(\Hh_C^d)/\ker d} \arrow[-,double line with arrow={-,-}]{d} \arrow[r] 
     &[-1em] H^i_{\petc}(\Hh_C^d, \Bb_{\dR}^+(i)) \arrow[d, "q"] \arrow[r]      
     &[-1em] \underline{\Sp_i(K)^*}\solid_K B_{\dR}^+ \arrow[d] \arrow[r] 
     &[-1em] 0  \\
   0 \arrow[r] 
     &[-1em] \underline{\Omega^{i-1}(\Hh_C^d)/\ker d} \arrow[r]           
     &[-1em] \underline{\Omega^{i}(\Hh_C^d)}^{d=0} \arrow[r]  
     &[-1em] \underline{\Sp_i(K)^*}\solid_K  C \arrow[r]    
     &[-1em] 0
  \end{tikzcd}
  \end{center}
  where $q$ is the quotient by $tB_{\dR}^+$ map. By Proposition \ref{solidvsproj}, we have $\underline{\Sp_i(K)^*}\solid_K C=\underline{\Sp_i(K)^*\widehat \otimes_K C}$. From the diagram above we deduce in particular that $H^i_{\petc}(\Hh_C^d, \Qq_p(i))$ is a $\Qq_p$-Fréchet space, since it is the pullback of the following diagram of $\Qq_p$-Fréchet spaces
  $$\underline{\Sp_i(\Qq_p)^*}\longrightarrow \underline{\Sp_i(K)^*\widehat \otimes_K C}\longleftarrow \underline{\Omega^{i}(\Hh_C^d)}^{d=0}.$$
  Then, the statement follows recalling Remark \ref{frechetsepare} and Lemma \ref{acyclic}.
 \end{proof}
 
 In the proof above, we used the following observation.
 
  \begin{rem}\label{soliddual}
  Let $W$ be a  $K$-Fréchet space in $\Vect_K^{\cond}$. Writing $\Sp_i(K)=\varinjlim_n V_n$ as a countable direct limit of finite-dimensional topological $K$-vector spaces $V_n$ along immersions, as in Remark \ref{spdual}, we deduce that
  $$R\underline{\Hom}(\Sp_i(\Zz), W)=R\varprojlim_n R\underline{\Hom}_K(\underline{V_n}, W)=R\varprojlim_n (\underline{\Hom}_K(\underline{V_n}, K)\solid_K W)=\underline{\Hom}_{K}(\underline{\Sp_i(K)}, K)\solid_K W$$
   concentrated in degree 0, where, in the last step, we used Corollary \ref{commlim}, and the fact that $R^j\varprojlim_n \underline{\Hom}_K(\underline{V_n}, K)=0$ for all $j>0$ (see Lemma \ref{ML}).
 \end{rem}

 \clearpage

 \appendix
 \addtocontents{toc}{\protect\vskip5pt}

 \section{\for{toc}{\textbf{Non-archimedean condensed functional analysis} {\small(after Clausen--Scholze)}}\except{toc}{\textbf{Non-archimedean condensed functional analysis}\\ {\small (after Clausen--Scholze)}}}\label{condfun}
 \sectionmark{}

 The main results of this appendix, which is devoted to the condensed functional analysis over a non-archimedean local field, are due to Clausen--Scholze.\footnote{We have learned most of the results presented here from Dustin Clausen via private communication, or during the ``Masterclass in Condensed Mathematics'', \cite{Cop}. We also thank the participants of the La Tourette's conference in September 2020 for their comments on a preliminary version of this appendix.}
 \subsection{Notation and conventions}\label{ape}
 In this appendix, contrary to \S \ref{conve}, we work in the category of all condensed sets (and not only $\kappa$-condensed sets), \cite[Definition 2.11]{Scholzecond}, which will be denoted by $\CondSet$. \medskip
 
 We denote by $T\mapsto \underline{T}$ the natural functor from the category of topological spaces/groups/rings/etc. to the category of sheaves of sets/groups/rings/etc. on the site of profinite sets with coverings given by finite families of jointly surjective maps; here, $\underline{T}$ is defined via sending a profinite set $S$ to the set/group/ring/etc. of the continuous functions $\mathscr{C}^0(S, T)$.  \medskip
 
 Recall that if $T$ is a T1\footnote{A topological space $T$ is T1 if for all $x\in T$, the set $\{x\}$ is closed.} topological space, then $\underline{T}$ is a condensed set,\cite[Proposition 2.15]{Scholzecond}. The functor $T\mapsto \underline{T}$ from T1 topological spaces to condensed sets admits a left adjoint $X\mapsto X(*)_{\topp}$, sending a condensed set $X$ to the topological space $X(*)_{\topp}$ defined endowing the underlying set $X(*)$ with the quotient topology of $\bigsqcup_{S\to X}S\to X(*)$ where the disjoint union runs over $\kappa$-small extremally disconnected sets, for any $\kappa$ uncountable strong limit cardinal such that $X$ is determined by its values on $\kappa$-small $S$, \cite[Proposition 2.15, Proposition 1.7]{Scholzecond}. \medskip
 
 Next, we explain the relation between the set-theoretic conventions adopted here and \S \ref{conve}.
 
  \begin{rem}[On set-theoretic bounds]\label{setcut}
  Let $\kappa$ be an uncountable strong limit cardinal as in \S \ref{conve}.
  \begin{enumerate}[(i)]
   \item  Recall that the category of $\kappa$-condensed sets embeds fully faithfully into the category of all condensed sets, via left Kan extension along the full embedding of the category of $\kappa$-small extremally disconnected sets into the category of all extremally disconnected sets. By \cite[Proposition 2.9]{Scholzecond}, and the choice of the cardinal $\kappa$, the inclusion of the category of $\kappa$-condensed sets into the category of  condensed sets commutes with all colimits, and with limits along an index category with cardinality less than $\kappa$.
   \item If we temporarily denote by $T\mapsto \underline{T}_{\kappa}$ the functor from the category of topological spaces to the category of $\kappa$-condensed sets, then, the proof of \cite[Proposition 2.15]{Scholzecond} shows that, given a T1 topological space $T$ with cardinality less than $\kappa$, we have $\underline{T}_{\kappa}=\underline{T}$ (in particular, in this case, $\underline{T}_{\kappa}$ does not depend on our choice of the cardinal $\kappa$).
  \end{enumerate}
  \end{rem}

  We denote by $\CondAb$ the category of condensed abelian groups. All condensed rings will be assumed to be condensed associative unital rings. Given a condensed ring $A$, we denote by $\Mod_A^{\cond}$ the category of $A$-modules in $\CondAb$, and we write $\underline{\Hom}_A(-, -)$ for the internal Hom in the category $\Mod_A^{\cond}$ (in the case $A=\Zz$, we often omit the subscript $\Zz$).\footnote{Recall that, for $M, N\in \Mod_{A}^{\cond}$, we define $\underline{\Hom}_A(M, N)(S)=\Hom_{A}(M\otimes_{\Zz}\Zz[S], N)$ for all extremally disconnected sets $S$.} \medskip

  We will denote by $(F, |\cdot|)$ a non-archimedean local field, i.e. a field that is complete with respect to a non-trivial discrete absolute value $|\cdot|$, and having finite residue field. We let $\cl O_F$ denote the ring of integers of $F$, and we fix $\varpi$ a uniformizer of $\cl O_F$. 
  By abuse of notation, $F$ (resp. $\cl O_F$) will also denote the condensed ring $\underline{F}$ (resp. $\underline{\cl O_F}$). We will refer to the objects of $\Mod_F^{\cond}$ as \textit{condensed $F$-vector spaces}.
  
  \subsection{Analytic rings}
  In this section, we recall some generalities on the category of analytic rings.
  \begin{notation}
   We write $\ExtrDisc$ for the category of extremally disconnected sets. 
  \end{notation}

 \begin{df}[{\cite[Definition 7.1, Definition 7.4]{Scholzecond}}]\label{yt} A \textit{pre-analytic ring} $(A, \cl M)$ is a condensed ring $A$, together with a functor
 $$\cl M[-]:\ExtrDisc\to \Mod_A^{\cond}$$
 called \textit{functor of measures}, taking finite disjoint unions to products, and a natural transformation $\underline{S}\to \cl M[S]$ of functors from $\ExtrDisc$ to $\CondSet$.
 
 An \textit{analytic ring} is a pre-analytic ring $(A, \cl M)$ such that for any complex of $A$-modules in $\CondAb$
 $$M_{\bullet}: \cdots \to M_i\to \cdots M_1\to M_0\to 0$$
 such that all $M_i$ are direct sums of objects of the form $\cl M[T]$ for some $T\in \ExtrDisc$, the natural map $$R\underline{\Hom}_A(\cl M[S], M_{\bullet})\to R\underline{\Hom}_A(A[S], M_{\bullet})$$ is an isomorphism for all $S\in \ExtrDisc$.
 
 We say that an analytic ring $(A, \cl M)$ is \textit{commutative} if $A$ is commutative, and \textit{normalized} if the map $A\to \cl M[*]$ is an isomorphism.
 \end{df}

 \begin{df}
  Let $(A, \cl M)$ be an analytic ring.
  \begin{enumerate}[(i)]
   \item We say that a module $M\in \Mod_A^{\cond}$ is \textit{$(A, \cl M)$-complete} if, for all $S\in \ExtrDisc$, the natural map
 $$\Hom_A(\cl M[S], M)\to \Hom_{A}(A[S], M)=M(S)$$
 is an isomorphism.
 \item We say that a complex $M\in D(\Mod_A^{\cond})$ is \textit{$(A, \cl M)$-complete} if, for all $S\in \ExtrDisc$, the natural map 
  $$R\Hom_A(\cl M[S], M)\to R\Hom_{A}(A[S], M)$$
 is an isomorphism.
  \end{enumerate}
 \end{df}

 \begin{prop}[{\cite[Proposition 7.5]{Scholzecond}}]\label{fevre} Let $(A, \cl M)$ be an analytic ring.
 \begin{enumerate}[(i)]
  \item\label{fevre:1} The full subcategory of $(A, \cl M)$-complete modules
  \begin{equation}\label{incux}
   \Mod_{(A, \cl M)}^{\cond}\subset \Mod_A^{\cond}
  \end{equation}
  is an abelian category stable under all limits, colimits, and extensions.
 The objects $\cl M[S]$, for varying $S\in \ExtrDisc$, form a family of compact projective generators of $\Mod_{(A, \cl M)}^{\cond}$.
 The inclusion (\ref{incux}) admits a left adjoint
  \begin{equation}\label{solidification}
   \Mod_A^{\cond}\to \Mod_{(A, \cl M)}^{\cond}: M\mapsto M\otimes_{A}(A, \cl M)
  \end{equation}
  that is the unique functor sending $A[S]\mapsto \cl M[S]$ and preserving all colimits.
  Moreover, if $A$ is commutative, there is a unique symmetric monoidal tensor product $\otimes_{(A, \cl M)}$ on $\Mod_{(A, \cl M)}^{\cond}$ making the functor (\ref{solidification}) symmetric monoidal.
  \item\label{fevre:2} The functor 
  \begin{equation}\label{incux2}
   D(\Mod_{A}^{\cond})\to D(\Mod_{(A, \cl M)}^{\cond})
  \end{equation}
  is fully faithful, and its essential image is stable under all limits and colimits, and given by the $(A, \cl M)$-complete complexes. For any $(A, \cl M)$-complete complex $M\in D(\Mod_{A}^{\cond})$, for all $S\in \ExtrDisc$, the natural map 
  $$R\underline{\Hom}_A(\cl M[S], M)\to R\underline{\Hom}_{A}(A[S], M)$$
 is an isomorphism.
  An object $M\in D(\Mod_A^{\cond})$ lies in $D(\Mod_{(A, \cl M)}^{\cond})$ if and only if $H^i(M)$ lies in $\Mod_{(A, \cl M)}^{\cond}$ for all $i$. The functor (\ref{incux2}) admits left adjoint
  \begin{equation}\label{dsolidification}
   D(\Mod_A^{\cond})\to D(\Mod_{(A, \cl M)}^{\cond}): M\mapsto M\otimes_{A}^{\LL}(A, \cl M)
  \end{equation}
  that is the left derived functor of (\ref{solidification}). Moreover, if $A$ is commutative, there is a unique symmetric monoidal tensor product $\otimes_{(A, \cl M)}^{\LL}$ on $D(\Mod_{(A, \cl M)}^{\cond})$ making the functor (\ref{dsolidification}) symmetric monoidal.
 \end{enumerate}
 \end{prop}

\begin{notation}
 In the following, for an analytic ring $(A, \cl M)$, we often abbreviate $D(A):=D(\Mod_A^{\cond})$ and $D(A, \cl M):=D(\Mod_{(A, \cl M)}^{\cond})$. 
\end{notation}

 Given $M, N\in D(A, \cl M)$, we define the derived internal Hom $R\underline{\Hom}_{(A, \cl M)}(M, N)$ in $D(A, \cl M)$, by setting 
 $$\Hom_{D(A, \cl M)}(\cl M[S], R\underline{\Hom}_{(A, \cl M)}(M, N))=\Hom_{D(A, \cl M)}((M\otimes_{\Zz}\Zz[S])\otimes_{A}^{\LL}(A, \cl M), N)$$
 for varying extremally disconnected sets $S$.
  
 
 \begin{lemma}\label{intsolid}
  Let $(A, \cl M)$ be an analytic ring. For all $M, N\in D(A, \cl M)$, the natural map
  $$R\underline{\Hom}_{(A, \cl M)}(M, N)\to R\underline{\Hom}_{A}(M, N)$$
  is an isomorphism.
 \end{lemma}
  \begin{proof}
   Using Proposition \ref{fevre}, and in particular the fact that the functor $D(A)\to D(A, \cl M)$ is fully faithful and admits left adjoint $M\mapsto M\otimes_{A}^{\LL}(A, \cl M)$, for all extremally disconnected sets $S$, we have the following natural isomorphisms
  \begin{align*}
   \Hom_{D(A)}(A[S], R\underline{\Hom}_{(A, \cl M)}(M, N))&\simeq \Hom_{D(A)}(\cl M[S], R\underline{\Hom}_{(A, \cl M)}(M, N)) \\
   &\simeq  \Hom_{D(A, \cl M)}(\cl M[S], R\underline{\Hom}_{(A, \cl M)}(M, N)) \\
   &= \Hom_{D(A, \cl M)}((M\otimes_{\Zz}\Zz[S])\otimes_{A}^{\LL}(A, \cl M), N) \\
   &\simeq \Hom_{D(A)}(M\otimes_{\Zz}\Zz[S], N) \\
   &=\Hom_{D(A)}(A[S],  R\underline{\Hom}_{A}(M, N))
  \end{align*}
  and the statement follows from Yoneda's lemma.
  \end{proof}

 One of the main results of \cite{Scholzecond} is the following theorem, that gives an important example of analytic ring, on which the results of this appendix will build on.
 
 \begin{theorem}[{\cite[Theorem 5.8]{Scholzecond}}]\label{mainsolid}
  The pre-analytic ring $\Zz_{\solidif}:=(\Zz, \cl M_\Zz)$  with underlying condensed ring $\Zz$, and functor of measures 
  $$\cl M_\Zz:\ExtrDisc\to \CondAb: S\mapsto \Zz[S]^{\solidif}$$
  where $\Zz[S]^{\solidif}:=\varprojlim_{i}\Zz[S_i]$, for $S=\varprojlim_i S_i$ written as a cofiltered limit of finite sets $S_i$, is an analytic ring.
 \end{theorem}

 \begin{notation}\label{solidnot}
  For the analytic ring $\Zz_{\solidif}$, we call $\Solid:=\Mod_{\Zz_{\solidif}}^{\cond}$ the category of \textit{solid abelian groups}, we write
 $$\CondAb\to \Solid: M\mapsto M^{\solidif}$$
 for the functor (\ref{solidification}) of Proposition \ref{fevre}, and call it \textit{solidification}, and we denote by $\solid_{\Zz}$ the unique symmetric monoidal tensor product  making the solidification functor symmetric monoidal.
 \end{notation}

 In the main body of the paper we need to work in the category of $\kappa$-condensed sets (see \S \ref{conve}), therefore, we make the following remark on set-theoretic bounds.
 
 \begin{rem}\label{solidcut}
  Let $\kappa$ be an uncountable strong limit cardinal as in \S \ref{conve}. It is possible to define the notion of \textit{(pre-)analytic ring} for $\kappa$-condensed associative unital rings, replacing extremally disconnected sets with $\kappa$-small extremally disconnected sets in Definition \ref{yt}. 
  Then, with this definition,  the pre-analytic ring $(\Zz, \cl M_\Zz)_\kappa$  given by the $\kappa$-condensed ring $\Zz$, together with the functor sending a $\kappa$-small extremally disconnected set $S$ to the $\kappa$-condensed abelian group $\cl M_\Zz[S]:=\Zz[S]^{\solidif}$ is an analytic ring.  We observe that the category $\Solid_\kappa:=\Mod_{(\Zz, \cl M_\Zz)_\kappa}^{\cond}$ of \textit{$\kappa$-solid abelian groups} embeds fully faithfully in $\Solid$.
 \end{rem}

 Next, we define morphisms between analytic rings.
  
  \begin{df}
 A morphism of analytic rings $(A,\cl M)\to (B, \cl N )$ is a map $A \to B$ of underlying condensed rings such that, for all $S\in \ExtrDisc$, we have  $\cl N[S]\in \Mod_{(A, \cl M)}^{\cond}$.
 \end{df}
 
  We recall that the formation of $\Mod_{(A, \cl M)}^{\cond}$ and $D(\Mod_{(A, \cl M)}^{\cond})$ is functorial in $(A, \cl M)$, along morphisms of analytic rings, \cite[Proposition 7.7]{Scholzecond}. \medskip
  
  Analytic rings over $\Zz_{\solidif}$ (i.e. having a morphism from $\Zz_{\solidif}$) enjoy the following additional remarkable properties.
  
 \begin{prop}[{\cite[Proposition 2.11]{Andr}}]\label{intproj}
  Let $(A, \cl M)$ be an analytic ring over $\Zz_{\solidif}$. For any profinite set $S$, the object $\cl M[S]:=A[S]\otimes_{A}^{\LL}(A, \cl M)$ is concentrated in degree 0, compact, and projective.\footnote{In \textit{loc. cit.} $A$ is assumed to be commutative, however this is not needed here.} If $A$ is commutative, for any profinite sets $S$ and $S'$, we have $\cl M[S]\otimes_{(A, \cl M)}^{\LL}\cl M[S']\cong \cl M[S\times S']$.
 \end{prop}
 
 The next results immediately follow.

  \begin{cor}
   Let $(A, \cl M)$ be a commutative analytic ring over $\Zz_{\solidif}$.  Tensor products of (compact) projective objects in $\Mod_{(A, \cl M)}^{\cond}$ are (compact) projective.
  \end{cor}

  \begin{cor}\label{coinci}
  Let $(A, \cl M)$ be a commutative analytic ring over $\Zz_{\solidif}$. The symmetric monoidal tensor product $\otimes_{(A, \cl M)}^{\LL}$ is the left derived functor of $\otimes_{(A, \cl M)}$.
 \end{cor}
 
  In view of Corollary \ref{coinci}, we can introduce the following definition that we will use later on.
 
  \begin{df}\label{platit}
   Let $(A, \cl M)$ be a commutative analytic ring over $\Zz_{\solidif}$. We say that an object $M\in \Mod_{(A, \cl M)}^{\cond}$ is \textit{flat (for the tensor product $\otimes_{(A, \cl M)}$)} if, for all $N\in \Mod_{(A, \cl M)}^{\cond}$, the derived tensor product $M\otimes_{(A, \cl M)}^{\LL} N$ is concentrated in degree 0.
  \end{df}

  We will use the following simple criterion to produce analytic rings over $\Zz_{\solidif}$.

  \begin{lemma}[Analytic ring structure induced from $\Zz_{\solidif}$]\label{basiclemma}
   Let $A$ be a condensed ring. Let $(A, \Zz)_{\solidif}:=(A, \cl M_{A^{\solidif}})$ be the pre-analytic ring with underlying condensed ring $A$, and functor of measures
 $$\cl M_{A^{\solidif}}[-]:\ExtrDisc \to \Mod_A^{\cond}: S\mapsto A[S]^{\solidif}.$$
 Suppose that $\cl M_\Zz[S]\dsolid_\Zz A^{\solidif}$ is concentrated in degree 0 for all $S\in \ExtrDisc$. Then, $(A, \Zz)_{\solidif}$ is an analytic ring, and $\Mod_{(A, \Zz)_{\solidif}}^{\cond}$ is the category of $A^{\solidif}$-modules in $\Solid$.
  \end{lemma}
  \begin{proof}
  By \cite[Proposition 12.8]{Scholzeanalytic}, in order to prove that $(A, \Zz)_{\solidif}$ is an analytic ring, it suffices to check that, $A[S]^{\LL \solidif}$ is concentrated in degree 0 for all $S\in \ExtrDisc$: for this, we note that $A[S]=\Zz[S]\otimes_{\Zz}^{\LL}A$ (using that $\Zz[S]$ is a flat object in $\CondAb$), and then $A[S]^{\LL \solidif}=\cl M_\Zz[S]\dsolid_\Zz A^{\solidif}$, which is concentrated in degree 0 by assumption. Recalling that $\Mod_{(A, \Zz)_{\solidif}}^{\cond}$ is generated under colimits by $\cl M_{A^{\solidif}}[S]$, for varying $S\in \ExtrDisc$, the last assertion follows from the identification $\cl M_{A^{\solidif}}[S]=\cl M_{\Zz}[S]\solid_{\Zz}A^{\solidif}$.
  \end{proof}

 \subsection{Solid vector spaces}\label{ssk}
 Recall from \ref{ape} that we denote by $(F, |\cdot|)$ a non-archimedean local field. In this section, we study some basic results on the category of \textit{solid $F$-vector spaces}, that we use throughout the paper. \medskip
 
\begin{prop}\label{Kanalytic}
 Let $(F, \Zz)_{\solidif}:=(F, \cl M_F)$ be the pre-analytic ring with underlying condensed ring $F$, and functor of measures
 $$\cl M_F[-]:\ExtrDisc \to \Mod_F^{\cond}: S\mapsto F[S]^{\solidif}.$$
 \begin{enumerate}[(i)]
  \item \label{Kanalytic:1} The pre-analytic ring $(F, \Zz)_{\solidif}$ is analytic. The category $\Mod_{(F, \Zz)_{\solidif}}^{\cond}$ is the category of solid $F$-vector spaces, i.e. $F$-modules in $\Solid$, which will be denoted by $\Mod_F^{\ssolid}$. We write $\solid_F$ for the symmetric monoidal tensor product $\otimes_{(F, \Zz)_{\solidif}}$.
  \item \label{Kanalytic:4} The objects $(\prod_I\cl O_F)[1/\varpi]$, for varying sets $I$, form a family of compact projective generators of $\Mod_F^{\ssolid}$.
  \item \label{Kanalytic:5}  We have
  $(\prod_I\cl O_F)[1/\varpi]\dsolid_F(\prod_J\cl O_F)[1/\varpi]=(\prod_{I\times J}\cl O_F)[1/\varpi]$ for any sets $I$ and $J$.
 \end{enumerate}
 \end{prop}
 \begin{proof}
 First, we note that $F\in \Solid$.
  Since $F=\cl O_F[1/\varpi]$, and $\cl O_F$ is a profinite abelian group (recall that $F$ has finite residue field, by assumption), by Lemma \ref{trivial}, for any set $I$, we have that 
  \begin{equation}\label{79}
  (\prod\nolimits_{I} \Zz)\dsolid_\Zz F=(\prod\nolimits_I\cl O_F)[1/\varpi]
  \end{equation}
 (concentrated in degree 0). Hence, recalling \cite[Corollary 5.5]{Scholzecond}, part \listref{Kanalytic:1} follows from Lemma \ref{basiclemma}. Part \listref{Kanalytic:4} follows by adjunction from (\ref{79}) and the fact that the objects $\prod_I \Zz$, for varying sets $I$, form a family of compact projective generators of $\Solid$, \cite[Theorem 5.8, (i)]{Scholzecond}. The final statement \listref{Kanalytic:5} follows from \cite[Proposition 6.3]{Scholzecond} and (\ref{79}).
 \end{proof}

  \begin{rem}\label{notOk}
  The proof of Proposition \ref{Kanalytic} also shows that the pre-analytic ring $(\cl O_F, \Zz)_{\solidif}$ from Lemma \ref{basiclemma}
  is analytic, the category $\Mod_{\cl O_F}^{\ssolid}:=\Mod_{(\cl O_F, \Zz)_{\solidif}}^{\cond}$ is the category of solid $\cl O_F$-modules, and it is generated by the family of compact projective objects $\prod_I\cl O_F$, for varying sets $I$, satisfying $\prod_I\cl O_F\dsolid_{\cl O_F}\prod_J\cl O_F=\prod_{I\times J}\cl O_F$ for any sets $I$ and $J$. Note that the category $\Mod_F^{\ssolid}$ is the full subcategory of $\Mod_{\cl O_F}^{\ssolid}$ of the objects on which $\varpi$ acts invertibly.
 \end{rem}

 We used the following general lemma.
 
 \begin{lemma}[{\cite{CS}}]\label{trivial}
  Let $M$ be a profinite abelian group. For any set $I$, we have $(\prod_I \Zz)\dsolid_\Zz \underline{M}=\prod_I \underline{M}$, which is concentrated in degree 0.
 \end{lemma}
 \begin{proof}
  Recall that any discrete abelian group has a 2-term free resolution; then, since $M$ is a compact abelian group, by Pontrjagin duality, Lemma \ref{pontt}\listref{pontt:2}, it admits a resolution of the form
  \begin{equation}\label{res}
   0\to M\to \prod_{J_0}\Tt\to \prod_{J_1}\Tt\to 0
  \end{equation}
  where $\Tt=\Rr/\Zz$ is the circle group. Moreover, since (\ref{res}) is a strictly exact sequence of locally compact abelian groups,\footnote{In fact, the right arrow of (\ref{res}) is an open map by the following well-known fact: suppose that $f:G\to H$ is a surjective morphism of Hausdorff topological groups and $G$ is compact, then $f$ is open.} it remains exact after applying the functor $T\mapsto \underline{T}$ (see \cite[page 26]{Scholzecond}). Hence,  applying the left derived solidification functor to (\ref{res}), we get the distinguished triangle
  \begin{equation}\label{Z[1]}
   \underline{M}^{\dsolidif}\to \prod_{J_0}\Zz[1]\to \prod_{J_1}\Zz[1]
  \end{equation}
  where we used that $\Rr^{\dsolidif}=0$ (see \cite[Corollary 6.1, (iii)]{Scholzecond}). Then, applying the functor $(\prod_I \Zz)\dsolid_\Zz-$ to (\ref{Z[1]}), we obtain the statement from \cite[Proposition 6.3]{Scholzecond}.
 \end{proof}
 
 We recall the following classical result (see e.g. \cite{Morris}).
 
 \begin{lemma}[Pontrjagin duality]\label{pontt} Denote by $\LCA$ the category of locally compact abelian groups. Given $A, B\in \LCA$ we endow the group $\Hom(A, B)$ of continuous homomorphisms from $A$ to $B$ with the compact-open topology.
 \begin{enumerate}[(i)] 
  \item\label{pontt:1} Let $\Tt=\Rr/\Zz$ be the circle group. The functor from the category $\LCA$ to the category of topological abelian groups
  \begin{equation*}\label{funpon}
   A\mapsto \Dd(A):=\Hom(A, \Tt)
  \end{equation*}
  induces a contravariant self-equivalence of the category $\LCA$, such that the natural map
  $$A\to \Dd(\Dd(A))$$ is an isomorphism.
  \item\label{pontt:2} The functor $A\mapsto \Dd(A)$ of part \listref{pontt:1} restricts to a contravariant duality between compact abelian groups and discrete abelian groups.
  \item\label{pontt:3} The duality of part \listref{pontt:2} restricts to a contravariant duality between profinite abelian groups and torsion discrete abelian groups.
 \end{enumerate}
 \end{lemma}

 \subsection{Quasi-separated solid vector spaces}\label{qsol} In this section, we study \textit{quasi-separated solid $F$-vector spaces}, i.e. the objects of the category $\Mod_F^{\ssolid}$ whose underlying condensed set is quasi-separated. The goal here is to prove the flatness of such solid $F$-vector spaces, Corollary \ref{solidqs}. \medskip
 
  First, we recall that a condensed set $X$ is quasi-compact (resp. quasi-separated) if and only if there exists a surjection $S\to X$ from a profinite set $S$ (resp. for any  profinite sets $S_1$, $S_2$ mapping to  $X$, the fibre product $S_1\times_X S_2$ is quasi-compact).\footnote{Here, the profinite sets are regarded as condensed sets.}
 
 \begin{rem}\label{ovoo} We will repeatedly use the following elementary observations.
 \begin{enumerate}[(i)]
  \item Any subobject of a quasi-separated condensed set is quasi-separated. 
  \item A filtered colimit of quasi-separated condensed sets along injections is quasi-separated.
 \end{enumerate}
 \end{rem}

 We recall the following result that relates quasi-separated condensed sets to classical topological objects.
 
 \begin{prop}[{\cite[Theorem 2.16]{Scholzecond}, \cite[Proposition 1.2]{Scholzeanalytic}}]\label{qcqs}\
  \begin{enumerate}[(i)]
   \item\label{qcqs:1} The functor $T\mapsto \underline{T}$ from T1 topological spaces to condensed sets induces an equivalence between the category of compact Hausdorff spaces, and the category of quasi-compact quasi-separated condensed sets.
   \item\label{qcqs:2}  The functor $T\mapsto \underline{T}$ induces a fully faithful functor from the category of compactly generated weak Hausdorff spaces to quasi-separated condensed sets. Moreover, it induces an equivalence between the category of ind-(compact Hausdorff spaces) whose transition maps are closed immersions, and the category of quasi-separated condensed sets, via sending an object $(T_i)_{i\in \cl I}$ in the source to the condensed set $\colim_{i\in \cl I}  \underline{T_i}$. 
  \end{enumerate}
 \end{prop}
 
 Next, we need to define the notion of \textit{closedness} for a subobject of a condensed set.
 
 \begin{df}\label{clss}
  Let $X$ be a condensed set. We say that a subobject $\iota: Z \hookrightarrow X$ is \textit{closed} if the morphism $\iota$ is a quasi-compact injection of condensed sets.
 \end{df}

 In the quasi-separated case the definition above is compatible with the topological one, as stated in the following lemma. 
 
 \begin{lemma}[{\cite[Proposition 4.13]{Scholzeanalytic}}]\label{closvs}
  Let $X$ be a quasi-separated condensed set. The functor $Z\mapsto Z(*)_{\topp}$ induces an equivalence between closed subobjects $\iota: Z\hookrightarrow X$ and closed topological subspaces of $X(*)_{\topp}$.\footnote{We refer the reader to \S \ref{ape} for the notation used here.}
 \end{lemma}

 Clausen--Scholze proved the following nice characterization of the quasi-separated solid $F$-vector spaces in terms of the compact projective generators $(\prod_I\cl O_F)[1/\varpi]$ of $\Mod_F^{\ssolid}$.
 
 \begin{prop}[{\cite{CS}}]\label{pontrfun}
  A solid $F$-vector space is quasi-separated if and only if it is the filtered union of its subobjects isomorphic to $(\prod_I\cl O_F)[1/\varpi]$ for some set $I$.
 \end{prop}

 The proof is based on the following lemma.
 
 \begin{lemma}[Smith spaces, {\cite{CS}}]\label{surprise}
  Let $V\in \Mod_F^{\ssolid}$. The following properties are equivalent.
  \begin{enumerate}[(i)]
   \item\label{surprise:1} The solid $F$-vector space $V$ is a Smith space, i.e. it is quasi-separated, and there exists a quasi-compact $\cl O_F$-submodule $M\subset V$ such that $V=M[1/\varpi]$.
   \item\label{surprise:2} $V\cong (\prod_I\cl O_F)[1/\varpi]$, for some set $I$.
  \end{enumerate}
 Moreover, the class of such solid $F$-vector spaces, satisfying equivalently \listref{surprise:1} or \listref{surprise:2}, is stable under extensions, closed subobjects, and quotient by closed subobjects.
 \end{lemma}
\begin{proof}
 We observe that, by Proposition \ref{qcqs}\listref{qcqs:1}, for any set $I$, the $\cl O_F$-module $\prod_I\cl O_F$ is a quasi-compact quasi-separated condensed set, and we note that $$(\prod\nolimits_I\cl O_F)[1/\varpi]=\varinjlim(\prod\nolimits_I\cl O_F\overset{\varpi}{\to}\prod\nolimits_I\cl O_F\overset{\varpi}{\to}\cdots)$$ is a filtered colimit of quasi-separated condensed sets along injections, hence it is quasi-separated. This shows that \listref{surprise:2} implies \listref{surprise:1}. 
 
 Conversely, recalling again Proposition \ref{qcqs}\listref{qcqs:1}, it suffices to show that, given a $\varpi$-torsion-free, compact, Hausdorff, topological $\cl O_F$-module $M_{\ttop}$, we have $M:=\underline{M_{\ttop}}\cong\prod_I\cl O_F$, as condensed $\cl O_F$-modules, for some set $I$. For this, we observe that by Lemma \ref{pontt}\listref{pontt:2} the Pontrjagin dual $\Dd(M_{\ttop})$ is a discrete $\varpi$-power torsion topological $\cl O_F$-module. By Pontrjagin duality, Lemma \ref{pontt}\listref{pontt:3}, we deduce that $M_{\ttop}\cong \Dd(\Dd(M_{\ttop}))$ is a profinite abelian group; in particular, $M=\underline{M_{\ttop}}$ is a solid abelian group. Moreover, denoting by $k=\cl O_F/\varpi$ the residue field of $F$, and using that any (discrete) $k$-vector space is a free $k$-module, we have that $$M_{\ttop}/\varpi\cong\Dd(\Dd(M_{\ttop})/\varpi)\cong \prod_I k$$ for some set $I$ (here, for the first isomorphism, we uses that, for a locally compact abelian group $A$, we have that $\Ext_{\LCA}^1(A, \Tt)=0$, \cite[Proposition 4.14, (vii)]{HS}). Then, using that the object $\prod_I \cl O_F$ is projective in the category of solid $\cl O_F$-modules, we can complete the following diagram 
  $$
   \begin{tikzcd}
                                                            & M \arrow[d, twoheadrightarrow, "\pi"] \\
   \prod_I \cl O_F \arrow[r]\arrow[ur, dashrightarrow, "f"]  & \prod_I k
  \end{tikzcd}
  $$
  where the morphism $\pi$ is induced by the quotient map $M\to M/\varpi$.\footnote{Here, we are implicitly using that $0\to \varpi M_{\topp}\to M_{\topp}\to M_{\topp}/\varpi\to 0$ is a  strictly exact sequence of locally compact abelian groups, so it remains exact after applying the functor $T\mapsto \underline{T}$ (see \cite[page 26]{Scholzecond}).}  Now, since both $\prod_I \cl O_F$ and $M$ are derived $\varpi$-complete (for the latter we use that $M$ is $\varpi$-torsion-free), and the map $f$ is an isomorphism mod $\varpi$, by the derived Nakayama lemma we conclude that $f$ is an isomorphism, as we wanted.
  
  Regarding the last statement: the stability under extensions follows from the property \listref{surprise:2} (using that such extensions are split by the projectivity of the objects $\prod_I\cl O_F$ for varying sets $I$), and the stability under closed subobjects and quotients by closed subobjects follows from the property \listref{surprise:1}.
\end{proof}

 \begin{proof}[Proof of Proposition \ref{pontrfun}]
  Let $V$ be a quasi-separated solid $F$-vector space. By Proposition \ref{Kanalytic}\listref{Kanalytic:4}, we have that $V$ is the filtered colimit of the images of the maps $(\prod_I \cl O_F)[1/\varpi]\to V$. Now, we observe that given a map $f:(\prod_I \cl O_F)[1/\varpi]\to V$, since $V$ is assumed to be quasi-separated, we have that $\ker(f)=f^{-1}(0)$ is a closed subobject of the source. Then, by Lemma \ref{surprise}, the image of $f$ is isomorphic to $(\prod_I \cl O_F)[1/\varpi]/\ker(f)\cong (\prod_J\cl O_F)[1/\varpi]$, for some set $J$.
  
  For the converse, we use again Lemma \ref{surprise}: we note that the category of Smith spaces is filtered, and, since Smith spaces are quasi-separated, by Remark \ref{ovoo} any solid $F$-vector space that is the filtered union of its subobjects that are Smith spaces is quasi-separated.  
 \end{proof}

 We are ready to show the flatness of quasi-separated solid $F$-vector spaces, as promised.
 
 We recall from Definition \ref{platit} that we say that an object $V\in \Mod_F^{\ssolid}$ is \textit{flat} if, for all $W\in \Mod_F^{\ssolid}$,  we have that $V\dsolid_F W$ is concentrated in degree 0.

 \begin{prop}[{\cite{CS}}, cf. {\cite[Theorem 3.14]{CScomplex}}]\label{proj-flat}
  For any set $I$, the solid $F$-vector space $(\prod_I \cl O_F)[1/\varpi]$ is flat.
 \end{prop}
 \begin{proof}
  We need to show that, for all $W\in \Mod_F^{\ssolid}$, we have that $(\prod_I \cl O_F)[1/\varpi]\dsolid_F W$ is concentrated in degree 0. We note that $W$ admits a resolution
  \begin{equation}\label{hio}
   0\to W''\to W'\to W\to 0
  \end{equation}
  with $W'$ and $W''$ quasi-separated solid $F$-vector spaces: by Proposition \ref{Kanalytic}\listref{Kanalytic:4} we can take $W'$ a direct sum of objects of the form $(\prod_J \cl O_F)[1/\varpi]$, and $W''=\ker(W'\to W)$. Tensoring (\ref{hio}) by $(\prod_I \cl O_F)[1/\varpi]\dsolid_F-$, the statement follows from Proposition \ref{pontrfun} and Proposition \ref{Kanalytic}\listref{Kanalytic:5}, using that the latter tensor product commutes with filtered colimits recalling that the abelian category $\Mod_F^{\ssolid}$ satisfies Grothendieck's axiom (AB5).
 \end{proof}

 \begin{cor}[{\cite{CS}}]\label{solidqs}
  Any quasi-separated solid $F$-vector space is flat.
 \end{cor}
 \begin{proof}
  Recalling that the abelian category $\Mod_F^{\ssolid}$ satisfies Grothendieck's axiom (AB5), the statement follows from Proposition \ref{pontrfun} and  Proposition \ref{proj-flat}.
 \end{proof}

 As a consequence of Proposition \ref{proj-flat}, we also  have the following result.
 
  \begin{prop}\label{solidA}
  Let $A$ be a solid $F$-algebra. The pre-analytic ring $(A, \Zz)_{\solidif}:=(A, \cl M_A)$ with underlying condensed ring $A$, and functor of measures
  $$\cl M_A[-]:\ExtrDisc\to \Mod_A^{\cond}: S\mapsto A[S]^{\solidif}$$
  is analytic. The category $\Mod_A^{\ssolid}:=\Mod_{(A, \Zz)_{\solidif}}^{\cond}$ is the category of $A$-modules in $\Solid$. If $A$ is commutative, we write $\solid_A$ for the symmetric monoidal tensor product $\otimes_{(A, \Zz)_{\solidif}}$.
 \end{prop}
  \begin{proof}
   By Lemma \ref{basiclemma}, it suffices to show that $\cl M_{\Zz}[S]\dsolid_{\Zz}A=(\cl M_{\Zz}[S]\dsolid_{\Zz}F)\dsolid_{F}A$ is concentrated in degree 0,  for all extremally disconnected sets $S$. By the proof of Proposition \ref{Kanalytic}, we have that $\cl M_{\Zz}[S]\dsolid_{\Zz}F\simeq (\prod_I \cl O_F)[1/\varpi]$ for some set $I$. Then, the statement follows from Proposition \ref{proj-flat}.
  \end{proof}

 \subsection{Locally convex vector spaces in the condensed world}\label{giuo}
 
 Next, we study some familiar objects of the classical theory of locally convex $F$-vector spaces, from the point of view of condensed mathematics. \medskip
 
  Let us begin with the following crucial observation, that we formulate over a ground field more general that $F$, for future reference.
 
  \begin{lemma}\label{crucial}
  Let $L$ be a complete discretely valued non-archimedean field. Let $V$ be a $L$-Banach space. Then, there exists a profinite set $S$, and an isomorphism of condensed $L$-vector spaces
  $$\underline{V}\cong \underline{\Hom}(\Zz[S], L).$$
 \end{lemma}
 \begin{proof}
  Since the valuation of the field $L$ is discrete, by \cite[Theorem 2.5.4]{Garcia}, the $L$-Banach space $V$ admits a basis. Therefore, by \cite[Proposition 4.6]{PG}, there exists a profinite set $S$ such that $V$ is isomorphic to the $L$-Banach space $\mathscr{C}^0(S, L)$, endowed with the sup-norm. 
  Then, we have $\underline V\cong \underline{\mathscr{C}^0(S, L)}\cong \underline{\Hom}(\Zz[S], L)$, recalling that, for any profinite set $S'$,
  $$\mathscr{C}^0(S', \mathscr{C}^0(S,L))\cong \mathscr{C}^0(S'\times S, L)$$ (see \cite[Theorem 10.5.6]{Garcia}).
 \end{proof}

 \begin{prop}[\cite{CS}]\label{corbello}
  Let $V$ be a complete locally convex $F$-vector space. Then, $\underline{V}$ is a quasi-separated solid $F$-vector space.
 \end{prop}
 \begin{proof}
  Recall that any complete locally convex $F$-vector space is isomorphic to a cofiltered limit of $F$-Banach spaces, in the category of topological $F$-vector spaces (cf. \cite[Chapter II, \S 5.4]{SWolff}). Then, since the subcategory of quasi-separated solid $F$-vector spaces in $\Mod_F^{\cond}$ is stable under limits, we can reduce to the case that $V$ is a $F$-Banach space. In the latter case, $\underline{V}$ is quasi-separated by Proposition \ref{qcqs}\listref{qcqs:2}. 
  
  To show that $\underline{V}$ is a solid $F$-vector space, by Lemma \ref{crucial}, we can suppose $\underline{V}=\underline{\Hom}(\Zz[S], F)$, with $S$ a profinite set. Since a solid $F$-vector space is a $F$-module in $\Solid$, we need to prove that $\underline{\Hom}(\Zz[S], F)$ is a solid abelian group. This follows formally from the fact that $F$ is a solid abelian group: in fact, we have
  $$\underline{\Hom}(\Zz[S], F)\cong \underline{\Hom}(\Zz[S]^{\solidif}, F)\cong \underline{\Hom}_{\Solid}(\Zz[S]^{\solidif}, F)\in \Solid$$
  where for the second isomorphism we used Lemma \ref{intsolid}.
 \end{proof}

  \subsubsection{\normalfont{\textbf{Fréchet spaces}}}
  Now, we focus our attention on the category of \textit{$F$-Fréchet spaces}, i.e. the category of complete metrizable locally convex $F$-vector spaces. 
 
 \begin{rem}\label{frechetsepare}
  Any $F$-Fréchet space is compactly generated, being metrizable,\footnote{More precisely, it is also $\kappa$-compactly generated, for any uncountable cardinal $\kappa$ (see \cite[Remark 1.6]{Scholzecond}).} hence, by \cite[Proposition 1.7, Proposition 2.15]{Scholzecond}, the category of $F$-Fréchet spaces embeds fully faithfully into $\Mod_F^{\cond}$, via the functor $V\mapsto \underline{V}$. 
 \end{rem}

 The next result is often useful when passing from $F$-Fréchet spaces to the associated condensed $F$-vector spaces.
 
 \begin{lemma}\label{acyclic}
   The functor $V\mapsto \underline{V}$ sends acyclic complexes of $F$-Fréchet spaces to acylic complexes of condensed $F$-vector spaces.
  \end{lemma}
  \begin{proof}
   Note that the kernel of a map of $F$-Fréchet spaces is a closed subspace of the source, hence it is a $F$-Fréchet space. Then, since the functor $V\mapsto \underline{V}$ is left exact, by the open mapping theorem for $F$-Fréchet spaces, \cite[Theorem 3.5.10]{Garcia}, it suffices to prove the following general statement: given an open surjective map $f: V\twoheadrightarrow W$ of complete metrizable topological $F$-vector spaces, then the induced map $\underline{V}\to \underline {W}$ is surjective. 
   
   We need to show that, for $S$ an extremally disconnected set, we have $\mathscr{C}^0(S, V)\twoheadrightarrow\mathscr{C}^0(S, W)$, i.e. given $\psi\in \mathscr{C}^0(S, W)$, there exists $\widetilde{\psi}\in \mathscr{C}^0(S, V)$ making the following diagram commute
   $$
   \begin{tikzcd}
                                                            & V \arrow[d, "f"] \\
   S \arrow[r, "\psi"]\arrow[ur, dashrightarrow, "\widetilde{\psi}"]  & W
  \end{tikzcd}
  $$
  By \cite[Lemma 45.1]{Treves}, since $\psi(S)$ is compact in $W$, it is the image $f(H)$ of a compact subset $H$ of $V$. We conclude by recalling that the extremally disconnected sets are the projective objects of the category of compact Hausdorff topological spaces.
  \end{proof}
  
  In particular, by Lemma \ref{acyclic}, given $V$ a $F$-Fréchet space, and $W\subset V$ a \textit{closed} subspace, we have that $\underline{V}/\underline{W}= \underline{V/W}$. However, in general, given $V$ a topological $F$-vector space, and $W\subset V$ a subspace, then, taking the quotient $V/W$ in the category of topological $F$-vector spaces may cause a high loss of topological information; instead, taking the quotient $\underline{V}/\underline{W}$ in the category of condensed $F$-vector spaces keeps track of the topological information of both $V$ and $W$. This phenomenon is illustrated in the following example.
  
  \begin{example}\label{card}
   Let $V$ be a Hausdorff topological $F$-vector space, having a subspace $W\subset V$ such that the topological $F$-vector quotient space $V/W$ is non-Hausdorff. Recall that this is the case if and only if $W\subset V$ is non-closed, if and only if $V/W$ is not T1. Then, we have that
   $\underline{V}/\underline{W}\neq \underline{V/W}$, since $\underline{V/W}$ is not even a condensed set (see \cite[Warning 2.14]{Scholzecond}).
   
   We note that this is not a pathology of the definition of the category of condensed sets, \cite[Definition 2.11]{Scholzecond}. In fact, let us denote by $T\mapsto \underline{T}_{\kappa}$ the functor from the category of topological spaces to the category of $\kappa$-condensed sets; then, we claim that, given an uncountable strong limit cardinal $\kappa>|V|$, we have that
   $\underline{V}_\kappa/\underline{W}_\kappa\neq \underline{V/W}_\kappa$
   as $\kappa$-condensed sets. In order to show this, we need to find a $\kappa$-small extremally disconnected set $S$ such that the map $\mathscr{C}^0(S, V)\to\mathscr{C}^0(S, V/W)$, induced by the quotient morphism $V\to V/W$, is not surjective. Let $\lambda$ be the cardinality of $V$, let $S_0$ be a discrete set with cardinality $\lambda$, and let $S:=\beta S_0$ be the Stone-\v{C}ech compactification of $S_0$. Note that $|S|=2^{2^\lambda}< \kappa$. Since the (continuous) inclusion map $S_0\to S$ has dense image, a continuous function from $S$ to the Hausdorff space $V$ is determined by its values on $S_0$, in particular the set $\mathscr{C}^0(S, V)$ has cardinality at most $\lambda^\lambda=2^\lambda$. On the other hand, denoting by $\overline W$ the closure of $W$ in $V$, we observe that the subspace $\overline W/W\subseteq V/W$ has the indiscrete topology, therefore, every function $S\to \overline W/W$ is continuous; recalling that $|S|=2^{2^\lambda}$, we deduce that the set $\mathscr{C}^0(S, V/W)$ has cardinality at least $2^{2^{2^\lambda}}$. Hence, there cannot be a surjective map from $\mathscr{C}^0(S, V)$ onto $\mathscr{C}^0(S, V/W)$.
  \end{example}

  Relatedly, in general, the functor $V\mapsto \underline{V}$ does not commute with filtered colimits; however, it does in some special cases.
  
  \begin{lemma}[{\cite[Lemma 4.3.7]{BS}}]\label{bspro}
   Let $T_1\hookrightarrow T_2\hookrightarrow \cdots$ be a countable inductive system of Hausdorff topological spaces whose transition maps are closed immersions, and let $T:= \varinjlim_n T_n$. Then, we have
  $$\underline{T}=\varinjlim\nolimits_n{\underline{T_n}}.$$
  \end{lemma}

 \begin{example}\label{LFspace}
  Let $\{V_n\}$ be a countable direct system of Hausdorff topological $F$-vector spaces, whose transitions maps are closed immersions. Let us define the $F$-vector space $V:=\varinjlim_n V_n$, and suppose that, endowing it with the direct limit topology, $V$ is a topological $F$-vector space (see \cite{Hirai} for a discussion on this condition). Then, by Lemma \ref{bspro}  we have that $\underline{V}=\varinjlim_n \underline{V_n}$ as condensed $F$-vector spaces.
 \end{example}

 Let us also record the following condensed version of the classical topological Mittag-Leffler lemma for $F$-Banach spaces.
 
 \begin{lemma}[{\cite{CS}}]\label{condML}
  Let $\{V_n, f_{nm}\}$ be an inverse system of $F$-Banach spaces indexed by $\Nn$. Suppose that for each $n\in \Nn$, there exists $k\ge n$ such that, for every $m\ge k$, $f_{nm}(V_m)$ is dense in $f_{nk}(V_k)$. Then, for all $j>0$, we have $$R^j\varprojlim\nolimits_n \underline{V_n}=0.$$
 \end{lemma}
 \begin{proof}
  By \cite[Lemma 3.18]{Scholze}, it suffices to show that, for any $S$ extremally disconnected set, the inverse system of $F$-Banach spaces $\{\mathscr C^0(S, V_n)\}$ (endowed with the sup-norm), satisfies $$R^1\varprojlim\nolimits_n \mathscr C^0(S, V_n)=0.$$
  Let $S$ be an extremally disconnected set, and let us denote by $\psi_{nm}: \mathscr C^0(S, V_m)\to \mathscr C^0(S, V_n)$ the transition map given by the composition with $f_{nm}$. Recall that, for any $n\in \Nn$, the locally constant functions $\LC(S, V_n)$ are dense in the continuous functions $\mathscr C^0(S, V_n)$, and observe that, for any $m\ge n$, we have $\psi_{nm}(\LC(S, V_m))=\LC(S, f_{nm}(V_m))$. We deduce that, for each $n\in \Nn$, there exists $k\ge n$ such that, for every $m\ge k$, $\psi_{nm}(\mathscr C^0(S, V_m))$ is dense in $\psi_{nk}(\mathscr C^0(S, V_k))$. Then, the statement follows from the topological Mittag-Leffler lemma \cite[Remarques 13.2.4]{Groth}.
 \end{proof}
 
 In view of the results above, we can adopt the following terminology.
 
 \begin{df}\label{term}
  An object of the category $\Mod_F^{\cond}$ is called \textit{$F$-Banach/Fréchet space} if it is of the form $\underline{V}$ for a (topological) $F$-Banach/Fréchet space $V$.
 \end{df}

 \subsection{Nuclear modules}\label{nucnuc}
  In this section, we study a useful characterization of the category of \textit{nuclear $F$-vector spaces}, introduced by Clausen--Scholze, and explain some of its consequences. 
  
  \begin{convnot}
   In this section, all condensed rings will be assumed to be commutative, and all analytic rings will be assumed to be commutative and normalized. Given an analytic ring $(A, \cl M)$, for $M\in \Mod_{(A, \cl M)}^{\cond}$ we write $M^{\vee}:=\underline{\Hom}_{(A, \cl M)}(M, A)$ for its \textit{dual}.
  \end{convnot}

  \begin{df}[cf. {\cite[Definition 13.10]{Scholzeanalytic}}]\label{nuc}
  Let $(A, \cl M)$ be an analytic ring.   An object $N\in\Mod_{(A, \cl M)}^{\cond}$ is \textit{nuclear} if, for all $S\in \ExtrDisc$, the natural map
  \begin{equation}\label{aine}
   (\cl M[S]^{\vee}\otimes_{(A, \cl M)}N)(*)\to N(S)
  \end{equation}
  in $\Ab$ is an isomorphism. We denote by $\Mod_{(A, \cl M)}^{\nuc}$ the full subcategory of $\Mod_{(A, \cl M)}^{\cond}$ whose objects are nuclear.
 \end{df}

  \begin{rem}\label{genecol}
   We note that the subcategory $\Mod_{(A, \cl M)}^{\nuc}\subset \Mod_{(A, \cl M)}^{\cond}$ is stable under all colimits, as both the source and the target of (\ref{aine}) commute with colimits in $N$.
  \end{rem}

 In this section, we will be mainly interested in the categorical properties of nuclear modules for analytic rings  over $(F, \Zz)_{\solidif}$.

 \begin{notation}
  Let $A$ be a (commutative) solid $F$-algebra. 
  For the analytic ring $(A, \Zz)_{\solidif}$ from Proposition \ref{solidA}, we write $\solid_A$ for the symmetric monoidal tensor product $\otimes_{(A, \Zz)_{\solidif}}$ and we denote by $\Mod_{A}^{\nuc}:=\Mod_{(A, \Zz)_{\solidif}}^{\nuc}$ the full subcategory of $\Mod_A^{\ssolid}$ of \textit{nuclear (solid) $A$-modules}. In the case $A=F$, we call the objects of $\Mod_{F}^{\nuc}$ \textit{nuclear (solid) $F$-vector spaces}.
 \end{notation}

  We warn the reader that, in the case $A=F$, the definition of \textit{nuclear $F$-vector space} given above is quite different from the one adopted in the classical non-archimedean functional analysis' literature (see e.g. \cite[Chapter IV, \S 19]{Schneider}). In fact, using the latter definition, a $F$-Banach space is nuclear if and only if it is finite-dimensional over $F$, \cite[Chapter IV, \S 19, p. 120]{Schneider}. Instead, adopting Definition \ref{nuc}, any $F$-Banach space is nuclear, as we will see. \medskip
  
  In general, the notion of \textit{nuclearity} is better behaved in the derived setting. However, in the following case of interest it has extremely nice categorical properties.
  
  \begin{theorem}[\cite{CS}]\label{nuclearbanach}\
  Let $A$ be a nuclear solid $F$-algebra.\footnote{Given a solid $F$-algebra $A$ we say that it is \textit{nuclear} if the underlying solid $F$-module is nuclear.} 
  \begin{enumerate}[(i)]
   \item \label{nuclearbanach:1} The subcategory $\Mod_A^{\nuc}\subset \Mod_A^{\ssolid}$ is an abelian category, stable under the tensor product $\solid_A$, finite limits, countable products, and all colimits.
   \item \label{nuclearbanach:2} The category $\Mod_A^{\nuc}$ is generated under colimits by the objects $\underline{\Hom}_A(A[S], A)$, for varying $S$ profinite sets, which are flat for the tensor product $\solid_A$.
  \end{enumerate}
 \end{theorem}

 As we will see, Theorem \ref{nuclearbanach} applies in particular in the case $A$ is a Banach $F$-algebra (Corollary \ref{ova}). We refer the reader to the next subsection for a proof of Theorem \ref{nuclearbanach}.
  
  \subsubsection{\normalfont{\textbf{Trace-class maps}}}
  
 The proof of Theorem \ref{nuclearbanach} will rely crucially on the following notion.

 \begin{df}[cf. {\cite[Definition 13.11]{Scholzeanalytic}, \cite[Definition 8.1]{CScomplex}}]
  Let $(A, \cl M)$ be an analytic ring.
   A map $f:M\to N$ in $\Mod_{(A, \cl M)}^{\cond}$ is called \textit{trace-class} if it lies in the image of the natural map
   $$(M^{\vee}\otimes_{(A, \cl M)}N)(*)\to \Hom_A(M, N).$$
  \end{df}


 \begin{rem}\label{compp}
  It follows formally from the definition above that if $f:M\to N$ is trace-class, and $g:N\to N'$ and $h:M'\to M$ are arbitrary maps in $\Mod_{(A, \cl M)}^{\cond}$, then $g\circ f\circ h$ is trace-class.
 \end{rem}

 \begin{df}[{\cite{CS}}]\label{tracedef}
  Let $(A, \cl M)$ be an analytic ring. We define the \textit{trace-class functor}
   $$(-)^{\trace}:\Mod_{(A, \cl M)}^{\cond}\to \Mod_A^{\cond}: M \mapsto M^{\trace}$$
 where  $M^{\trace}$ is defined on $S$-valued points, for $S\in \ExtrDisc$, as
  $$M^{\trace}(S)=(\cl M[S]^{\vee}\otimes_{(A, \cl M)}M)(*).$$
 \end{df}

 Next, we study some elementary properties of the trace-class functor and its relation with nuclear modules and trace-class maps.

 \begin{lemma}\label{incc}   Let $(A, \cl M)$ be an analytic ring.
  \begin{enumerate}[(i)]
  \item\label{incc:0} The trace-class functor $(-)^{\trace}:\Mod_{(A, \cl M)}^{\cond}\to \Mod_A^{\cond}$ takes values in $\Mod_{(A, \cl M)}^{\cond}$. 
  \item\label{incc:1} A map $f:P\to M$ in $\Mod_{(A, \cl M)}^{\cond}$ with $P$ compact object is trace-class if and only if it factors through $M^{\trace}$.
  \item\label{incc:2} An object $M\in \Mod_{(A, \cl M)}^{\cond}$ is nuclear if and only if the natural map $M^{\trace}\to M$ is an isomorphism.
  \end{enumerate}
 \end{lemma}
 \begin{proof}
   For part (\ref{incc:0}), we need to check that for any $M\in \Mod_{(A, \cl M)}^{\cond}$ we have $M^{\trace}\in \Mod_{(A, \cl M)}^{\cond}$. Denoting by $\Hom_{(A, \cl M)}(-, -)$ the Hom functor on $\Mod_{(A, \cl M)}^{\cond}$, we define the object $M^{\trace '}\in \Mod_{(A, \cl M)}^{\cond}$, via Yoneda's lemma, as
  $$\Hom_{(A, \cl M)}(\cl M[S], M^{\trace '})=(\cl M[S]^{\vee}\otimes_{(A, \cl M)}M)(*).$$
  for varying $S\in \ExtrDisc$. We claim that we have a natural isomorphism $M^{\trace}\cong M^{\trace '}$: in fact, we have  natural isomorphisms
  $$M^{\trace}(S)=  \Hom_{(A, \cl M)}(\cl M[S], M^{\trace '})\cong \Hom_{A}(\cl M[S], M^{\trace '})\cong \Hom_A(A[S], M^{\trace '})=M^{\trace '}(S)$$
  for all $S\in \ExtrDisc$, using that $M^{\trace '}$ is $(A, \cl M)$-complete.
  
  For part (\ref{incc:1}), the case $P=\cl M[S]$, with $S$ extremally disconnected set, is clear as $f$ is trace-class if and only if it lies in the image of the natural map
  $$\Hom_A(\cl M[S], M^{\trace})\to \Hom_A(\cl M[S], M)$$
  i.e. if and only if it factors through $M^{\trace}$. For a general $P$ compact object of $\Mod_{(A, \cl M)}^{\cond}$, we can reduce to the previous case writing $P$ as a retract of $\cl M[S]$ for some $S$ extremally disconnected set.
  
  Lastly, part (\ref{incc:2}) follows immediately from the definitions.
  
 \end{proof}

  We learned the following result from Andreychev.
 
 \begin{prop}[\cite{CS}, \cite{AndrPhd}]\label{incu}
  Let $(A, \cl M)$ be an analytic ring.  For any $M\in \Mod_{(A, \cl M)}^{\cond}$, the object $M^{\trace}$ (and in particular any nuclear object in $\Mod_{(A, \cl M)}^{\cond}$) can be written as a colimit of objects of the form $\cl M[S]^{\vee}$ for $S$ extremally disconnected sets.
 \end{prop}
 \begin{proof}
 Recalling that $\Mod_{(A, \cl M)}^{\cond}$ is generated under colimits by $\cl M[T]$ for varying $T\in \ExtrDisc$, we have
 $$M^{\trace}=\underset{\cl M[T]\to M^{\trace}}{\colim} \cl M[T]=\underset{A\to \cl M[T]^{\vee}\otimes_{(A, \cl M)}M}{\colim}\cl M[T]$$
 where $T$ runs over $\kappa$-small extremally disconnected sets, for a chosen $\kappa$ uncountable strong limit cardinal such that $M^{\trace}$ is determined by its values on $\kappa$-small $T$. Then, for any $T$ as before, writing $$\cl M[T]^{\vee}=\underset{\cl M[S]\to \cl M[T]^{\vee}}{\colim}\cl M[S]$$
 where $S$ runs over $\kappa'$-small extremally disconnected sets, for $\kappa'$ an uncountable strong limit cardinal (depending on $\kappa$ and $A$) such that each $\cl M[T]^{\vee}$ is determined by its values on $\kappa'$-small $S$, we have
 $$M^{\trace}=\underset{A\to \cl M[S]\otimes_{(A, \cl M)}M}{\colim}\;\;\underset{\cl M[T]\to \cl M[S]^{\vee}}{\colim}  \cl M[T]=\underset{A\to \cl M[S]\otimes_{(A, \cl M)}M}{\colim} \cl M[S]^{\vee}$$
 which implies the statement.
 \end{proof}

 Now, we focus on the proof of Theorem \ref{nuclearbanach} in the case $A=F$. As a first step, we show that the objects $\cl M_F[S]^{\vee}\cong\underline{\Hom}(\Zz[S], F)$, for varying $S$ profinite sets, are nuclear $F$-vector spaces.

 \begin{prop}[{\cite{CS}}]\label{solidbanach}
 For all profinite sets $S$ and $S'$, the natural map
  \begin{equation}\label{SS'}
     \underline{\Hom}(\Zz[S], F)\solid_F \underline{\Hom}(\Zz[S'], F)\to \underline{\Hom}(\Zz[S\times S'], F)
  \end{equation}
  is an isomorphism. In particular, $\underline{\Hom}(\Zz[S], F)$ lies in $\Mod_{F}^{\nuc}$.
 \end{prop}

 Before proving Proposition \ref{solidbanach}, we collect an immediate corollary.
 
 \begin{cor}\label{ova}
  Any $F$-Banach space lies in $\Mod_{F}^{\nuc}$.
 \end{cor}
 \begin{proof}
  Combine Lemma \ref{crucial} and Proposition \ref{solidbanach}.
 \end{proof}

 To prove Proposition \ref{solidbanach}, we will rely crucially on Lemma \ref{dirty:1} and Lemma \ref{dirty:2} below, which combined express the solid $F$-vector space $\underline{\Hom}(\Zz[S], F)$ as a filtered colimit of objects isomorphic to $(\prod_I \cl O_F)[1/\varpi]$, on which we know how to compute the tensor product $\solid_F$ thanks to Proposition \ref{Kanalytic}\listref{Kanalytic:5}.
 
 \begin{notation}
  In the following, for an $\cl O_F$-module $M$ in condensed abelian groups, we denote by $M^{\wedge}_{\varpi}:=\varprojlim_{n\in \Nn} M/\varpi^n$ its \textit{$\varpi$-adic completion}. 
  
  Given a set $I$, we define the condensed $F$-vector space $$V_I:=(\bigoplus\nolimits_I \cl O_F)^{\wedge}_{\varpi}[1/\varpi].$$
 \end{notation}
 
 \begin{lemma}\label{dirty:1}
   Let $S$ be a profinite set, and fix an isomorphism $\Zz[S]^{\solidif} \cong \prod_I \Zz$;\footnote{Here, we use \cite[Corollary 5.5]{Scholzecond}.} then, we have an isomorphism
   \begin{equation}\label{feltr}
    \underline{\Hom}(\Zz[S], F)\cong V_I.
   \end{equation}
  Conversely, for any set $I$ there exist a profinite set $S$ and an isomorphism (\ref{feltr}); in particular, $V_I$ is a $F$-Banach space.
 \end{lemma}
 \begin{proof}
  For the first assertion, we recall that, by Proposition \ref{intproj}, $\Zz[S]^{\solidif}$ is an internally compact object of $\Solid$. Then, we have an isomorphism
  $$\underline{\Hom}(\Zz[S], F)\cong\underline{\Hom}(\Zz[S]^{\solidif}, F)\cong(\varprojlim_{n\in \Nn}\underline{\Hom}(\Zz[S]^{\solidif}, \cl O_F/\varpi^n))[1/\varpi].$$ Next, we observe that, for any discrete condensed abelian group $M$ (such as $\cl O_F/\varpi^n$),  we have a natural isomorphism
  \begin{equation}\label{ihuy}
   \underline{\Hom}(\Zz[S]^{\solidif}, M)\cong \underline{\Hom}(\Zz[S]^{\solidif}, \Zz)\otimes_\Zz M
  \end{equation}
  since discrete condensed abelian groups are generated under colimits by $\Zz$, and both sides of (\ref{ihuy}) commute with colimits in $M$. Hence, recalling that (by the proof of \cite[Proposition 5.7]{Scholzecond}) we have $\underline{\Hom}(\Zz[S]^{\solidif}, \Zz)\cong\bigoplus_I \Zz$, we conclude that $\underline{\Hom}(\Zz[S], F)\cong (\varprojlim_{n\in \Nn}\bigoplus_I \cl O_F/\varpi^n)[1/\varpi]=V_I$, as desired.
  
  For the last part of the statement, by Lemma \ref{crucial} it suffices to check that $V_I$ is a $F$-Banach space: for this, we observe that the topology on $V_I(*)_{\topp}$ is induced by the norm  $\left\Vert\cdot \right\Vert:V_I(*)\to \Rr_{\ge 0}$ defined by $\left\Vert x \right\Vert:=2^{\nu_{\varpi}(x)}$ where $\nu_{\varpi}(x):=\inf\{n\in \Zz\, :\, \varpi^n x\in (\bigoplus\nolimits_I \cl O_F)^{\wedge}_{\varpi}\}$.
 \end{proof}

 \begin{lemma}[{\cite{CS}}]\label{dirty:2}
  Let $I$ be a set. For a function $f:I\to \Rr_{\ge 0}$, we define the condensed $F$-vector space $$V_f:=(\prod_{i\in I}F_{\le f(i)})[1/\varpi]\subset \prod_{i\in I}F$$
  where, for a real number $c\ge 0$, we denote $F_{\le c}:=\{x\in F: |x|\le c\}$. 
  
  Then, regarding $V_I$ as a subobject of $\prod_{I}F$, we have the following identification
  \begin{equation}\label{VI}
   V_I=\varinjlim_{f:I\to \Rr_{\ge 0},\, f\to 0} V_f
  \end{equation}
  where the colimit runs over the functions $f:I\to \Rr_{\ge 0}$ tending to $0$ (i.e. for every $\varepsilon >0$, the set $\{i\in I: |f(i)|\ge \varepsilon\}$ is finite) partially ordered by the relation of pointwise inequality $f\le g$.
 \end{lemma}
 \begin{proof}
  It suffices to prove (\ref{VI}) on $S$-valued points, for all $S\in \ExtrDisc$, in a natural in $S$ way. Using that $S$ is a compact object of the category of condensed sets, we have that
  $$V_I(S)=\Hom(S, (\bigoplus_I \cl O_F)^{\wedge}_{\varpi}[1/\varpi])=(\varprojlim_{n\in \Nn}\bigoplus_I\Hom(S, \cl O_F/\varpi^n))[1/\varpi]$$ and then
  $$V_I(S)=\{(g_i)_{i\in I} \text{ with } g_i\in \mathscr{C}^0(S, F): \forall \varepsilon>0,\; g_i(S)\subseteq F_{\le \varepsilon} \text{ for all but finitely many } g_i\}$$
  which, in turn, identifies with
  $$\varinjlim_{f:I\to \Rr_{\ge 0},\, f\to 0} \left(\prod_{i\in I}\mathscr{C}^0(S, F_{\le f(i)})\right)[1/\varpi]$$
  thus showing the statement.
 \end{proof}


 \begin{rem}\label{omega}
  We observe that any function $f:I\to \Rr_{\ge 0}$ tending to $0$ has countable support. In particular, by Lemma \ref{dirty:2}, we can write $V_I=(\bigoplus_I \cl O_F)^{\wedge}_{\varpi}[1/\varpi]$ as the $\omega_1$-filtered colimit of its subobjects $V_J\subseteq V_I$, over all countable subsets $J\subseteq I$.\footnote{Here, $\omega_1$ denotes the first uncountable ordinal.}
 \end{rem}

 \begin{proof}[Proof of Proposition \ref{solidbanach}]
 By Lemma \ref{dirty:1}, using that $\Zz[S]^{\solidif} \solid_\Zz \Zz[S']^{\solidif}\cong \Zz[S\times S']^{\solidif}$, we can reduce to showing that, for any sets $I$ and $J$, the natural map $$V_I\solid_F V_J\to V_{I\times J}$$ is an isomorphism. To prove this, we use Lemma \ref{dirty:2}. First we note that, for any functions $f:I\to \Rr_{\ge 0}$ and $g:J\to \Rr_{\ge 0}$, by Proposition \ref{Kanalytic}\listref{Kanalytic:5}, we have the identification
  \begin{equation}\label{fg}
  V_f\solid_F V_g=V_{f\times g}\subset \prod_{I\times J}F
  \end{equation}
  where $f\times g:I\times J\to \Rr_{\ge 0}$ is defined by $(f\times g)(i, j):=f(i)\cdot g(j)$ for all $(i, j)\in I\times J$. Then, using that the solid tensor product $\solid_F$ commutes with colimits in both variables, we see that it suffices to show that the filtered union of the $V_{f\times g}$ over all $f:I\to \Rr_{\ge 0}$ and $g:J\to \Rr_{\ge 0}$ tending both to 0, can be identified with the filtered union of the $V_{h}$ over all $h:I\times J\to \Rr_{\ge 0}$ tending to 0. Clearly, if $f$ and $g$ tend both to 0, the $f\times g$ tends to 0. Therefore, it suffices to prove that, given $h:I\times J\to \Rr_{\ge 0}$ tending to 0, there exist functions $f:I\to \Rr_{\ge 0}$ and $g:J\to \Rr_{\ge 0}$ tending to 0 such that $h\le f\times g$ pointwise: for this, we can take $f(i):=\sqrt{\max_{j\in J}h(i, j)}$, and $g(j):=\sqrt{\max_{i\in I}h(i, j)}$.
 \end{proof}

 The following result implies in particular Theorem \ref{nuclearbanach}\listref{nuclearbanach:2}. 
 
 \begin{prop}\label{rew}
   Let $A$ be a nuclear solid $F$-algebra.
   \begin{enumerate}[(i)]
  \item\label{rew:1} We have that $N\in \Mod_A^{\nuc}$ if and only if, for all profinite sets $S$, the natural map
   \begin{equation}\label{afess}
   \underline{\Hom}_A(A[S], A)\solid_A N\to  \underline{\Hom}_A(A[S], N)
   \end{equation}
   is an isomorphism.
   \item\label{rew:2} The category $\Mod_A^{\nuc}$ is generated under colimits by the objects $\underline{\Hom}_A(A[S], A)$, for varying $S$ profinite sets.
   \item\label{rew:3} The objects $\underline{\Hom}_A(A[S], A)$, for varying $S$ profinite sets, are flat for the tensor product $\solid_A$.
   \end{enumerate}
 \end{prop}
 \begin{proof}
  Suppose first $A=F$. Then, part \listref{rew:2} follows from Proposition \ref{incu} and Proposition \ref{solidbanach}. Part \listref{rew:3} follows from Corollary \ref{solidqs}, recalling that the objects $\underline{\Hom}(\Zz[S], F)$ are $F$-Banach spaces and, in particular, are quasi-separated (Proposition \ref{corbello}).  For part \listref{rew:1}, noting that both sides of (\ref{afess}) commute with colimits in $N$ (using that, by Proposition \ref{intproj}, $\Zz[S]^{\solidif}$ is an internally compact object in $\Solid$), by part \listref{rew:2} we can reduce to the case $N=\underline{\Hom}(\Zz[S'], F)$, with $S'$ profinite set, which follows from Proposition \ref{solidbanach}. 
  
  For the general case, we note that, by part \listref{rew:1} of the previous case, since $A$ is a nuclear solid $F$-algebra, we have $\underline{\Hom}_A(A[S], A)\cong \underline{\Hom}(\Zz[S], F)\solid_F A$, for all profinite sets $S$. Then, part \listref{rew:1} follows from part \listref{rew:1} in the case $A=F$ shown above. Part \listref{rew:2} follows using again Proposition \ref{incu} and Proposition \ref{solidbanach}. For part \listref{rew:3}, we observe that for any $M\in \Mod_{A}^{\ssolid}$ 
   $$\underline{\Hom}_A(A[S], A)\dsolid_{A}M\cong \underline{\Hom}(\Zz[S], F)\dsolid_F M$$ 
   is concentrated in degree 0, using that, by the previous case, $\underline{\Hom}(\Zz[S], F)$ is flat for the tensor product $\solid_F$.
 \end{proof}

 We refer the reader to \cite[Proposition 5.35]{Andr} for a statement that generalizes Proposition \ref{rew}\listref{rew:1} in the derived setting.

 \begin{cor}\label{A=F}
  Let $A$ be a nuclear solid $F$-algebra and let $M\in \Mod_A^{\ssolid}$. We have that $M$ lies in $\Mod_A^{\nuc}$ if and only if $M$ (regarded in $\Mod_F^{\ssolid}$) lies in $\Mod_F^{\nuc}$.
 \end{cor}
 \begin{proof}
  It suffices to recall Definition \ref{nuc} and note that, by Proposition \ref{rew}\listref{rew:1} (for $A=F$), for all extremally disconnected sets $S$, we have natural isomorphisms $\cl M_A[S]^\vee\cong\underline{\Hom}_A(A[S], A)\cong\underline{\Hom}_F(F[S], F)\solid_F A\cong \cl M_F[S]^\vee\solid_F A$.
 \end{proof}

 As a last step toward the proof of Theorem \ref{nuclearbanach}, we prove the following result.

 \begin{prop}\label{isino} Let $A$ be a nuclear solid $F$-algebra.
 \begin{enumerate}[(i)]
  \item\label{isino:1} The subcategory $\Mod_A^{\nuc}\subset \Mod_A^{\ssolid}$ is stable under countable products.
  \item\label{isino:2} If $A$ is a Banach $F$-algebra, the subcategory of  $\Mod_A^{\nuc}$ generated under filtered colimits by the objects $\underline{\Hom}_A(A[S], A)$, for varying $S$ profinite sets, is stable under countable products.
 \end{enumerate}
 \end{prop}
 
 The proof will rely on the following key lemma.
 
 \begin{lemma}[{\cite{CS}}]\label{apptr}
  The solid $F$-vector space $\prod_{\Nn}F$ is nuclear and it can be written as a filtered colimit of $F$-Banach spaces.
 \end{lemma}
 \begin{proof}
  First, we note that, for any extremally disconnected set $S$, we have $(\prod_\Nn F)(S)=\prod_\Nn \mathscr{C}^0(S, F)$, and, since $S$ is compact, any continuous map $S\to F$ has image contained in $F_{\le c}$ for some real number $c\ge 0$ (in the notation of Lemma \ref{dirty:2}). We deduce that
 \begin{equation}\label{countK}
   \prod_\Nn F=\varinjlim_{f:\Nn\to \Rr_{\ge 0}} V_f
 \end{equation}
 and $V_f=(\prod_{n\in \Nn}F_{\le f(n)})[1/\varpi]$ denotes condensed $F$-vector space defined in Lemma \ref{dirty:2}. We claim that, applying the trace functor (Definition \ref{tracedef}), the natural map
 \begin{equation}\label{fbs}
  \varinjlim_{f:\Nn\to \Rr_{\ge 0}} V_f^{\trace}\to \varinjlim_{f:\Nn\to \Rr_{\ge 0}} V_f
 \end{equation}
 is an isomorphism, and $V_f^{\trace}$ is a $F$-Banach space.
 
  For this, given $V=(\prod_{\Nn}\cl O_F)[1/\varpi]$, we define the $F$-Banach space $V^{\Ban}:=((\prod_{\Nn}\cl O_F)_{\disc})^{\wedge}_{\varpi}[1/\varpi]$, where for a condensed abelian group $M$ we denote by $M_{\disc}$ the associated discrete condensed abelian group such that $M(*)=M_{\disc}(*)$.\footnote{Note that $V^{\Ban}$ is indeed a $F$-Banach space by the last assertion in Lemma \ref{dirty:1}, observing that $(\prod_{\Nn}\cl O_F)_{\disc}$ is a free (discrete) $\cl O_F$-module.} Then, for all extremally disconnected sets $S$, we have
 $$V^{\trace}(S)=(\cl M_{F}[S]^{\vee}\solid_F V)(*)\cong(\cl M_{F}[S]^{\vee}\solid_F V^{\Ban})(*)$$
 where we used Lemma \ref{dirty:1} and Lemma \ref{dirty:2} to write $\cl M_{F}[S]^{\vee}$ as a filtered colimit of objects isomorphic to $(\prod_I \cl O_F)[1/\varpi]$ for some set $I$, together with Proposition \ref{Kanalytic}\listref{Kanalytic:5}. Since the $F$-Banach space $V^{\Ban}$ is a nuclear $F$-vector space (by Corollary \ref{ova}), we deduce that we have a natural isomorphism $V^{\trace}(S)\cong V^{\Ban}(S)$ for all extremally disconnected sets $S$, and then
 $$V^{\trace}\cong V^{\Ban}.$$
 Coming back to (\ref{fbs}), since $V_f\cong V$, we deduce that, for any $f:\Nn\to \Rr_{\ge 0}$, we have that $V_f^{\trace}$ is a $F$-Banach space, and the natural map $V_{f}^{\trace}\to V_f$ is injective. 
 
 Then, it remains to show that, given $f:\Nn\to \Rr_{\ge 0}$, there exists $g:\Nn\to \Rr_{\ge 0}$ with $g\ge f$ pointwise, such that the natural map $V_f\to V_g$ factors through $V_g^{\trace}$, i.e. it is a trace-class map, Lemma \ref{incc}\listref{incc:1}. For this, assuming, without loss of generality, that $|\cdot|$ on $F$ is normalized, we take $g(n):=|\varpi|^{-n}\cdot f(n)$. Then, using Remark \ref{compp}, we can reduce to checking that, for $V=(\prod_{\Nn}\cl O_F)[1/\varpi]$, the map $H:V\to V$ given by the diagonal matrix whose $(n, n)$-entry is $\varpi^n$ is trace-class: this is clear as $H$ comes from the element in $(V^{\vee}\solid_F V)(*)$ given by $\sum_{n\in \Nn}\varphi_n\otimes e_n$, where $e_n\in V$ is the vector having $1$ in the $n$-th coordinate and 0 otherwise, and $(\varphi_n)_{n\in \Nn}$ is the null-sequence in $V^{\vee}$ with $\varphi_n$ sending a vector in $V$ to its $n$-th coordinate multiplied by $\varpi^n$.
 \end{proof}

 \begin{cor}[\cite{CS}]\label{vand}
  Let $V$ be a $F$-Banach space. The natural map
  $$V\solid_F \prod_{\Nn} F\to \prod_{\Nn} V$$
  is an isomorphism. 
 \end{cor}
 \begin{proof}
 By Lemma \ref{crucial} we have $V\cong \underline{\Hom}(\Zz[S], K)$ for some profinite set $S$. By Lemma \ref{apptr}, $\prod_{\Nn} F$ is a nuclear $F$-vector space, then we have that
 $$V\solid_F \prod_{\Nn}F\cong \underline{\Hom}(\Zz[S], F)\solid_F \prod_{\Nn}F\overset{\sim}{\to}\underline{\Hom}(\Zz[S], \prod_{\Nn}F)\cong \prod_{\Nn}\underline{\Hom}(\Zz[S], F)\cong\prod_{\Nn}V$$   
 where we used Proposition \ref{rew}\listref{rew:1}
 \end{proof}

 \begin{proof}[Proof of Proposition \ref{isino}]
  First, we note that, thanks to Corollary \ref{A=F}, it suffices to prove part \listref{isino:1} in the case $A=F$. Thus, we can assume that $A$ is a Banach $F$-algebra and prove both part \listref{isino:1} and part \listref{isino:2}.
  
  Let $\cl C_A$ denote the full subcategory of $\Mod_{A}^{\ssolid}$ consisting of the objects of the form $\underline{\Hom}_A(A[S], A)$ with $S$ profinite set. By Proposition \ref{rew}\listref{rew:2}, writing an object of $\Mod_A^{\nuc}$ as a filtered colimit of quotients of objects in $\cl C_A$, using that the abelian category $\Mod_A^{\ssolid}$ satisfies Grothendieck's axioms (AB4*) and (AB6), we can reduce to showing that a countable product of objects in $\cl C_A$ can be written as a filtered colimit of objects in $\cl C_A$.
   Applying Proposition \ref{rew}\listref{rew:1} (for $A=F$), we have that $\underline{\Hom}_A(A[S], A)\cong\underline{\Hom}(\Zz[S], F)\solid_F A$, for all profinite sets $S$ (recall that, since $A$ is a Banach $F$-algebra, it is a nuclear $F$-algebra by Corollary \ref{ova}).
  Then, by Lemma \ref{dirty:1}, Lemma \ref{dirty:2}, and Remark \ref{omega}, using again the axiom (AB6) in $\Mod_A^{\ssolid}$, we can further reduce to proving that a countable product of copies of $M_A:=(\bigoplus_{\Nn} \cl O_F)^{\wedge}_{\varpi}[1/\varpi]\solid_F A$ can be written as a filtered colimit objects in $\cl C_A$. By the second statement in Lemma \ref{dirty:1}, there exist a profinite set $S_0$, and an isomorphism $M_A\cong \underline{\Hom}(\Zz[S_0], F)\solid_F A$. Then, denoting $P:=\prod_{\Nn}F$, we have that
  \begin{equation}\label{ghi}
   \prod_{\Nn}M_A\cong \prod_{\Nn} \underline{\Hom}(\Zz[S_0]^{\solidif}, A)\cong\underline{\Hom}(\Zz[S_0]^{\solidif}, \prod_{\Nn}A)\cong  \underline{\Hom}(\Zz[S_0], P\solid_F A)
  \end{equation}
  where for the second isomorphism we used that, by Proposition \ref{intproj}, $\Zz[S_0]^{\solidif}$ is an internally compact object of $\Solid$, and in the last isomorphism we used that $A$ is a Banach $F$-algebra and Corollary \ref{vand}. Now, we claim that the natural map
  \begin{equation}\label{maa}
   \underline{\Hom}(\Zz[S_0], P)\solid_{F}A\to \underline{\Hom}(\Zz[S_0], P\solid_F A)
  \end{equation}
  is an isomorphism. By Lemma \ref{apptr}, we can reduce to showing that (\ref{maa}) is an isomorphism replacing $P$ by $\underline{\Hom}(\Zz[S_1], F)$ with $S_1$ profinite set; in the latter case both the source and the target naturally identify with $\underline{\Hom}(\Zz[S_0\times S_1], A)$ thanks to Proposition \ref{rew}\listref{rew:1} (using that $A\in \Mod_{F}^{\nuc}$). Then, combining (\ref{ghi}) with (\ref{maa}), we have
  $$\prod_{\Nn}M_A\cong (\prod_\Nn \underline{\Hom}(\Zz[S_0], F))\solid_F A\cong (\prod_{\Nn}F)\solid_F \underline{\Hom}(\Zz[S_0], F)\solid_F A$$
  where for the latter isomorphism we used again Corollary \ref{vand}. By Lemma \ref{apptr}, Lemma \ref{crucial}, and Proposition \ref{solidbanach}, we have that $(\prod_{\Nn}F)\solid_F \underline{\Hom}(\Zz[S_0], F)$ can be written as a filtered colimit of objects in $\cl C_F$, and then (using again that $A\in \Mod_{F}^{\nuc}$), $\prod_{\Nn}M_A$ can be written as a filtered colimit of objects in $\cl C_A$, as desired.
 \end{proof}

 We are ready to complete the proof of Theorem \ref{nuclearbanach}.
 
 \begin{proof}[Proof of Theorem \ref{nuclearbanach}]
  Part \listref{nuclearbanach:2} follows from Proposition \ref{rew}\listref{rew:2} and Proposition \ref{rew}\listref{rew:3}. For part \listref{nuclearbanach:1}, the stability of the subcategory $\Mod_A^{\nuc}\subset \Mod_A^{\ssolid}$ under colimits was observed more generally in Remark \ref{genecol}, and the stability under finite limits follows from the flatness of the objects $\cl M_A[S]^\vee=\underline{\Hom}_A(A[S], A)$, for varying $S$ profinite sets; the stability under countable products was shown in Proposition \ref{isino}\listref{isino:1}, and the stability under the tensor product $\solid_A$ follows combining part \listref{nuclearbanach:2} with Proposition \ref{solidbanach}, using that the tensor product $\solid_A$ commutes with colimits in both variables.
 \end{proof}

 \subsubsection{\normalfont{\textbf{Some corollaries}}}
 In this subsection, we focus on the functional analysis over a complete discretely valued extension of $F$, deducing some important consequences from Theorem \ref{nuclearbanach}. 
 
 \begin{convnot}
  We denote by $K$ a complete discretely valued extension of $F$. 
  We extend the terminology adopted in Definition \ref{term} to $K$-Banach/Fréchet spaces.\footnote{We note that a $K$-Fréchet space is in particular a $F$-Fréchet space, therefore, in the following, we can use the results in \S \ref{giuo} on $F$-Fréchet spaces.}
 \end{convnot}
 
 The reason for this restriction is that, for such $K$, thanks to Lemma \ref{crucial}, $K$-Banach spaces in $\Mod_{K}^{\cond}$ correspond to the objects $\Hom(\Zz[S], K)$, for varying profinite sets $S$. In particular, applying Theorem \ref{nuclearbanach}\listref{nuclearbanach:2} with $A=K$ we obtain the following result.

  \begin{cor}\label{ids}
  The subcategory $\Mod_K^{\nuc}\subset \Mod_K^{\ssolid}$  is generated under colimits by the $K$-Banach spaces, which are flat objects for the tensor product $\solid_K$.
  \end{cor}

 \begin{rem}\label{precision}
  Let $\kappa$ be an uncountable strong limit cardinal. We note that $\Mod_K^{\nuc}$ is a full subcategory of the category of $\kappa$-condensed $K$-vector spaces. In fact, by Corollary \ref{ids}, we can reduce to checking that $K$-Banach spaces are  $\kappa$-condensed sets; as any $K$-Banach space is a $F$-Banach space, we can further reduce to the case $K=F$, which follows combing Lemma \ref{dirty:1}, Lemma \ref{dirty:2} and Remark \ref{omega}.
 \end{rem}
 
 Next, we focus on $K$-Fréchet spaces. We start with the following observation.
 
 \begin{rem}\label{frechetcount}
 We recall that a locally convex $K$-vector space $W$ is a $K$-Fréchet space if and only if it is isomorphic, in the category of topological $K$-vector spaces, to the limit of a countable inverse system $\{W_n\}$ of $K$-Banach spaces along transition maps having dense image (cf. \cite[Chapter II, \S 5.4, Corollary 1]{SWolff}). In particular, by Lemma \ref{condML},\footnote{Alternatively, by topological Mittag-Leffler lemma \cite[Remarques 13.2.4]{Groth} and Lemma \ref{acyclic}.} we have a short exact sequence in $\Mod_K^{\cond}$
 \begin{equation}\label{iou}
  0\to \underline{W}\to \prod_n \underline{W_n}\to  \prod_n \underline{W_n}\to 0
 \end{equation}
 where the map $\prod_n\underline{W_n}\to  \prod_n \underline{W_n}$ is the difference of the identity and the transition morphisms.
 \end{rem}
 
 In the following, we refer to the objects of  $\Mod_{K}^{\nuc}$ as \textit{nuclear $K$-vector spaces}.
 
 \begin{prop}\label{frechetnuc}
   Any $K$-Fréchet space is a nuclear $K$-vector space that can be written as a filtered colimit of $K$-Banach spaces.
  \end{prop}
  \begin{proof}
  By Proposition \ref{isino}\listref{isino:2} (for $A=K$), a countable product of $K$-Banach spaces is a nuclear $K$-vector space that can be written as a filtered colimit of $K$-Banach spaces. Then, the statement follows from Remark \ref{frechetcount}, writing a $K$-Fréchet space as a kernel of a map whose source is a filtered colimit of $K$-Banach spaces and the target is quasi-separated (and then using that filtered colimits commute with finite limits in $\Mod_K^{\cond}$).
  \end{proof}
 
 \begin{cor}\label{frescoflat}
  Any $K$-Fréchet space is flat for the tensor product $\solid_K$.
 \end{cor}
 \begin{proof}
  Writing a $K$-Fréchet space as a filtered colimit of $K$-Banach spaces, thanks to Proposition \ref{frechetnuc}, the statement follows from the flatness of $K$-Banach spaces for the tensor product $\solid_K$, Corollary \ref{ids}.
 \end{proof}

  We know that the tensor product $\solid_K$ commutes with colimits in both variables.  Next, we study some special cases in which it commutes with limits.

  \begin{prop}[\cite{CS}]\label{nuccomm}
    Let $\{V_n\}$ be a countable family of nuclear $K$-vector spaces, and let $W$ be a $K$-Fréchet space. Then, the natural map of condensed $K$-vector spaces
    $$ (\prod_n V_n)\solid_K W\to\prod_n (V_n\solid_K W)$$
    is an isomorphism.
  \end{prop}
  \begin{proof}
   By Corollary \ref{ids}, we can write a nuclear $K$-vector space a filtered colimit of quotients of $K$-Banach spaces, then, recalling that the abelian category $\Mod_K^{\ssolid}$ satisfies the axioms (AB4*) and (AB6), we can reduce to the case the condensed $K$-vector spaces $V_n$ are $K$-Banach spaces. Using Remark \ref{frechetcount}, and recalling the flatness of $K$-Fréchet spaces for the tensor product $\solid_K$, Corollary \ref{frescoflat}, we can also suppose that $W$ is a countable product of $K$-Banach spaces. Then, it suffices to show that given two countable families $\{V_n\}$ and $\{W_m\}$ of $K$-Banach spaces, we have
  \begin{equation}\label{redf}
   (\prod_n V_n)\solid_K (\prod_m W_m)=\prod_{n, m} (V_n\solid_K W_m).
  \end{equation}
   By Lemma \ref{crucial}, any $K$-Banach space is isomorphic to $\underline{\Hom}(\Zz[S], K)\cong \underline{\Hom}(\Zz[S], F)\solid_F K$, for some profinite set $S$ (here, for the latter isomorphism, we used Proposition \ref{rew}\listref{rew:1}); then, we can reduce to showing (\ref{redf}) in the case $K=F$. By Lemma \ref{dirty:1} and Lemma \ref{dirty:2}, combined with Remark \ref{omega}, any $F$-Banach space can be written as a filtered colimit of objects isomorphic to $(\bigoplus_{\Nn} \cl O_F)^{\wedge}_{\varpi}[1/\varpi]$; therefore, using the axiom (AB6) in $\Mod_F^{\ssolid}$, we can further reduce to showing that, given a $F$-Banach space $V$, we have
 $$(\prod_{\Nn} V)\solid_F (\prod_{\Nn} V)=\prod_{\Nn\times \Nn} (V\solid_F V).$$
 Recalling that, by Proposition \ref{solidbanach}, the solid tensor product $V\solid_F V$ is also a $F$-Banach space,  using Corollary \ref{vand}, it suffices to show that
 \begin{equation}\label{NN}
   (\prod_\Nn F)\solid_F(\prod_\Nn F)=\prod_{\Nn\times \Nn}F.
 \end{equation}
 Using (\ref{countK}) together with (\ref{fg}), recalling that the solid tensor product $\solid_F$ commutes with colimits in both variables, we deduce that, in order to show (\ref{NN}), it suffices to prove that the filtered union of the $V_{f\times g}$ over all $f:\Nn\to \Rr_{\ge 0}$ and $g:\Nn\to \Rr_{\ge 0}$, can be identified, inside $\prod_{\Nn\times\Nn} K$, with the filtered union of the $V_{h}$ over all $h:\Nn\times \Nn\to \Rr_{\ge 0}$. For this, it suffices to show that, given $h:\Nn\times \Nn\to \Rr_{\ge 0}$, there exist functions $f:\Nn\to \Rr_{\ge 0}$ and $g:\Nn\to \Rr_{\ge 0}$ such that $h\le f\times g$ pointwise: we can take $f(n)=g(n)=\max\{1, \max_{i, j\le n}h(i, j)\}$, thus concluding the proof. 
  \end{proof}

  Let us collect some corollaries of Proposition \ref{nuccomm}.

 \begin{cor}\label{commlim}\ 
 \begin{enumerate}[(i)]
  \item \label{commlim:1} Let $\{V_n\}$ be a countable inverse system of nuclear $K$-vector spaces, and let $W$ be a $K$-Fréchet space. Then, we have
  $$(\varprojlim\nolimits_n V_n)\solid_K W=\varprojlim\nolimits_n (V_n\solid_K W).$$
  \item \label{commlim:2} Let $\{\cl V_n\}$ be a countable inverse system of objects in $D(\Mod_K^{\ssolid})$ such that each $\cl V_n$ is representable by a complex of nuclear $K$-vector spaces. Let $\cl W\in D(\Mod_K^{\ssolid})$ be representable by a bounded above complex of $K$-Fréchet spaces. Then, we have
  $$(R\varprojlim\nolimits_n \cl V_n)\dsolid_K \cl W=R\varprojlim\nolimits_n(\cl V_n\dsolid_K \cl W).$$
 \end{enumerate}
 \end{cor}
 \begin{proof}
  For part \listref{commlim:1}, we consider the exact sequence
  \begin{equation}\label{distt0}
   0\to \varprojlim_n V_n\to \prod_n V_n\to \prod_n V_n
  \end{equation}
   where the last arrow is the difference of the identity and the transition morphisms. By Corollary \ref{frescoflat}, the $K$-Fréchet space $W$ is a flat solid $K$-vector space, hence the exact sequence (\ref{distt0}) remains exact after applying the functor $-\solid_K W$, and the statement follows from Proposition \ref{nuccomm}.
 
  For part \listref{commlim:2}, consider the distinguished triangle in $D(\Mod_K^{\ssolid})$
  \begin{equation}\label{distt}
   R\varprojlim_n \cl V_n\to \prod_n \cl V_n\to \prod_n \cl V_n
  \end{equation}
  where the last arrow is the difference of the identity and the transition morphisms. Applying the functor $-\dsolid_K \cl W$ to (\ref{distt}), we see that it suffices to prove that the natural map
  \begin{equation}\label{dproduct}
  (\prod_n \cl V_n)\dsolid_K \cl W\to\prod_n (\cl V_n\dsolid_K \cl W).
  \end{equation}
   is an isomorphism. Let $\cl W^\bullet$ be a bounded above complex of $K$-Fréchet spaces, representing $\cl W$. By Corollary \ref{frescoflat}, $\cl W^\bullet$ is a bounded above complex of flat solid $K$-vector spaces; in particular, for any complex $\cl K^\bullet \in D(\Mod_K^{\ssolid})$, the derived tensor product $\cl K^\bullet \dsolid_K \cl W^\bullet$ is represented by the total complex $\Tot(\cl K^\bullet \solid_K \cl W^\bullet)$. Moreover, since the category $\Mod_K^{\ssolid}$ satisfies Grothendieck's axiom (AB4*), by \cite[Tag 07KC]{Thestack}, we know that the countable product $\prod_n \cl V_n$ in $D(\Mod_K^{\ssolid})$ is computed by taking the termwise products of any of the complexes representing the $\cl V_n$. Therefore, the assertion that the natural map (\ref{dproduct}) is an isomorphism reduces to the statement of Proposition \ref{nuccomm}.
  
 \end{proof}

 As a consequence, we can show that on $K$-Fréchet spaces the completed projective tensor product $\widehat \otimes_K$ agrees with $\solid_K$.\footnote{Let $V$ and $W$ be two (possibly non-Hausdorff) locally convex $K$-vector spaces. Recall that the \textit{completed projective tensor product} $V\widehat\otimes_K W$ is the Hausdorff completion of the projective tensor product of $V$ and $W$, \cite[Definition 10.3.2]{Garcia}. \label{projj}}
 
 \begin{prop}[\cite{CS}]\label{solidvsproj}
  Let $V$ and $W$ be $K$-Fréchet spaces. Then, the natural map of condensed $K$-vector spaces
  \begin{equation*}\label{projfrech}
   \underline V\solid_K \underline W\to\underline{V\widehat\otimes_K W}
  \end{equation*}
  is an isomorphism.\footnote{Here, the map (\ref{projfrech}) is constructed using that, by Proposition \ref{corbello}, $\underline{V\widehat\otimes_K W}$ is a solid $K$-vector space.}
 \end{prop}
 \begin{proof}
 By Remark \ref{frechetcount}, we can suppose that $V$ (resp. $W$) is the limit of a countable inverse system of $K$-Banach spaces $\{V_n\}$  (resp. $\{W_m\}$) with transition maps having dense image. We have
 $V\widehat\otimes_K W=\varprojlim_{n, m}V_n\widehat\otimes_K W_m.$\footnote{This follows from \cite[Proposition 9, p. 192]{SJ} and (the proof of) \cite[Corollary 1.7, (b)]{Varol}.} Moreover, by Corollary \ref{commlim}\listref{commlim:1}, combined with Corollary \ref{frechetnuc}, we have
 $$\underline{V}\solid_K \underline{W}=\varprojlim_{n, m}\underline{V_n}\solid_K \underline{W_m}.$$
 Therefore, to show the statement, we can reduce to the case that $V$ and $W$ are $K$-Banach spaces. In the latter case, by (the proof of) Lemma \ref{crucial}, there exist profinite sets $S$ and $S'$, such that $V\cong \mathscr C^0(S, K)$ and $W\cong \mathscr C^0(S', K)$ (endowed with the sup-norm), and we have $\underline{V}\cong \underline{\Hom}(\Zz[S], K)$ and $\underline{W}\cong \mathscr \underline{\Hom}(\Zz[S'], K)$. By \cite[Corollary 10.5.7]{Garcia}, we have a natural isomorphism $$C^0(S, K)\widehat\otimes_K C^0(S', K)\cong C^0(S\times S', K).$$
 On the other hand, combining Proposition \ref{solidbanach} with Proposition \ref{rew}\listref{rew:1}, we have a natural isomorphism
 $$\underline{\Hom}(\Zz[S], K)\solid_K \underline{\Hom}(\Zz[S'], K)\cong \underline{\Hom}(\Zz[S\times S'], K)$$
 from which the statement follows.
 \end{proof}

 \begin{rem}\label{coinban}
  Let $A$ and $B$ be two $K$-Banach algebras, and denote by $A\widehat \otimes_K B$ their completed tensor product in the category of $K$-Banach algebras.
  Recall that the topological $K$-vector space underlying the $K$-Banach algebra $A\widehat \otimes_K B$ is the completed projective tensor product of $A$ and $B$, regarded as $K$-Banach spaces (see e.g. \cite[Appendix B]{Bosch}). Therefore, by Proposition \ref{solidvsproj}, the natural map of condensed $K$-algebras
  \begin{equation*}\label{Kalg}
   \underline{A}\solid_K \underline{B}\to \underline{A\widehat\otimes_K B}
  \end{equation*}
  is an isomorphism.
 \end{rem}

  \section{\textbf{Cohomology of condensed groups}}\label{ccg}
  \sectionmark{}

 Our goal in this appendix is to revisit the definition of \textit{condensed group cohomology}, due to Bhatt--Scholze, and explain its relation to Koszul complexes (Definition \ref{Kosz}), in some cases of particular interest to us. \medskip

  We keep the notation and the conventions of Appendix \ref{condfun}, \S \ref{ape}. In particular, in this appendix, contrary to \S \ref{conve}, we work in the category of all condensed sets (and not only $\kappa$-condensed sets).   \medskip 
 
 Let $G$ be a condensed group, and let $M$ be a $G$-module in condensed abelian groups, i.e. a condensed abelian group $M$ endowed with a left $G$-action $G\times M\to M$ in the category of condensed sets. Denoting by $\Zz[G]$ the condensed group ring of $G$ over $\Zz$, we can regard $M$ as a $\Zz[G]$-module in $\CondAb$, and give the following definition (cf. \cite[\S 4.3]{BS}).
 \begin{df}\label{cgc}
 We define the \textit{condensed group cohomology of $G$ with coefficients in $M$} as complex of $D(\CondAb)$
 $$R\Gamma_{\underline{\cond}}(G, M):=R\underline{\Hom}_{\Zz[G]}(\Zz, M)$$
 where $\Zz$ is endowed with the trivial $G$-action.\footnote{Contrary to the category of $\kappa$-condensed abelian groups, the category of all condensed abelian groups has no non-zero injective objects (see \cite{ScholzeMO}). However, we can compute $H^i_{\underline{\cond}}(G, M)=\underline{\Ext}_{\Zz[G]}^i(\Zz, M)$ by taking a projective resolution of the $\Zz[G]$-module $\Zz$.} We denote by $R\Gamma_{\cond}(G, M):=R\Gamma_{\underline{\cond}}(G, M)(*)$ its underlying complex of abelian groups.
 \end{df}

  As we now explain, in many cases, the condensed group cohomology recovers the continuous group cohomology (cf. \cite[Lemma 4.3.9]{BS}). However, Definition \ref{cgc} has some favorable extra features: for example, given a short exact sequence of $G$-modules in condensed abelian groups, one always gets long exact sequences in cohomology. Note also that, contrary to the continuous group cohomology, the condensed group cohomology has a natural ``topological structure'' (more precisely, a condensed structure) by definition, and one can show it admits a Hochschild-Serre spectral sequence.
  
  We learned the following proposition, which is certainly well-known to experts, from Anschütz and Le Bras.

 \begin{prop}\label{cond=cont}
  Let $G$ be a profinite group, and let $M$ be a $\underline{G}$-module in solid abelian groups. 
  \begin{enumerate}[(i)]
   \item\label{cond=cont:1} The complex  $R\Gamma_{\underline{\cond}}(\underline{G}, M)$ is quasi-isomorphic to the complex of solid abelian groups
 $$M\to \underline{\Hom}(\Zz[\underline{G}], M)\to \underline{\Hom}(\Zz[\underline{G}\times \underline{G}], M)\to \cdots$$
  sitting in non-negative cohomological degrees.\footnote{The differentials of the complex are described in the proof.}
  \item\label{cond=cont:2} Suppose that $M=\underline{M_{\topp}}$, with $M_{\topp}$ a (T1) topological $G$-module over $\Zz$. Then, for all $i\ge 0$, we have a natural isomorphism of abelian groups
 $$H^i_{\cond}(\underline{G}, M)\cong H^i_{\cont}(G, M_{\topp}).$$
  \end{enumerate}
 \end{prop}

  \begin{proof}[Proof of Proposition \ref{cond=cont}]
  We consider the bar resolution\footnote{Let $\kappa$ be an uncountable strong limit cardinal such that the condensed set $\underline{G}$ is the left Kan extension of its restriction to $\kappa$-small extremally disconnected sets. Then, the resolution (\ref{bar}) is the left Kan extension of the sheafification (on the site of $\kappa$-small extremally disconnected sets, with coverings given by finite families of jointly surjective maps) of the functor sending a $\kappa$-small extremally disconnected set $S$ to the bar resolution of $\Zz$ over $\Zz[\underline{G}(S)]$, \cite[Chapter I, \S 5]{Brown}.} of $\Zz$ over $\Zz[\underline{G}]$
  \begin{equation}\label{bar}
   \cdots\to \Zz[\underline{G}\times \underline{G}]\to \Zz[\underline{G}]\to \Zz\to 0.
  \end{equation}
  Applying $R\underline{\Hom}_{\Zz[\underline{G}]}(-, M)$, we obtain the spectral sequence
   $$E_1^{i, j}=\underline{\Ext}_{\Zz[\underline{G}]}^j(\Zz[\underline{G}^i], M)\implies \underline{\Ext}_{\Zz[\underline{G}]}^{i+j}(\Zz, M).$$
   Note that $E_1^{i, j}=\underline{\Ext}^j(\Zz[\underline{G}^{i-1}], M)$, for all $j\ge 0$ and $i>0$. Since $M$ is a solid abelian group, by \cite[Corollary 6.1, (iv)]{Scholzecond}, and Lemma \ref{intsolid},   for all profinite sets $S$, and for all $j>0$, we have 
   $$\underline{\Ext}^j(\Zz[S], M)\cong \underline{\Ext}^j(\Zz[S]^{\solidif}, M)\cong \underline{\Ext}_{\Solid}^j(\Zz[S]^{\solidif}, M)=0$$
   where in the last step we used that, by Proposition \ref{intproj}, $\Zz[S]^{\solidif}$ is an internally projective object of the category $\Solid$. In particular, we have that $E_1^{i, j}=0$, for all $j>0$ and $i>0$, which gives part \listref{cond=cont:1}.

   Part \listref{cond=cont:2} follows from part \listref{cond=cont:1}. In fact, by \cite[Proposition 1.7]{Scholzecond}, for every integer $n\ge 0$, we have that
   $\Hom(\underline{G}^n, M)=\mathscr{C}^0(G^n, M_{\topp})$. We conclude observing that the differentials of the complexes computing respectively $H^i_{\cond}(\underline{G}, M)$ and  $H^i_{\cont}(G, M_{\topp})$ agree as well.
  \end{proof}

 The following result generalizes \cite[Lemma 7.3]{BMS1}.
 
 \begin{prop}\label{contkosz}
  Given an integer $n\ge 1$, let $\Gamma:=\Zz_p^n$, and let $\gamma_1, \ldots, \gamma_n$ denote the canonical generators of $\Gamma$. Let $M$ a $\Gamma$-module in the category $\Mod_{\Zz_p}^{\ssolid}$ of $\Zz_p$-modules in $\Solid$. Then, we have a quasi-isomorphism
  $$R\Gamma_{\underline{\cond}}(\Gamma, M)\simeq\Kos_{M}(\gamma_1-1, \ldots, \gamma_n-1).$$
 \end{prop}
 
 First, we prove a general lemma.
 
 \begin{lemma}\label{iwa}
   Let $G$ be a profinite group, and let $R$ be a profinite commutative unital ring. 
   
   We define the Iwasawa algebra of $G$ over $R$ as the condensed $R$-algebra
   $$R\llbracket G \rrbracket:=\varprojlim_{U}R[G/U]$$
   where $U$ runs over all the open normal subgroups of $G$.\footnote{Here, and in the following, to keep notation light, we write $R$ (resp. $G$) to denote $\underline{R}$ (resp. $\underline{G}$).}
   \begin{enumerate}[(i)]
    \item\label{iwa:1} We have $R\llbracket G \rrbracket=R[G]^{\solidif}$, i.e. $R\llbracket G \rrbracket$ is the solidification of the condensed $R$-algebra $R[G]$.
    \item\label{iwa:2} The pre-analytic ring $(R[G], \Zz)_{\solidif}$ from Lemma \ref{basiclemma} is an analytic ring. The category of $(R[G], \Zz)_{\solidif}$-complete modules in $\Mod_{R[G]}^{\cond}$  is the category of $R\llbracket G \rrbracket$-modules in $\Solid$, and it is generated by the compact projective objects $\prod_I R\llbracket G \rrbracket$, for varying sets $I$.
    \item\label{iwa:3} Denoting by $R\underline{\Hom}_{(R[G], \Zz)_{\solidif}}(-, -)$ the derived internal Hom in $D((R[G], \Zz)_{\solidif})$, given $M$ a $R\llbracket G \rrbracket$-module in $\Solid$, we have a natural isomorphism
    \begin{equation}\label{fulliwa}
     R\Gamma_{\underline{\cond}}(G, M)\simeq R\underline{\Hom}_{(R[G], \Zz)_{\solidif}}(\Zz, M).
    \end{equation}
   
   \end{enumerate}
  \end{lemma}
  \begin{proof}
   For part \listref{iwa:1}, we first note that $\Zz[G]^{\solidif}=\varprojlim_{U}\Zz[G/U]$, where $U$ runs over all the open normal subgroups of $G$. Then, the statement follows observing that, by \cite[Corollary 5.5]{Scholzecond} and Lemma \ref{trivial}, for any profinite set $S$, we have
   \begin{equation}
    R[S]^{\dsolidif}=\Zz[S]^{\solidif} \dsolid_\Zz R\simeq(\prod_I \Zz)\dsolid_\Zz R= \prod_I R
   \end{equation}
   concentrated in degree 0, for some set $I$ depending on $S$.
   
   For part \listref{iwa:2}, we note that, by the proof of the previous point, $R\llbracket G \rrbracket$ is profinite: in fact, we have $R\llbracket G \rrbracket\cong \prod_J R$, for some set $J$. Hence, applying again Lemma \ref{trivial}, for any set $I$, we have $(\prod_I \Zz)\dsolid_\Zz R\llbracket G \rrbracket=\prod_IR\llbracket G \rrbracket$ (concentrated in degree 0). Then, the first statement of part \listref{iwa:2} follows from Lemma \ref{basiclemma}, and the second assertion follows from \cite[Theorem 5.8, (i)]{Scholzecond}.
   
   For part \listref{iwa:3}, by the analyticity of $(R[G], \Zz)_{\solidif}$ from part \listref{iwa:2}, by Lemma \ref{intsolid} the natural map
  $$R\underline{\Hom}_{(R[G], \Zz)_{\solidif}}(R, M)\to R\underline{\Hom}_{R[G]}(R, M)$$
  is an isomorphism, and by adjunction the target identifies with $R\underline{\Hom}_{\Zz[G]}(\Zz, M)=R\Gamma_{\underline{\cond}}(G, M)$.
  \end{proof} 
 
 \begin{proof}[Proof of Proposition \ref{contkosz}]  Applying Lemma \ref{iwa} for $G=\Gamma$, and $R=\Zz_p$, we have a natural isomorphism
 $$R\Gamma_{\underline{\cond}}(\Gamma, M)\simeq R\underline{\Hom}_{(\Zz_p[G], \Zz)_{\solidif}}(\Zz_p, M).$$
 Now, we note that the condensed $\Zz_p$-algebra $\Zz_p\llbracket \Gamma\rrbracket$ is given by applying the functor $T\mapsto \underline{T}$ to the topological Iwasawa algebra of $\Gamma$ over $\Zz_p$. Then, we have the following projective resolution of $\Zz_p$ as a $\Zz_p\llbracket \Gamma\rrbracket$-module in $\Solid$
 $$\bigotimes_{i=1}^n (\Zz_p\llbracket \Gamma\rrbracket\overset{\gamma_i-1}{\longrightarrow} \Zz_p\llbracket \Gamma\rrbracket)\overset{\sim}{\longrightarrow} \Zz_p.$$
 Taking $R\underline{\Hom}_{(\Zz_p[G], \Zz)_{\solidif}}(-, M)$ gives the statement.
 \end{proof}


\normalem
 
\nocite{Scholze2}
 
\bibliography{biblio}{}
\bibliographystyle{amsalpha}

\end{document}